\documentclass[a4paper]{article}

\usepackage[a4paper, total={6in, 9in}]{geometry}

\usepackage[utf8]{inputenc}
\usepackage[T1]{fontenc}
\usepackage{
amsmath,
amsthm,
amscd,
amssymb,
}
\usepackage{comment}
\usepackage{tikz}
\usetikzlibrary{positioning,arrows,calc}
\usepackage{xspace}
\usepackage{underscore}

\usepackage{graphicx}

\usepackage{wasysym}
\usepackage{listings}

\usepackage{url}
\usepackage{tikz-cd}

\numberwithin{equation}{section}
\theoremstyle{plain}
\newtheorem{lemma}[equation]{Lemma}
\newtheorem{theorem}[equation]{Theorem}
\newtheorem{proposition}[equation]{Proposition}
\newtheorem{corollary}[equation]{Corollary}
\newtheorem{question}[equation]{Question}

\newtheorem{remark}[equation]{Remark}

\theoremstyle{definition}
\newtheorem{definition}[equation]{Definition}

\newtheorem{example}[equation]{Example}

\newcommand{\R}{\mathbb{R}}
\newcommand{\Q}{\mathbb{Q}}
\newcommand{\Z}{\mathbb{Z}}

\newcommand{\N}{\mathbb{N}}

\newcommand{\B}{\mathcal{L}}

\newcommand{\T}{\mathbb{T}}
\newcommand{\I}{\mathbb{I}}

\newcommand{\id}{\mathrm{id}}

\newcommand{\fr}[1]{\mathrm{Frac}\left(#1\right)} 
\newcommand{\orb}[1]{\mathcal{O}(#1)} 
\newcommand{\orbc}[1]{\overline{\orb{#1}}} 

\newcommand\xqed[1]{%
  \leavevmode\unskip\penalty9999 \hbox{}\nobreak\hfill
  \quad\hbox{#1}}
\newcommand\qee{\xqed{\fullmoon}}

\newcommand{\Aut}{\mathrm{Aut}}

\newcommand{\real}{\mathrm{real}}
\newcommand{\ttfrac}[2]{\ensuremath{\frac{\text{#1}}{\text{#2}}}}
\newcommand{\ccfrac}[2]{\ensuremath{(#1:#2)}}

\newcommand{\Homeo}{\mathrm{Homeo}}

\newcommand{\abs}[1]{{\left\vert #1 \right\vert}}
\newcommand{\cl}[1]{\overline{#1}} 

\newcommand{\bla}{\,\mbox{\textvisiblespace}\,}

\newcommand{\bin}{\mathrm{bin}} 

\title{Sofically presented dynamical systems}

\author{Johan Kopra \and Ville Salo}

\begin{document}
\maketitle

\begin{abstract}
Systems obtained by quotienting a subshift of finite type (SFT) by another SFT are called finitely presented in the literature. Analogously, if a sofic shift is quotiented by a sofic equivalence relation, we call the resulting system sofically presented. Generalizing an observation of Fried, for all discrete countable monoids $M$, we show that \ttfrac{$M$-subshift}{SFT} systems are precisely the expansive dynamical $M$-systems, where \ttfrac{$S_1$}{$S_2$} denotes the class of systems obtained by quotienting subshifts in $S_1$ by (relative) subshifts in $S_2$. We show that for all finitely generated infinite monoids $M$,
\[ \text{$M$-SFT} \subsetneq \text{$M$-sofic} \subsetneq \ttfrac{$M$-SFT}{SFT} = \ttfrac{$M$-sofic}{SFT} \subsetneq \ttfrac{$M$-SFT}{sofic} = \ttfrac{$M$-sofic}{sofic}, \]
and that Ma\~n\'e's theorem about the dimension of expansive systems characterizes the virtually cyclic groups. 

In the case of one-dimensional actions, Ma\~n\'e's theorem generalizes to sofically presented systems, which also have finite topological dimension. The basis of this is the construction of an explicit metric for a sofically presented system. We show that any finite connected simplicial complex is a connected component of a finitely presented system, and prove that conjugacy of one-dimensional sofically presented dynamical systems is undecidable. A key idea is the introduction of so-called automatic spaces. We also briefly study the automorphism groups and periodic points of these systems.

We also perform two case studies. First, in the context of $\beta$-shifts, we define the $\beta$-kernel -- the least subshift relation that identifies $1$ with its orbit. We give a classification of the $\beta$-shift/$\beta$-kernel pair as a function of $\beta$. Second, we revisit the classical study of toral automorphisms in our framework, and in particular for the toral automorphism $\left(\begin{smallmatrix} 1 & 1 \\ 1 & 0 \end{smallmatrix}\right)$, we explicitly compute the kernel of the standard presentation.
\end{abstract}

\section{Introduction}

\subsection{Mission statement and motivation}

Symbolic dynamics originated as a tool for studying other dynamical systems, and this is still an important aspect of the theory. For example, it is classical that the toral automorphism defined by $\left(\begin{smallmatrix} 1 & 1 \\ 1 & 0 \end{smallmatrix}\right)$ is a factor of the golden mean shift, and the difference -- the kernel of the factor map -- is small enough that one can study the former using the latter. One can also study symbolic dynamics as an independent subject, and the golden mean shift is an important example on its own. It is easy to show that the kernel of the factor map is itself a subshift of the Cartesian square of the golden mean shift, and thus, once the kernel is also given a purely symbolic form, we have, in some sense, presented a classical system (say, a homeomorphism acting on a manifold) internally in the category of SFTs, which is an interesting connection between a priori very different categories. One of our starting points for this paper was the observation that no one has apparently clearly stated what the kernel is in this classical example, so the complete internal presentation has not actually been given.

Finitely presented systems -- SFT quotients of SFTs -- were introduced in \cite{Fr87}, and the toral automorphism described above is the prime example of such a system. In many respects, such systems resemble SFTs and sofics, and one can work with explicit combinatorial representations to prove properties of the systems. Generalizing the case of a sofic shift, the $\zeta$-function is shown to be rational in these systems in \cite{Fr87}. In the context of thermodynamical formalism it has been shown that equilibrium states coincide with invariant Gibbs states in topologically mixing finitely presented systems \cite{Ba91}. Unlike SFTs and sofics, finitely presented systems are not restricted to be zero-dimensional, making them more attractive for the more classically minded if studied for their own sake rather than as a tool. Generalizing the example, it is known that all hyperbolic automorphisms of the two-dimensional torus are finitely presented.

In this paper, we take the idea further in a direction that is natural from the point of view of the general theory of symbolic dynamics \cite{LiMa95}, and study the more general class of sofically presented dynamical systems, defined by taking a sofic quotient of a sofic shift. These systems turn out to have some similarities to the class of finitely presented systems: their underlying spaces are finite-dimensional, and they can be presented by combinatorial objects -- indeed they are also ``finitely presented'' in the broad sense of the term.

Returning to the representation of classical systems by symbolic means, in general finding combinatorial representations can be useful for computational purposes (whether the point of view is algebraic or recursion-theoretic); in particular this allows the discussion of decidability questions. Sofically presented systems over a one dimensional acting monoid are presented by automata theoretic objects, and the decidability of the monadic second-order logic of the integers under the successor relation implies the computability of many basic manipulations.  For example in our second case study, which is about the toral automorphism described above, we find this useful: we define the kernel as a sofic shift before computing it explicitly (and thereby seeing that it is an SFT, which we would have known in advance by~\cite{Fr87}), and sofically presented systems seem like the correct conceptual intermediate object in this calculation.

Perhaps the most interesting question about combinatorial representations of dynamical systems is how the underlying space and the dynamical properties of the transformation are reflected in the combinatorial object, and whether this can help understand the concepts themselves. As an optimistic analog, we can see that a torus has a hole in the middle just by looking at a sufficiently good drawing of it, but mathematically this is usually expressed by algebraic invariants. Somewhat similarly, perhaps some properties of a dynamical system are easier to see or prove from its automata-theoretic presentations (perhaps with the help of a computer). For example, expansivity of a sofically presented system can be easily solved by checking whether the kernel is an SFT, and this is easy to implement in practice on the computer (and we have done so).

We are optimistic that for the class of sofically presented dynamical systems it will be possible to realize this goal to a reasonable extent, at least in smaller dimensions. We are also optimistic that sofically presented systems can cover some classes of classical examples which are not finitely-presented, or that at least the study of such representability is fruitful. Even if this mission ``fails'', we believe that many interesting problems, examples and counterexamples will be found in pursuing it, and we report our initial findings in this document.

Sofic presentability and finite presentability also suggest the more general notion of an ``\ttfrac{$S_1$}{$S_2$}'' system, namely one obtained by quotienting subshifts in $S_1$ by a (relative) subshift in $S_2$, and thus sofically presented systems are a natural link between finitely-presented systems and the class of \ttfrac{subshift}{subshift} systems, which have also been studied before (since they are just the factors of subshifts), see for example \cite{BoFiFi02}. Such formal quotients provide an interesting way to categorize dynamical systems, and for example we do not know whether all \ttfrac{subshift}{subshift} systems are \ttfrac{subshift}{sofic} (but they are not all \ttfrac{sofic}{subshift}). Examples of \ttfrac{subshift}{sofic} system are irrational rotations, and here the difficulty of seeing the shape of the underlying space by purely combinatorial means is well known (as noted in the remark following Exercise~6.2.13 in~\cite{Fogg02}).

\subsection{Definitions of the classes}

A \emph{topological dynamical system} (TDS) is a continuous action of a discrete monoid $M$ on a compact metrizable space $X$. We may also say that $(X,M)$ is a TDS or that $X$ is an $M$-system. Often $M=\N$ (resp. $M=\Z$), in which case there is a single continuous map (resp. homeomorphism) $T:X\to X$ such that each $m\in M$ acts on $X$ by the map $T^m$. Then we also say that $(X,T)$ is a TDS. If $T : M \to M$, the underlying monoid (either $\N$ or $\Z$) of $(X,T)$ will be clear from the context or explicitly stated. For arbitrary sets $A,B$ we denote by $B^A$ the set of all functions with domain $A$ and range $B$. Particularly when $\Sigma$ is a finite alphabet and $M$ is a monoid, for $x\in \Sigma^M$ we denote by $x_n$ the value of $x$ at $n\in\N$. A \emph{subshift} over a monoid $M$ is a subset of $\Sigma^M$ which is topologically closed and closed under the translations $mx_n = x_{nm}$ (which are an $M$-action on $\Sigma^M$).

We liberally use the identification $(A \times B)^M \cong A^M \times B^M$, in particular subshifts over a pair alphabet can be seen as defining a subset of the Cartesian product of two subshifts, with the diagonal action. Through this identification, if $Y \subset \Sigma^M$ is a subshift and $Z \subset (\Sigma^2)^M$ is a subshift, then we can see $Z \cap Y^2$ as defining a binary relation on $Y$. The notation $Y/Z$ is defined when $Z \cap Y^2$ is a closed equivalence relation on $Y$, and in that case denotes the dynamical system whose points are the topological quotient $Y/(Z \cap Y^2)$, i.e.\ equivalence classes $[y]_{\equiv}$ of points $y \in Y$ under the relation $y \equiv y' \iff (y, y') \in Z$, with the action $m[y]_{\equiv} = [my]_{\equiv}$ (see Section~\ref{sec:Definitions}).

\begin{definition}
Let $S_1$ and $S_2$ be classes of subshifts. We write \ttfrac{$S_1$}{$S_2$}, $S_1/S_2$ or $S_1$-per-$S_2$ for the class of dynamical systems $X$ such that $X \cong Y/Z$ for some subshifts $S_1 \ni Y \subset \Sigma^M$, $S_2 \ni Z \subset (\Sigma^2)^M$ such that the notation $Y/Z$ is defined, i.e.\ $Z \cap Y^2$ is an equivalence relation on $Y$.
\end{definition}

Our emphasis is on the classes \emph{SFT} (the class of subshifts of finite type), \emph{sofic} (the class of sofic shifts), and \emph{subshift} (the class of all subshifts).

In general, if $X$ and $X'$ are subshifts and $X'\in S$ for some class of subshifts $S$, we say that $X' \cap X$ is a relative $S$-subshift (in $X$). In the previous definition, we are quotienting a subshift $Y \in S_1$ by a relative $S_2$-subshift of $Y^2$, in particular $\ttfrac{subshift}{SFT}$ more precisely consists of quotients of subshifts by \emph{relative} SFTs. It is true on all monoids that SFTs are precisely the relative SFTs of full shifts, and a relative SFT of an SFT is an SFT.

\begin{definition}
Dynamical systems in \ttfrac{SFT}{SFT} are called \emph{finitely presented} \cite{Fr87}. We call dynamical systems in \ttfrac{sofic}{sofic} \emph{sofically presented}.
\end{definition}

Often we restrict to $M = \N$ or $M = \Z$. This will be clear from context, or is specified with $S_1$, e.g. \ttfrac{$\Z$-SFT}{SFT}.

Note that though we call these ``finitely presented systems'', they are more correctly ``finitely presentable systems'', and for example in our terminology, for \emph{every} subshift $X$, \ttfrac{$X$}{$X^2$} is \ttfrac{SFT}{SFT}, i.e. finitely-presented, since it is the trivial system. On occasion we do discuss presentations as first-class objects. We write formal quotients of subshifts as $\ccfrac{Y}{Z}$, and say such a quotient is \ttfrac{$S_1$}{$S_2$} if and only if $Y \subset \Sigma^M$ and $Z \subset (\Sigma^2)^M$ for some $\Sigma$, $Z \cap Y^2$ is an equivalence relation on $Y$, and $Y \in S_1$ and $Z \in S_2$. Of course for a subshift $X$, \ccfrac{X}{X^2} is \ttfrac{SFT}{SFT} if and only if $X$ is an SFT.

\subsection{The results}

The following slightly extends a result from \cite{Fr87}, and shows that ``\ttfrac{subshift}{SFT}'' is a synonym for ``expansive'' for all monoid actions (on compact metrizable spaces), that expansive systems have no \ttfrac{subshift}{subshift} representations other than \ttfrac{subshift}{SFT} ones, and that (at least) in the group case, the cover can always be picked to be almost $1$-$1$.

\begin{theorem}
\label{thm:ExpansivityCharacterization}
Let $M$ be any discrete countable monoid and $X$ a dynamical $M$-system. The following are equivalent:
\begin{itemize}
\item $X$ is expansive,
\item $X$ is \ttfrac{subshift}{SFT}, and
\item if $X \cong Y/Z$, then $\ccfrac{Y}{Z}$ is \ttfrac{subshift}{SFT}.
\end{itemize}
If $M$ a group, then all are equivalent to the following:
\begin{itemize}
\item $X$ is conjugate to $Y/Z$ for a subshift $Y$ and a relative SFT $Z$ in $Y$ such that the set $\{ y \in Y \;|\; \exists y' \in Y \setminus \{y\}: (y, y') \in Z \}$ is meager in $Y$.
\end{itemize}
\end{theorem}

This is proved in Section~\ref{sec:DivisionBySFT}. The techniques used are existing ones: in the group case, we use the construction of the cover from \cite{Ad98,LiMa95}, which we call the \emph{small orbit subshift}, while in the proper monoid case we use what we call the \emph{large orbit subshift}, which is used in \cite{Re68}. We include a technical discussion and comparison of these covers, and also study a third cover in-between which we call the \emph{medium orbit subshift}.

The main new observation is that the need for multiple constructions seems to be essential, even for simple choices of $M$: all these covering subshifts can differ even when computed for a $\Z$-subshift (when using a non-standard cover), and the small orbit subshift, which gives the good covers in the group case, may fail to give a subshift cover at all for a surjective $\N$-subshift.

The following is proved in Theorem~\ref{thm:ProperSofic}.

\begin{theorem}
\label{thm:RelationsBetweenClasses}
For $M$ any finitely generated infinite monoid,
\[ \text{$M$-SFT} \subsetneq \text{$M$-sofic} \subsetneq \ttfrac{$M$-SFT}{SFT} = \ttfrac{$M$-sofic}{SFT} \subsetneq \ttfrac{$M$-SFT}{sofic} = \ttfrac{$M$-sofic}{sofic}. \]
\end{theorem}

The equalities, proved in Lemma~\ref{lem:EqsBetweenClasses}, are not difficult, and essentially follow from the definitions. After proving the equalities, we see that the four classes are separated by dynamical properties: it is classical that SFTs are separated from proper sofics by the shadowing property \cite{LiMa95} and sofics from finitely presented systems by having zero dimension, and Theorem~\ref{thm:ExpansivityCharacterization} shows that systems are finitely presented exactly when they are expansive. To realize the separating systems, we use SFTs that draw paths on the monoid, and overlay one-dimensional systems on the paths.

A theorem of Ma\~n\'e \cite{Ma79} states that cyclic groups admit expansive actions only on finite-dimensional spaces. Since expansive dynamical systems are precisely the \ttfrac{subshift}{SFT} systems, Ma\~n\'e's theorem becomes the statement that all \ttfrac{subshift}{SFT} on $\Z$ have finite topological dimension. A variant of the proof of Theorem~\ref{thm:RelationsBetweenClasses} gives a converse.

\begin{theorem}
\label{thm:ManeGroups}
Let $G$ be a finitely generated group. The following are equivalent:
\begin{itemize}
\item $G$ is not virtually cyclic,
\item $G$ acts expansively on an infinite-dimensional compact metrizable space,
\item there is an infinite-dimensional \ttfrac{$G$-subshift}{SFT},
\item there is a \ttfrac{$G$-SFT}{SFT} which contains a homeomorphic copy of every compact metrizable space.
\end{itemize}
\end{theorem}

This is proved in Section~\ref{sec:ManeGroups}. In \cite{MeTs19}, a ``mean'' version of Ma\~n\'e's theorem is proved for actions of $\Z^d$, and it is observed that Ma\~n\'e's theorem as such fails for $\Z^d$-actions. The difficult case in the proof of Theorem~\ref{thm:ManeGroups} is when $G$ is a torsion group. In this case, the \ttfrac{subshift}{subshift} point of view is used essentially, as the construction method comes from symbolic dynamics on groups.

The group-theoretic basis of the construction is the existence of infinite rays (injective paths) in any group that is not virtually cyclic. This is Lemma~\ref{lem:Rays}, for which we cite Ballier and Stein \cite{BaSt18} who deduce this from results of Woess \cite{Wo89}. We also include another argument in Appendix~\ref{sec:RaysAlternative} based on results of Halin \cite{Ha65} and an appeal to Stallings theorem about ends of groups \cite{St68,St71} (where the use of Stallings theorem can be replaced by an appeal to an argument of Tointon and Yadin \cite{ToYa16}).

We show that a simple modification of the formula $d(x, y) = \inf \{2^{-n} \;|\; x_{[-n,n]} = y_{[-n,n]} \}$ for the Cantor metric leads to a metric in one-dimensional \ttfrac{sofic}{sofic} systems, i.e. \ttfrac{$M$-sofic}{sofic} systems with $M \in \{\N, \Z\}$, more precisely:

\begin{theorem}
\label{thm:SoficSoficMetric}
If $\ccfrac{Y}{Z}$ is \ttfrac{$\Z$-sofic}{sofic}, then for some $c \in \N$, there is a unique maximal function $d : Y^2 \to \R_+$ satisfying
\[ \forall n: d(x, y) \leq d_{cn}(x_{[-cn,cn]}, y_{[-cn,cn]}) / 2^{n - 1}, \]
where $d_m$ is the path metric in the graph with nodes the words of $Y$ of length $2m+1$, with an edge $(u, v)$ if and only if $([u] \times [v]) \cap Z \neq \emptyset$. This function $d$ metrizes $Y/Z$, is upper semicomputable, and given points $x, y$ and $\epsilon > 0$, one can effectively compute a rational $r$ such that $d(x, y) \in [r, \frac52 r+\epsilon]$. The obvious analog holds for \ttfrac{$\N$-sofic}{sofic} systems.
\end{theorem}

It follows from general topological principles that the first sentence of the statement holds, with $n \mapsto cn$ replaced by a general function. The fact that it can be chosen linear seems to be less trivial, and follows from a type of hyperbolicity, which holds in the sofic-per-sofic case by an application of the pigeonhole principle. We show in Section~\ref{sec:Metrization} that for general \ttfrac{subshift}{sofic} systems, the function used in this construction may need to grow arbitrarily fast.

As a simple corollary, we obtain in Section \ref{sec:SoficSoficMane} a version of Ma\~n\'e's theorem for one-dimensional sofically presented systems.

\begin{theorem}
\label{thm:TopDimFinite}
Suppose $M \in \{\N, \Z\}$ and $X$ is sofically presented on $M$. Then the topological dimension of $X$ is finite.
\end{theorem}

It is an open problem whether it is decidable whether two given subshifts of finite type, or more generally two sofic shifts, are topologically conjugate. Homeomorphism of sofic shifts is decidable \cite{He91}. For finitely presented systems, we also do not know whether conjugacy is decidable. However, we do have semidecidability.

\begin{theorem}
For $M \in \{\N, \Z\}$, the conjugacy of finitely presented $M$-systems is semidecidable.
\end{theorem}

The reason for this is that morphisms between finitely presented systems have a combinatorial presentation, and we can decide properties of these morphisms using automata theory (on infinite words). In Proposition~\ref{prop:DiscreteInCompactOpen}, we prove a result similar in spirit, that \ttfrac{subshift}{SFT} systems have countable automorphism groups which are discrete in the compact-open topology.

We define \emph{automatic spaces}, a class of topological spaces. We show that the underlying space of a \ttfrac{sofic}{sofic} system on $M \in \{\N, \Z\}$ is automatic, realize simplicial complexes as automatic spaces and realize a large class of automatic spaces as sofically presented systems. This gives the following:

\begin{theorem}
\label{thm:ConjugacyUndecidable}
Let $M = \N$ or $M = \Z$. Given two sofically presented systems $X, Y$, the following properties are recursively inseparable:
\begin{itemize}
\item $X$ and $Y$ are conjugate,
\item $X$ and $Y$ are not homeomorphic.
\end{itemize}
\end{theorem}

In particular, both conjugacy and homeomorphism of given sofically presented systems are undecidable. We do not know if these problems are even semidecidable. We illustrate the non-local nature of morphisms between such systems by showing that the automorphism group of a countable \ttfrac{$\Z$-sofic}{sofic} system can be uncountable and need not be locally compact in the compact-open topology.

The class \ttfrac{$\Z$-subshift}{subshift} and its relation to general dynamical systems has been studied previously, see e.g. \cite{BoDo04,BoFiFi02,DoNe05}. It follows from this work that the identity map on any compact metrizable spaces is \ttfrac{$\Z$-subshift}{subshift} (we also give a self-contained proof). This shows that Ma\~n\'e's theorem does not extend to this class. It also shows that while \ttfrac{$\Z$-subshift}{subshift} systems have finite topological entropy, there is no restriction on the number of periodic points of a given period in such systems, since we can have uncountably many points of period $1$. We show that the compact ordinal $\omega+1$ under trivial dynamics is \ttfrac{$\Z$-subshift}{sofic}, and infinite \ttfrac{$\Z$-subshift}{countable sofic} systems can consist entirely of periodic points, but infinite \ttfrac{$\Z$-sofic}{subshift} systems always have an infinite orbit.

Sofics appears naturally in the study of $\beta$-shifts: Given a real number $\beta > 1$, one can define the $\beta$-shift $S_\beta$ as the orbit closure of all greedy $\beta$-expansions of real numbers. A special role is played by $d^*_\beta(1)$, the greedy expansion of the number $1$. It is natural to try to induce a dynamical system back onto the interval, to obtain a dynamical system where the action is ``multiplication by $\beta$''. There is a canonical way to do so, in the sense that there exists a unique quotient that makes this well-defined and makes a minimal number of identifications in the sense of a universal property. Let $K_\beta$ be the uniquely determined equivalence relation that makes these identifications.

It is known that $S_\beta$ is sofic if and only if $d^*_\beta(1)$ is eventually periodic, and SFT if and only if $d^*_\beta(1)$ is totally periodic. The following theorem extends this, and classifies all possible soficity/SFTness types of the pair $\ccfrac{S_\beta}{K_\beta}$.

\begin{theorem}
\label{thm:BetaShiftClassification}
Let $\beta > 1$. Then the following hold:
\begin{enumerate}
\item $\ccfrac{S_\beta}{K_\beta}$ is \ttfrac{SFT}{SFT} if and only if $\beta \in \N$,
\item $\ccfrac{S_\beta}{K_\beta}$ is \ttfrac{SFT}{proper sofic} if and only if $\beta \notin \N$ and $d^*_\beta(1)$ is totally periodic,
\item $\ccfrac{S_\beta}{K_\beta}$ is \ttfrac{proper sofic}{proper sofic} if and only if $d^*_\beta(1)$ is strictly eventually periodic,
\item $\ccfrac{S_\beta}{K_\beta}$ is \ttfrac{non-sofic}{SFT}
if and only if $\{\sigma^i(d^*_\beta(1)) \;|\; i\in\N\}$ is dense in $S_\beta$.
\item $\ccfrac{S_\beta}{K_\beta}$ is \ttfrac{non-sofic}{non-sofic} in other cases.
\end{enumerate}
In particular, $\ccfrac{S_\beta}{K_\beta}$ is never \ttfrac{SFT}{non-sofic}, \ttfrac{non-sofic}{proper sofic} or \ttfrac{proper sofic}{SFT}.
\end{theorem}

This theorem is proved in Section~\ref{sec:BetaShifts}.

Another case study we perform is the study of expansive two-dimensional toral automorphisms. It is well-known that these maps are precisely the expansive $\Z$-actions on the torus. It is well-known that they have SFT covers \cite{SiVe98}, and by a result of Fried \cite{Fr87}, the kernels of the maps must be SFT by expansivity.

As there is a natural finite description (or G\"odel numbering) of sofically presented dynamical systems, one can ask decidability questions about them, and in particular one can ask to what extent the SFT covers and SFT kernels can be computed. We show that expansive toral automorphisms on the $2$-torus form a recursively enumerable set:

\begin{theorem}
\label{thm:TorusClassification}
Every expansive homeomorphism on the two-dimensional torus is finitely presented. The class of finitely presented dynamical systems (represented as pairs $\ccfrac{Y}{Z}$ in \ttfrac{$\Z$-SFT}{SFT}) that are systems on the two-dimensional torus is recursively enumerable.
\end{theorem}

This theorem is proved in Section~\ref{sec:ToralAutomorphisms}.

We give an algorithm for computing a finite presentation from a matrix presentation of the expansive homeomorphism, and apply the algorithm to the well-known expansive $\Z$-action given by the matrix $\left(\begin{smallmatrix} 1 & 1 \\ 1 & 0 \end{smallmatrix}\right)$. We denote the topological group $\R/\Z$ by $\T$.

\begin{theorem}
\label{thm:TorusExample}
Let $M = \left(\begin{smallmatrix} 1 & 1 \\ 1 & 0 \end{smallmatrix}\right)$. Then the linear action of $M$ induces an expansive homeomorphism on $\T^2$. This dynamical system is \ttfrac{SFT}{SFT}, namely conjugate to $X/K$, where the minimal forbidden patterns for $X$ are $\{11\}$ and for $K$ those listed in Figure~\ref{fig:ptrns}.
\end{theorem}

\begin{figure}[h]
\input{relationaslatex}
\vspace{-10pt}
\caption{The minimal forbidden patterns for $K$ in Theorem~\ref{thm:TorusExample}.}
\label{fig:ptrns}
\end{figure}

In Section~\ref{sec:ToralAutomorphisms}, we give a human understandable description of this relation $K$ as the composition of two sofic relations, and in Appendix~\ref{sec:ToralAutoCode} we show how to derive the forbidden patterns by computer. The complete program is included in ancillary files. The topological partition used in the proof of the theorem arises from the general method, and is the same as the one used in \cite{Ad98,SiVe98}. A square root of the kernel appears in \cite{SiVe98} (it is not stated explicitly that it needs to be squared).

\section{Conventions, definitions and basic topology}
\label{sec:Definitions}

In general $Y$ is a topological space and $Z \subset Y^2$, we denote by $Y/Z$ the corresponding topological quotient, whose underlying set is the set of $Z$-equivalence classes of $Y$ (this should not cause confusion with the subshift case, where we do not necessarily require $Z \subset Y^2$, and take this intersection explicitly). It is known that $Y/Z$ is a (compact) metrizable space whenever $Y$ is compact metrizable and $Z$ is a closed equivalence relation, see Lemma~\ref{lem:QuotientMapsMetrizable}. We denote the points of $Y/Z$ by $[y]$ where $y \in Y$ is a representative, or simply directly see points $y \in Y$ as points of $Y/Z$ when this causes no confusion. The \emph{$Z$-closure} of a set $U\subset Y$ consists of those points that are $Z$-equivalent to some point of $U$ (which need not be topologically closed), and $U$ is \emph{$Z$-closed} if $U$ is equal to its $Z$-closure. We sometimes equate without explicit mention a $Z$-closed set $U\subset Y$ with the set $V\subset Y/Z$ such that $U$ is the union of equivalence classes contained in $V$.

As in the introduction, by a dynamical system we mean an action of a discrete monoid $M$ on a compact metrizable space $X$ by continuous maps. We act from the left and $1_M \cdot x = x$, $mn \cdot x = m \cdot (n \cdot x)$, where we often omit ``$\cdot$''. Also for $X$ not necessarily a topological space, we write $M\curvearrowright X$ if $M$ acts on $X$, and the action of $m\in M$ on $x\in X$ is denoted by $mx$.  We say that a subset $X'\subset X$ is $M$-invariant or closed under the action of $M$, if $mx'\in X'$ for every $m\in M$, $x'\in X'$. If $M\curvearrowright X$ and $M\curvearrowright X'$, then the monoid $M$ acts on $X\times X'$ by the \emph{diagonal action} $m(x,x')=(mx,mx')$ for $m\in M$ and $(x,x')\in X\times X'$. If $R$ is an equivalence relation on a set $X$, for every $x\in X$ we denote by $[x]_R$ the equivalence class of $x$ with respect to $R$ and by $X/R$ the set of all equivalence classes of $X$ with respect to $R$. The trivial equivalence relation or the diagonal relation of $X$ is $\Delta_X=\{(x,x)\mid x\in X\}$. The canonical projection $\pi_R:X\to X/R$ maps $\pi_R(x)=[x]_R$ for every $x\in X$. If $R$ is an $M$-invariant equivalence relation for $X$, this induces a unique $M$-action on $X$ such that $m\pi_R(x) = \pi_R(mx)$ for $m\in M$ and $x\in X$.

When $M\curvearrowright X$ and $M\curvearrowright Y$ are topological dynamical systems and $f:X\to Y$ is a surjective continuous map that commutes with the translation maps (i.e. $f(mx)=mf(x)$ for $m\in M$, $x\in X$), we say that $f:X\to Y$ is a \emph{factor map} and $Y$ is a \emph{factor} of $X$. We also say that $f:X\to Y$ is a \emph{cover} or \emph{representation} of $Y$ or simply that $X$ is a cover or representation of $Y$ (via $f$). The kernel of a cover $f:X\to Y$ is the equivalence relation $K=\{(x,y)\in X\times X\mid f(x)=f(y)\}$. If $f$ is a bijection, we say that $f$ is a conjugacy from $X$ to $Y$ and that $X$ and $Y$ are conjugate, which is denoted by $\cong$.

If $X\subset \Sigma^M$ is a subshift, the continuous maps via which the monoid $M$ acts are called \emph{shift maps} or \emph{shifts}. A \emph{block map} $f:X\to Y$ (for subshifts $X \in \Sigma^M, Y \subset (\Sigma')^M$) is a continuous map that commutes with the shifts, i.e. $f(mx)=mf(x)$ for $m\in M$, $x\in X$. Such maps are also commonly known as sliding block codes or \emph{cellular automata} when $X = Y$. A subshift $X \subset \Sigma^M$ is an \emph{SFT} if there exists a clopen set $C \subset \Sigma^M$ such that for $x \in \Sigma^M$, we have $x \in X \iff \forall m \in M: mx \notin C$. We say $Y$ is a \emph{sofic shift} if it is the image of an SFT (possibly with a distinct alphabet) in a block map.

When we consider subshifts $X\subset\Sigma^M$ with $M\in\{\N,\Z\}$, the elements $x\in X$ are typically considered as one-way infinite or bi-infinite sequences over $\Sigma$ and the symbol at position $i\in M$ in $x$ is denoted by $x_i$. Any finite sequence $w = w_1w_2\dots w_n$ (also the empty sequence, which is denoted by $\epsilon$) where $w_i\in \Sigma$ is a word over $\Sigma$. The number $n$ of symbols in $w$ is the length of $n$ and we write $\abs{w}=n$, $w\in\Sigma^n$. For any $a\in\Sigma^n$ the number of occurrences of $a$ in $w$ is denoted by $\abs{w}_a$.  The set of all words over $\Sigma$ is denoted by $\Sigma^*$, and the set of non-empty words is $A^+= A^*\setminus\{\epsilon\}$. For $w\in\Sigma^+$ we denote by $w^\infty$ the infinite sequence of repetitions of $w$ extending to the right and similarly by ${}^\infty w$ we denote the infinite sequence of repetitions of $w$ extending to the left. For $x\in X$ and $i,j\in\Z$ we denote $x_{[i,j]}=x_i x_{i+1}\dots x_{j}\in\Sigma^*$ (this is the empty word $\epsilon$ if $j<i$), $x_{[i,\infty)}=x_i x_{i+1}\dots$ and $x_{(\infty,i]}=\dots x_{i-1}x_i$. We say that a word $w\in\Sigma^+$ is a \emph{subword} of $x\in \Sigma^M$ if $w=x_{[i,j]}$ for some $i,j\in\Z$ and then we write $w\sqsubset x$. We also say that $w$ \emph{occurs} in $x$ (at position $i$). For $u,v\in \Sigma^*$ the cylinder set $[u.v]$ in $X$ is defined by
\[[u.v]=\{x\in X\mid x_{[-\abs{u}+1,0]}=u \mbox{ and } x_{[1,\abs{v}]}=v\},\]
and when $u$ is the empty word we denote $[v]=[\epsilon.v]$. For $n\in\N$ we denote $B_n(X)=\{w\in\Sigma^n\mid w\sqsubset x\mbox{ for some } x\in X\}$. The \emph{entropy} of $X$ is the limit $h(X)=\lim_{n\to\infty}\frac{\log B_n(X)}{n}$ and the language of $X$ is the set $L(X)=\bigcup_{n\in\N}B_n(X)$.

We say that sequences of the form $w^\infty$ are \emph{(totally) periodic} and sequences of the form $uw^\infty$ are \emph{eventually periodic} (where $w\in \Sigma^+$ and $u\in\Sigma^*$). Similarly, we say that sequences of the form ${}^\infty w.w^\infty$ are (totally) periodic and sequences of the form ${}^\infty wu.vw^\infty$ are eventually periodic (where $w\in \Sigma^+$ and $u,v\in\Sigma^*$). A \emph{strictly eventually periodic} sequence is one that is eventually, but not totally, periodic.

For a discrete monoid $M$, we say that a dynamical system $(X,M)$ is \emph{transitive} if for all nonempty open sets $U,V\subset X$ there is $x\in U$ and $m\in M$ such that $mx\in V$. We say that $(X,M)$ is \emph{expansive} if for some metric $d:X\times X\to \R$ and $\epsilon>0$ it holds that for any distinct $x,y\in X$ there is $m\in M$ such that $d(mx,my)>\epsilon$. Then we say that $\epsilon$ is an \emph{expansivity constant} of $(X,M)$ (with respect to the metric $d$). That $M\curvearrowright X$ is expansive is equivalent the following statement: There exists an open set $U\subset X\times X$ such that $\Delta_X\subset U$ and for any distinct $x,y\in X$ there is $m\in M$ such that $m(x,y)\notin U$. Then we say that $U$ is an expansive index of $(X,M)$.

For the purpose of computability, we need a way to encode sofic shifts and subshifts of finite type as a binary word. As all our results are about decidability, and we do not aim for efficient algorithms, any reasonably encoding can be used. One should think of a sofic shift as being given by providing a finite-state automaton (or regular expression) for its language or colanguage, and one can convert effectively between such representations, see \cite{HoMoUl06,LiMa95} and Section~\ref{sec:AutomataTheory}. SFTs can be presented also by an explicit list of its forbidden patterns, and one can transform to and from this representation as well (see Lemma~\ref{lem:DecidableThings}).

We recall some basic lemmas from topology. The following three can be found e.g. in Appendix~A of the book~\cite{Ku03}.

\begin{lemma}
\label{lem:QuotientMapsCharacterization}
Let $X$ be a compact Hausdorff space, and $f : X \to Y$ a quotient map. Then the following are equivalent:
\begin{enumerate}
\item $Y$ is Hausdorff, \label{it:Hd}
\item $f$ is closed, \label{it:Cl}
\item $\ker f$ is closed. \label{it:Rel}
\end{enumerate}
\end{lemma}

\begin{lemma}
\label{lem:QuotientMapsMetrizable}
Let $X$ be a compact metrizable space, and $f : X \to Y$ a quotient map, with $Y$ Hausdorff. Then $Y$ is metrizable.
\end{lemma}

\begin{lemma}
\label{lem:IsAHomeomorphism}
Let $X$ and $Y$ be topological spaces and let $f : X \to Y$. If $X$ is compact, $Y$ is Hausdorff, and $f$ is a continuous bijection, then $f$ is a homeomorphism between $X$ and $Y$.
\end{lemma}

The following is useful for constructing abstract quotients, usually used implicitly (see e.g. \cite{AuBaKeLe17}).

\begin{lemma}
\label{lem:SubsystemsFactors}
Let $M$ be a discrete monoid and $X$ be a compact metrizable $M$-system. Then the factors of $X$ are in one-to-one correspondence with equivalence relations in $X$ that are subsystems of $X^2$.
\end{lemma}

\begin{proof}
Let $E \subset X^2$ be an equivalence relation of $X$ which is a subsystem, i.e. closed and $M$-invariant. Then we define a dynamical system $X_E$ and a factor map $f_E : X \to X_E$ as follows: points of $X_E$ are equivalence classes of $E$, a set is open if it is a collection of equivalence classes whose union is open in $X$, the dynamics is $m[x]_E = [mx]_E$ where $[\cdot]_E$ denotes the $E$-equivalence class, and of course $f(x) = [x]_E$. If $f : X \to Y$ is a factor map, we define $E_f = \ker(f) = \{(x, x') \in X^2 \;|\; f(x) = f(x')\}$.

It is a straightforward computation to show that $X_E$ is well-defined and a factor of $X$.
It is also straightforward that $E_f$ is a subsystem and an equivalence relation.

We need to show ${E_{f_E}} = E$ and that $X_{E_f} \cong Y$ for any factor map $f : X \to Y$. The first claim comes directly from the definitions. For the second, let $g([x]) = f(x)$ for all $x$, and observe that this is a bijection from $X_{E_f} \cong Y$. It is continuous by the definition of the topology of $X_{E_f}$ and since $f$ is continuous. Since $X_{E_f}$ is compact and $Y$ is Hausdorff, it is a homeomorphism. Finally $ g(m[x]) = g([mx]) = f(mx) = mf(x) = mg([x])$ for all $x \in X$, so indeed $X_{E_f}$ is conjugate to $Y$ by $g$.
\end{proof}

We warn the reader that the formula $m[x]_E = [mx]_E$ in the previous lemma does not agree with the common convention that an operation on points, applied to a set, is simply applied to elements of the set, though this does happen when $M$ is a group.

In dynamical terms, division by a general subshift means the same as passing to an arbitrary factor, and the following is a direct consequence of Lemma~\ref{lem:SubsystemsFactors}.

\begin{lemma}
For any class of subshifts $S_1$, a dynamical system is \ttfrac{$S_1$}{subshift} if and only if it is a factor of a subshift in $S_1$.
\end{lemma}

\begin{lemma}
\label{lem:KClosedOpens}
Let $Y$ be a compact metrizable topological space, $K \subset Y^2$ a closed equivalence relation, $A \cap B = \emptyset$ two sets which are both closed and $K$-closed. Then there exist $K$-closed open sets $U \cap V = \emptyset$ such that $A \subset U, B \subset V$.
\end{lemma}

\begin{proof}
This is equivalent to the fact that $Y/K$ is normal, which follows from Lemma~\ref{lem:QuotientMapsMetrizable} and the fact that compact Hausdorff spaces are normal.
\end{proof}

\begin{example}
It is possible that $m[x]_E \neq \{my \;|\; (y, x) \in E\}$ even if $M$-translations are surjective. For example take the $\N$-subshift $X = \bigcup_{i \in \N} \{0^i 2 0^\omega\} \cup \{0,1\}^\omega$ and
\[ E = (\{0,1\}^2)^\N \cup \{(0^i 2 0^\omega, 0^i 2 0^\omega) \;|\; i \in \N\}. \]
Then $[20^\omega]_E = \{20^\omega\}$, $1 \cdot 20^\omega = 0^\omega$, yet $1 \cdot [20^\omega]_E = [0^\omega]_E = \{0,1\}^\N$ according to the definition. \qee
\end{example}

It is well-known that multiplication by $2$ modulo an interval is expansive \cite{Ku03}. The general theory \cite{Fr87} says that it must be in the class \ttfrac{subshift}{SFT}, and indeed this is the case, by a direct binary encoding. We define two useful finitely presented systems that arise from this observation.

\begin{definition}
\label{def:BinaryReals}
Let
\begin{align*} A_{\mathrm{req}} = \{
&\left( \begin{smallmatrix} a & b & c \\ a & b & c \end{smallmatrix} \right),
\left( \begin{smallmatrix} a & b & c \\ a & b & (1-c) \end{smallmatrix} \right),
\left( \begin{smallmatrix} a & b & (1-b) \\ a & (1-b) & b \end{smallmatrix} \right), \\
&\left( \begin{smallmatrix} a & (1-a) & (1-a) \\ (1-a) & a & a \end{smallmatrix} \right),
\left( \begin{smallmatrix} (1-a) & (1-a) & (1-a) \\ a & a & a \end{smallmatrix} \right)
\;|\; a,b,c \in \{0,1\} \}.
\end{align*}
\end{definition}

For $z \in \{0,1\}^\N$, write $\bin(z) \in [0,1)$ for the real number modulo~$1$ represented by the binary expansion $0.z$. The proofs are left as exercises, but the proof of the first lemma can also be obtained from more general results in Section~\ref{sec:BetaShifts} (i.e. Theorem~\ref{Kchar}, Lemma~\ref{Kint} and the proof of Theorem~\ref{thm:Kint}).

\begin{lemma}
\label{lem:TimesTwoN}
Let $S_\N = \{0,1\}^\N$ and define $E_\N$ as the SFT of $Y^2$ with the allowed patterns $A_{\mathrm{req}}$. Then $E_\N$ is an equivalence relation and $x, y \in Y \iff \bin(x) = \bin(y)$. The \ttfrac{$\N$-SFT}{SFT} $S_\N/E_\N$ is conjugate to $x \mapsto 2x$ on $\R/\Z$.
\end{lemma}

\begin{lemma}
\label{lem:TimesTwoZ}
Let $S_\Z = \{0,1\}^\Z$ and define $E_\Z$ as the SFT of $S_\Z^2$ with the allowed patterns $A_{\mathrm{req}}$. Then $E_\Z$ is an equivalence relation and $x, y \in Y \iff \bin(x_{[i, \infty)}) = \bin(y_{[i, \infty)})$ for all $i$. The \ttfrac{$\Z$-SFT}{SFT} $S_\Z/E_\Z$ is conjugate to the map $x \mapsto 2x$ on the $2$-adic solenoid.
\end{lemma}

These lemmas state that finite-state automata understand the equivalence between real numbers and their binary representations.

Countable sofic shifts are a useful class of (counter)examples. One important family of countable (and in many cases sofic) subshifts is the following.

\begin{definition}
For a right-cancellative monoid $M$ (clear from context) and for $n\in\N$, $X_{\leq n}$ denotes the subshift of $\{0,1\}^M$ consisting of those points $x$ that contain at most $n$ occurrences of the digit $1$. i.e. $|\{ m\in M \;|\; x_m = 1\}| \leq n$.
\end{definition}

\section{Expansivity and subshift quotients}
\label{sec:DivisionBySFT}

In this section, we prove Theorem~\ref{thm:ExpansivityCharacterization} from the introduction.

\begin{theorem}
\label{thm:ExpansivityCharacterizationProof}
Let $M$ be any discrete countable monoid and $X$ a dynamical $M$-system. The following are equivalent:
\begin{itemize}
\item $X$ is expansive,
\item $X$ is \ttfrac{subshift}{SFT}, and
\item $X$ is conjugate to $Y/Z$ for some subshifts $Y, Z$, and whenever $X \cong Y/Z$ where $Y$ and $Z$ are subshifts, $Z$ is a relative SFT in $Y$.
\end{itemize}
If $M$ a group, then all are equivalent to the following:
\begin{itemize}
\item $X$ is conjugate to $Y/Z$ for a subshift $Y$ and a relative SFT $Z$ in $Y$ such that the set $\{ y \in Y \;|\; \exists y' \neq y: (y, y') \in Z \}$ is meager in $Y$.
\end{itemize}
\end{theorem}

\begin{proof}
Suppose first that $X$ is an expansive $M$-system for some discrete monoid $M$. By Lemma~\ref{lem:ExistenceOfCovers}, there exists an $M$-subshift $Y$ and an $M$-SFT $Z$ such that $X \cong Y/Z$, so the second claim holds. By Lemma~\ref{lem:IffRelativeSFT}, the third claim then also holds. On the other hand, Lemma~\ref{lem:IffRelativeSFT}, shows that either of the second and third item implies expansivity, so all three conditions are equivalent.

When $M$ is a group, the cover given by Lemma~\ref{lem:ExistenceOfCovers} has the claimed property.
\end{proof}

The only importance of $M$ being a group is that this algebraic condition forces $M$ to act by quasi-open maps, so if we require dynamical systems to always have this property, then the last item is equivalent to the others for all monoids. We do not know whether it is always equivalent.

There are dynamical systems that are not conjugate to $Y/Z$ for any subshifts $Y, Z$, and in these cases, ``whenever $X \cong Y/Z$ where $Y$ and $Z$ are subshifts, $Z$ is an SFT in $Y$'' holds trivially, and thus this claim alone is not equivalent to expansivity.

The following is well-known, see e.g. \cite{Fr87}.

\begin{lemma}
\label{lem:IffRelativeSFT}
Let $M$ be a discrete monoid, $Y \subset \Sigma^M$ a subshift, and $K \subset Y^2$ a subshift equivalence relation, then $Y/K$ is expansive if and only if $K$ is a relative SFT in $Y^2$.
\end{lemma}

\begin{proof}
Suppose first that $K$ is a relative SFT in $Y^2$. Let $C \subset Y^2$ be the forbidden clopen set defining $K$. Then the set
\[K^2\doteq\{((x,y),(x',y'))\in(Y^2)^2\mid (x,x'),(y,y')\in K\}\]
is an equivalence relation on $Y^2$. The $K^2$-closure $C'$ of $C$ in $Y^2$ is closed and $K^2$-closed, and does not intersect $K$. By Lemma~\ref{lem:KClosedOpens}, there is an open $K^2$-closed set $D$ such that $K^2 \subset D \subset (C')^c \subset C^c$. It follows that $D$ is a neighborhood of the diagonal of $Y/K$ which separates points, so $Y/K$ is expansive.

Suppose then that $Y/K$ is expansive and let $D \subset (Y/K)^2$ be an open set defining an expansive index. Then $D$ is open also in $Y^2$ and the diagonal $K$ of $(Y/K)^2$ is closed also in $Y^2$ (by the definition of $K$), so we can find clopen set $C$ between $K$ and $D$. This $C$ is a set of allowed patterns for $K$.
\end{proof}

Now, let us construct symbolic presentations for expansive dynamical systems. There are several formulas one may use, and we begin with a brief study of their general properties.

\begin{definition}
Let $X$ be compact metrizable, $M$ acts on $X$. Let $\mathcal{U} = (U_1,\dots,U_n)$ be any finite family of open sets in $X$. To $\mathcal{U}$ we associate three subshifts over the alphabet $\Sigma = \{1,2,\dots,n\}$: The \emph{small orbit subshift} is defined as
\[ X_{\mathcal{U}}^s = \{ x \in \Sigma^M \;|\; \forall \mbox{ finite } F \subset M: \overline{\bigcap_{m \in F} m^{-1} U_{x_m}} \neq \emptyset \} \]
the \emph{medium orbit subshift} as
\[ X_{\mathcal{U}}^m = \{ x \in \Sigma^M \;|\; \forall \mbox{ finite } F \subset M: \bigcap_{m \in M} \overline{m^{-1} U_{x_m}} \neq \emptyset \} \]
and the \emph{large orbit subshift} as
\[ X_{\mathcal{U}}^\ell = \{ x \in \Sigma^M \;|\; \forall \mbox{ finite } F \subset M: \bigcap_{m \in M} m^{-1} \overline{U_{x_m}} \neq \emptyset \}. \]
\end{definition}

We also use the following alternate forms, which are obviously equivalent:
\[ X_{\mathcal{U}}^s = \{ x \in \Sigma^M \;|\; \forall \mbox{ finite } F \subset M: \bigcap_{m \in F} m^{-1} U_{x_m} \neq \emptyset \}, \]
\[ X_{\mathcal{U}}^m = \{ x \in \Sigma^M \;|\; \bigcap_{m \in M} \overline{m^{-1} U_{x_m}} \neq \emptyset \}, \]
\[ X_{\mathcal{U}}^\ell = \{ x \in \Sigma^M \;|\; \bigcap_{m \in M} m^{-1} \overline{U_{x_m}} \neq \emptyset \}. \]

The small orbit subshift is the standard one used in the construction of Markov partitions, see e.g. \cite{Ad98,LiMa95}. When the small orbit subshift gives a cover, it gives a very nice cover. The large orbit subshift on the other hand a gives subshift cover under no other assumption than expansivity, and is what we use in the proof. The large orbit subshift is used in \cite{Re68}.

We have no use for the medium orbit subshift, and only include it for completeness.

We prove some basic properties of these subshifts in the following lemma. 
A \emph{topological partition} is a family of disjoint open sets whose closures cover the space.

The following lemma shows some basic properties of these definitions.

\begin{lemma}
\label{lem:OrbitSubshiftProperties}
Let $X$ be any $M$-system and $\mathcal{U}$ any finite family of open sets.
\begin{itemize}

\item $X_{\mathcal{U}}^s \subset X_{\mathcal{U}}^m \subset X_{\mathcal{U}}^\ell$ are always subshifts,

\item if each $m \in M$ acts by a homeomorphism, then the medium and the large orbit subshifts coincide for all $\mathcal{U}$.
\end{itemize}

Suppose now that $M$ is expansive, the diameters of the sets $U_i$ are all smaller than half of the expansivity constant, and $\bigcup_i \overline{U_i} = X$.

\begin{itemize}
\item There is a well-defined continuous shift-commuting function $\pi_* : X_{\mathcal{U}}^* \to X$,

\item if $x \in X_{\mathcal{U}}^s$, then $\pi_s(x) = \pi_m(x)$, and if $x \in X_{\mathcal{U}}^m$, then $\pi_m(x) = \pi_\ell(x)$,

\item the map $\pi_\ell$ is surjective,

\item if all of the $m \in M$ act by quasi-open maps, then $\pi_s$ is surjective, and if also the $U_i$ form a topological partition, then $\pi_s$ is almost-$1$-$1$.
\end{itemize}
\end{lemma}

\begin{proof}
First let us show all are subshifts. They are closed by the first set of formulas, as they are defined by forbidding cylinders. In each case, shift invariance follows by straightforward computation.

The set $X_{\mathcal{U}}^s$ is shift-invariant because
\begin{align*}
x \in X_{\mathcal{U}}^s &\implies \forall \mbox{ finite } F \subset M: \bigcap_{m \in F} m^{-1} U_{x_m} \neq \emptyset \\
&\implies \forall \mbox{ finite } F \subset M: \bigcap_{m \in F} (mn)^{-1} U_{x_{mn}} \neq \emptyset \\
&\implies \forall \mbox{ finite } F \subset M: \bigcap_{m \in F} n^{-1} m^{-1} U_{nx_m} \neq \emptyset \\
&\implies \forall \mbox{ finite } F \subset M: n^{-1} \bigcap_{m \in F} m^{-1} U_{nx_m} \neq \emptyset \\
&\implies \forall \mbox{ finite } F \subset M: \bigcap_{m \in F} m^{-1} U_{nx_m} \neq \emptyset \\
&\implies nx \in X_{\mathcal{U}}^s.
\end{align*}

The set $X_{\mathcal{U}}^m$ is shift-invariant because
\begin{align*}
x \in X_{\mathcal{U}}^m &\implies \bigcap_{m \in M} \overline{m^{-1} U_{x_m}} \neq \emptyset \\
&\implies \bigcap_{m \in M} \overline{n^{-1} m^{-1} U_{nx_m}} \neq \emptyset \\
&\implies \bigcap_{m \in M} n^{-1} \overline{m^{-1} U_{nx_m}} \neq \emptyset \\
&\implies n^{-1} \bigcap_{m \in M} \overline{m^{-1} U_{nx_m}} \neq \emptyset \\
&\implies \bigcap_{m \in M} \overline{m^{-1} U_{nx_m}} \neq \emptyset \\
&\implies nx \in X_{\mathcal{U}}^m.
\end{align*}
where the third implication follows from the continuity of the $n$-translation.

For shift-commutation of $X_{\mathcal{U}}^\ell$,
\begin{align*}
x \in X_{\mathcal{U}}^\ell &\implies \bigcap_{m \in M} m^{-1} \overline{U_{x_m}} \neq \emptyset \\
&\implies \bigcap_{m \in M} n^{-1} m^{-1} \overline{U_{nx_m}} \neq \emptyset \\
&\implies n^{-1} \bigcap_{m \in M} m^{-1} \overline{U_{nx_m}} \neq \emptyset \\
&\implies \bigcap_{m \in M} m^{-1} \overline{U_{nx_m}} \neq \emptyset \\
&\implies nx \in X_{\mathcal{U}}^\ell.
\end{align*}

The inclusions $X_{\mathcal{U}}^s \subset X_{\mathcal{U}}^m \subset X_{\mathcal{U}}^\ell$ are also proved by direct computation.

For the second item, suppose that each $m$ is a homeomorphism. Then $\overline{m^{-1} U_{x_m}} = m^{-1} \overline{U_{x_m}}$, so the formulas defining the medium and large subshift coincide.

The third item is standard (see e.g. \cite{LiMa95}): if the diameters of the $U_i$ are smaller than half of the expansivity constant, then any $x \in X_{\mathcal{U}}^*$ is in particular in $X_{\mathcal{U}^\ell}$, and we map it to the unique point in the intersection $\bigcap_{m \in M} m^{-1} \overline{U_{x_m}}$.

For the fifth item, $\pi_\ell$ is always surjective because $mx$ is always in the closure of one of the $U_i$.

The sixth item is also essentially standard. Namely, observe that the preimage of a dense set in a quasi-open map is dense. From this, it follows that the set of points having at least one representation is residual. But clearly this set is also closed, so $\pi_s$ is surjective.

If the sets $U_i$ form a topological partition, and $y$ is a point such that $\pi_s(y) = \pi_s(y')$ for some $y' \neq y$, then let $m \in M$ be such that $y_m \neq y_m'$. Then in particular $\pi_\ell(y) = \pi_\ell(y')$, i.e. $m^{-1} (\overline{U_{y_m}} \cap \overline{U_{y_m'}}) \neq \emptyset$, and $\pi_s(y) \in m^{-1} (\overline{U_{y_m}} \cap \overline{U_{y_m'}})$. Then $y \in \pi_s^{-1}(\bigcup_{m \in M} m^{-1}(\bigcup_{a \neq b} \overline{U_a} \cap \overline{U_b}))$. The set $\bigcup_{a \neq b} \overline{U_a} \cap \overline{U_b}$ is nowhere dense, so the sets $m^{-1}(\bigcup_{a \neq b} \overline{U_a} \cap \overline{U_b})$ are nowhere dense since translations are quasi-open. Thus $\bigcup_{m \in M} m^{-1}(\bigcup_{a \neq b} \overline{U_a} \cap \overline{U_b})$ is meager.

Now, it is enough to show that $\pi_s^{-1}(A)$ is meager whenever $A$ is meager. For this, it is enough to show that $\pi_s$ is quasi-open. So let $U$ be open in $X_{\mathcal{U}}^s$ i.e. $U$ contains some cylinder $C = \{ x \in Y \;|\; \forall m \in F: x_m = y_m \}$ for some $y \in Y$ and finite $F \subset M$. Then every point in $V = \bigcap_{m \in F} m^{-1} U_{x_m}$ has a $\pi_s$-preimage in $C$, so $\pi_s(C)$ contains the open set $V$.
\end{proof}

Even when $M = \N$, we do not know if the small orbit subshift can always be used to give a subshift cover in the expansive case. We show by examples that at least the topological partition has to be chosen carefully for this to be the case.

\begin{proposition}
There exist
\begin{itemize}
\item $\Z$-subshifts having arbitrarily fine topological partitions $\mathcal{U}$ such that $X_{\mathcal{U}}^s \subsetneq X_{\mathcal{U}}^m$,

\item surjective $\N$-subshifts having arbitrarily fine topological partitions $\mathcal{U}$ such that $X_{\mathcal{U}}^m \subsetneq X_{\mathcal{U}}^\ell$, and

\item $\N$-subshifts having arbitrarily fine open covers $\mathcal{U}$ such that $\pi_m$ is not surjective.
\end{itemize}
\end{proposition}

\begin{proof}
Proof of the first: A natural (non-subshift) example is the shift map (i.e. multiplication by two) on the $2$-adic solenoid. We take a subshift analog of this: take the binary full shift, let $m\in\N$ and consider the following topological partition (that becomes arbitrarily fine as $m$ tends to infinity): $U_1 = \bigcup_{i \in \N} [10^{m+i}.0^{m+i+1}]$, $U_2 = \bigcup_{i \in \N} [10^{m+i}.0^{m+i}1] \cup \bigcup_{i \in \N} [0^{m+i+1}.0^{m+i}1]$, and $U_i$ for $i \in [3, n]$ enumerates cylinders $[u.w]$ where $uw \neq 0^{2m}$ and $u,w\in \{0,1\}^m$. Clearly all of $\{1, 2\}^\Z$ codes the orbit of $0^\Z$ in the medium and large orbit subshifts, but a straightforward computation shows that the word $21$ is forbidden in the small orbit subshift.

Proof of the second: Consider the subshift $\bigcup_{i \in \N} \{0^i 2 0^\omega\} \cup \{0,1\}^\omega$ and in $\mathcal{U}_n$ include the following open sets: $U_1 = [2]$, and $U_i$ for $i \geq 2$ is constructed as follows: take $U_{2,w}$ $[0w]$ for all $w \in \{0,1,2\}^n \setminus \{0^n\}$, $U_{3, w} = [1w]$ for all $w \in \{0,1\}^n$ and finally $U_4 = \bigcup_{m \geq 0} [00^{n+m}1] \cup [00^{n+m} 2]$. For any $n \in \N$ this is a finite topological partition containing all points except $0^\omega$, and the diameters of the open sets in $\mathcal{U}$ get arbitrarily small as $n \rightarrow \infty$. The shift action is clearly surjective.

Now, the corresponding medium orbit subshift does not contain any point containing the symbol $1$ (corresponding to the cylinder $U_1 = [2]$), since supposing $x_n = 1$, we have $\bigcap_{m \in M} \overline{m^{-1} U_{x_m}} \neq \emptyset$ so there is some point $y$ in this intersection, in particular $y \in \overline{n^{-1} [2]}$, i.e. $y_n = 2$. But then $y_{n+i} = 0$ for all larger $i$, so in fact $(n+1)y = 0^\omega$. But no element of $\mathcal{U}$ contains $0^\omega$, so $y \notin (n+1)^{-1} U_j$ for any $j$. Because $y$ is also an isolated point, it follows that $y \notin \overline{(n+1)^{-1} U_j}$ for any $j$. On the other hand, all shifts of the point $20^\omega$ are in closures of the $U_i$, and the representing point satisfies $x_0 = 1$, so $X_{\mathcal{U}}^\ell$ contains a point that begins with $1$.

The previous example and its proof also gives the third item.
\end{proof}

\begin{lemma}
\label{lem:ExistenceOfCovers}
Let $M$ be a discrete monoid, $X$ a compact metric space and $M \curvearrowright X$ an expansive monoid action. Then there exists a finite alphabet $\Sigma$, a subshift $Y \subset \Sigma^M$ and a relative SFT equivalence relation $Z \subset Y \times Y$ such that $X \cong Y/Z$. If each translation is quasi-open, these can be chosen so that the set of points $y$ such that $|\{z \;|\; (z,y) \in Z)\}| \geq 2$ is meager in $Y$.
\end{lemma}

\begin{proof}
Using any fine enough cover, the existence of subshifts $Y$ and $Z$ such that $Y/Z$ is conjugate to $X$ follows from the fact that $\pi_\ell$ is a factor map from the large orbit subshift onto $X$ by Lemma~\ref{lem:OrbitSubshiftProperties}. The fact that $Z$ is a relative SFT follows from Lemma~\ref{lem:IffRelativeSFT}.

For the second claim, by Lemma~\ref{lem:OrbitSubshiftProperties} the small orbit subshift also gives a cover in this case. Any family $U_1,\dots,U_n$ whose closures cover $X$ can be refined to a topological partition which is at least as fine, by removing the closure of $U_j$ from $U_i$ for every $j > i$, and then by Lemma~\ref{lem:OrbitSubshiftProperties} we obtain the second claim.
\end{proof}

By the theorem, \ttfrac{subshift}{SFT} is precisely the class of expansive systems, and moreover, \emph{every} representation of an expansive system by a subshift has a relative SFT kernel. The class of dynamical systems \ttfrac{subshift}{sofic} does not have this strong a characterization.

\begin{example}
Consider the dynamical system $X_{\leq 2}/(X_{\leq 1}^2 \cup \Delta_{X_{\leq 2}})$ on $\Z$. The base set is the unique countable compact metric space $X$ that has a unique accumulation point. It is the unique (up to conjugacy) system $(X,T)$ on $X$ with an infinite number of disjoint orbits and only one periodic point. For any set $S\subset \N_+$ let $X_S$ be the equivalence relation on $X_{\leq 2}$ which identifies all points of $X_{\leq 1}$ with those points where the distance of two 1s belongs to the set $S$. Now let $S$ be co-infinite. From this we get an uncountable number of equivalence relations $X_S$, so some of them cannot be (relatively) sofic. On the other hand, all such $X_{\leq 2}/X_S$ are conjugate to $(X,T)$. Therefore $(X,T)$ also has a sofic subshift cover whose kernel is not sofic. \qee
\end{example}

One could nevertheless hope for a dynamical characterization of \ttfrac{subshift}{sofic} systems up to conjugacy. We do not have such a characterization.

\section{Quotients between SFTs and sofics}

In this section we prove Theorem~\ref{thm:RelationsBetweenClasses}.

\subsection{Coinciding classes}

Consider the operation of taking all subshift factors of a class of subshifts. The classes of sofic subshifts, transitive subshifts, mixing subshifts, minimal subshifts and coded subshifts \cite[Corollaire~2.11]{BlHa86} are closed under this operation, and SFTs map precisely onto the sofic shifts under subshift factor maps by definition while synchronized shifts map onto coded shifts \cite[Proposition~4.1]{BlHa86}.

Write $F$ for the operator taking a class $S$ of subshifts to the class $F S$ of subshifts that are factors of subshifts in $S$. Write $E$ for the operator taking a class of subshifts to the class $E  S$ of subshifts that are preimages of subshifts in $S$ under block maps between full shifts, i.e. $X \in S, X \subset A^M, f : B^M \to A^M \implies f^{-1}(X) \in E S$.

\begin{lemma}
\label{lem:soficSFTSFTSFT}
If $S_2 = ES_2$, then $\ttfrac{$S_1$}{$S_2$} = \ttfrac{$FS_1$}{$S_2$}$.
\end{lemma}

\begin{proof}
Let $M$ be the monoid on which the subshifts are defined.

Of course $\ttfrac{$S_1$}{$S_2$} \subset \ttfrac{$F S_1$}{$S_2$}$ since the identity map is a factor map. Suppose then $X \in \ttfrac{$F S_1$}{$S_2$}$, so $X = Y/Z$ where $Y \subset A^M$, $f(Y') = Y$ for some block map $f : B^M \to A^M$, $Y' \in S_1$, $Z \subset (A^2)^M$, $Z \in S_2$.

Let $Z' \subset (B^2)^M$ be the preimage of $Z$ under $f \times f$. Clearly $f \times f : (B^2)^M \to (A^2)^M$ is a block map, so $Z'$ is in $S_2$. We now have $Y'/Z' \cong Y/Z$: The map from $Y'/Z'$ to $Y/Z$ mapping $[y']$ to $[f(y')]$ is well-defined by the definition of $Z'$, is continuous since $f$ is continuous, is a bijection by the definition of $Z'$, and therefore is a homeomorphism by Lemma~\ref{lem:QuotientMapsCharacterization} and Lemma~\ref{lem:IsAHomeomorphism}. That it commutes with the shift is a direct computation.
\end{proof}

\begin{lemma}
\label{lem:EqsBetweenClasses}
For any monoid $M$,
\[ \text{$M$-SFT} \subset \text{$M$-sofic} \subset \ttfrac{$M$-SFT}{SFT} = \ttfrac{$M$-sofic}{SFT} \subset \ttfrac{$M$-SFT}{sofic} = \ttfrac{$M$-sofic}{sofic} \]
\end{lemma}

\begin{proof}
Since SFT = $E$SFT, $\ttfrac{$S$}{SFT} = \ttfrac{$FS$}{SFT}$ for all $S$, in particular $\ttfrac{SFT}{SFT} = \ttfrac{sofic}{SFT}$ by the previous lemma. Similarly, since sofic = $E$sofic, $\ttfrac{$S$}{sofic} = \ttfrac{$FS$}{sofic}$ for all $S$, in particular $\ttfrac{SFT}{sofic} = \ttfrac{sofic}{sofic}$. Then the inclusions are also clear.
\end{proof}

\subsection{Separating the classes}

\begin{lemma}
Let $M$ be any finitely generated infinite monoid, and $S \subset M$ a finite generating set. Then there exists an infinite sequence $s_1, s_2, s_3, \ldots$, $s_i \in S$, such that all the elements
\[ 1_M, s_1, s_2s_1, s_3s_2s_1, \dots \]
are distinct.
\end{lemma}

\begin{proof}
Apply K\H{o}nig's lemma to the left Cayley graph, i.e. the simple undirected graph with nodes $M$ and edges $(m, sm)$ with $s \in S$.
\end{proof}

\begin{theorem}
\label{thm:ProperSofic}
For any finitely generated infinite monoid $M$,
\[ \text{$M$-SFT} \subsetneq \text{$M$-sofic} \subsetneq \ttfrac{$M$-SFT}{SFT} = \ttfrac{$M$-sofic}{SFT} \subsetneq \ttfrac{$M$-SFT}{sofic} = \ttfrac{$M$-sofic}{sofic}. \]
\end{theorem}

\begin{proof}
Lemma~\ref{lem:EqsBetweenClasses} shows everything except the separations between the classes.

For the properness of the inclusions on $\N$, $X_{\leq 1}$ on $\N$ is proper sofic, the finitely presented system $x \mapsto 2x$ on the unit circle $S^1$ from Lemma~\ref{lem:TimesTwoN} does not have a zero-dimensional underlying space so it is not sofic, and the sofically presented system $X_{\leq 2}/(X_{\leq 1}^2 \cup \Delta_{X_{\leq 2}})$ is not expansive so it is not finitely presented.

On a general f.g.\ infinite monoid $M$, we will simulate these separating systems. First, we construct an SFT that draws paths on $M$. Let $S \not\ni 1_M$ be a generating set for $M$, let $A = S \cup \{0\}$ where $0 \notin S$ and define an SFT $X \subset A^M$ by
\[ x \in X \iff (\forall m \in M: x_{m} = s \in S \implies x_{sm} \in S). \]

Every configuration $x \in X$ determines, for any $m \in M$ such that $x_m \neq 0$, a unique (possibly eventually cyclic) path in the left Cayley graph of $M$, obtained in a natural way by following the generators. By the previous lemma, there is an infinite path starting from the origin in the left Cayley graph, and thus $X$ contains a configuration $z \neq 0^M$ where $z_{1_M} \neq 0$, the path from $1_M$ that it determines is infinite, and all non-$0$ symbols in $z$ appear on this path.

Now, we can easily separate the levels, by adapting the constructions for $\N$ to take place on each of the infinite paths in configurations of $X$. We begin with the nonemptiness of the class \ttfrac{SFT}{SFT} $ \setminus $ sofic, i.e.\ finitely presented non-sofic systems. For this, we overlay configurations of $S_\N$ from Lemma~\ref{lem:TimesTwoN} (the binary full shift) on each path, and identify two configurations if on every path, the configurations overlaid on the path are in the $E_\N$-relation.

More precisely, we use the alphabet $\Sigma = \{0\} \cup (S \times \{0,1\})$, and define $Y \subset \Sigma^M$ with the SFT rules of $X$ described above, ignoring the bit in each symbol in $S \times \{0,1\}$. We define $Z$ to be the following SFT: The alphabet is $\Sigma^2$, and the SFT rules state that if $(x, y) \in Z$ then for all $m \in M$ either $(x, y)_m = (0, 0)$ or $(x, y)_m \in (S \times \{0,1\})^2$, and in this case we require
$x_m = (s, a)$, $x_{sm} = (s', a')$, $x_{s'sm} = (s'', a'')$ and $y_m = (s, b)$, $y_{sm} = (s', b')$, $y_{s'sm} = (s'', b'')$, for some $\left( \begin{smallmatrix}
a & a' & a'' \\
b & b' & b''
\end{smallmatrix} \right) \in A_{\mathrm{req}}$ where $A_{\mathrm{req}}$ is the set from Definition~\ref{def:BinaryReals}. It is easy to see that $Z \cap Y^2$ is an equivalence relation, so by definition, the system $Y/Z$ is \ttfrac{$M$-SFT}{SFT}.

The system is not sofic since the the underlying space is not zero-dimensional, as there are non-trivial path-connected components obtained by preserving the position of the infinite path to be as in $z$, and moving the real number written on this configuration on the interval $[0,1]$. More precisely, enumerate the path giving the non-$0$ symbols in $z$ in order as $m_0 = 1_G, m_1 = s_1, m_2 = s_2s_1, \cdots$, and for real number $r \in [0,1]$ write its usual binary representation as $\textrm{bin}(r) \in \{0,1\}^\N$ (if there are two, pick either one), and associate to $r$ configuration $x_r$ defined by
\[ (x_r)_m = \left\{\begin{array}{ll}
0 & \mbox{ if } z_m = 0, \\
(s_{i+1}, \textrm{bin}(r)_i) & \mbox{ if } m = m_i.
\end{array}\right. \]
Then $r \mapsto x_r$ is a homeomorphism between the sphere $S^1 \cong [0,1]/(0 \equiv 1)$ and a closed subset of $Y/Z$. This concludes the proof that the second inclusion in the statement is proper.

For the properness of the first inclusion, that is, existence of proper sofic systems, we construct a sofic shift that overlays configuration of $X_{\leq 1}$ on top of every path similarly as above, but we modify $X$ to allow paths to end abruptly. More precisely, let $A = S \cup \{0\}$ where $0 \notin S$ and define $X' \subset A^M$ by
\[ x \in X' \iff (\forall m \in M: x_{m} = s \in S \implies x_{sm} \in S \cup \{0\}). \]
Now let $\Sigma = \{0\} \cup (S \times \{0,1\})$ and define $Y \subset \Sigma^M$ by requiring that the projection where the bits in the second component are ignored gives $X'$, and along every path marked in the first component, the sequence of bits in the second component is nonincreasing. Let $Z$ be the sofic shift obtained from $Y$ by subtracting the follower from every bit along each path (or zero, if the path ends), so that every path ends up having at most one $1$.

To see that $Z$ is not SFT, suppose for a contradiction that it is, and that there is an SFT rule with neighborhood $F$. Let $k \geq |F|+2$ and consider a configuration $z \in \Sigma^M$ with a marked (infinite simple) path $m_0 = 1_M, m_1, m_2, \dots$ in the first component, where on the bit layer there are $1$s at $m_0$ and at $m_k$. Since $z \notin Z$, by the assumption there exists $n \in M$ such that $nz|_F$ is a forbidden pattern. We have $nz_m = z_{mn}$ for all $m \in F$, and by the pigeonhole principle there exists $0 < i < k$ such that $\forall m \in F: m_i \neq mn$. Now define $z' \in \Sigma^M$, where $z'_{m'} = z_{m'}$ for all $m' \in M \setminus \{m_i\}$ and $z'_{m_i} = 0$. Clearly $z' \in Z$, but $nz'|_F = nz|_F$ is forbidden. This contradiction proves $Z$ is not SFT.

To obtain proper sofically presented systems, overlay a configuration of $X_{\leq 2}$ on top of every path of $X'$, and identify two points by quotienting by the relation $X_{\leq 1}^2 \cup \Delta_{X_{\leq 2}}$ on each path. As above, it is easy to see that the resulting sofically presented system is not expansive, thus not a finitely presented system.
\end{proof}

\subsection{Which groups satisfy the conclusion of Ma\~n\'e's theorem?}
\label{sec:ManeGroups}

In this section, we give a full characterization of finitely generated groups which satisfy the conclusion of Ma\~n\'e's theorem, i.e.\ we characterize groups that admit expansive actions on infinite-dimensional spaces. The construction is the same as in the previous section.

\begin{theorem}
\label{thm:ManeGroupsProof}
Let $G$ be a finitely generated group. The following are equivalent:
\begin{itemize}
\item $G$ is not virtually cyclic,
\item $G$ acts expansively on an infinite-dimensional compact metrizable space,
\item there is an infinite-dimensional \ttfrac{$G$-subshift}{SFT},
\item there is a \ttfrac{$G$-SFT}{SFT} which contains every compact metrizable space.
\end{itemize}
\end{theorem}

The following easy lemma reduces the virtually $\Z$ case to the $\Z$ case.

\begin{lemma}
Let $G \curvearrowright X$ be an expansive action and $H \leq G$ of finite index. Then the subaction $H \curvearrowright X$ is expansive.
\end{lemma}

In the non-virtually cyclic case, the construction is the same as in the proof of Theorem~\ref{thm:ProperSofic}.

As will be clear from the proof, the case that $G$ is not a torsion group is somewhat trivial. The main technical difficulty comes from the fact that in the case of torsion groups, we cannot find copies of $\N$ on the group by a local rule, and need to force such copies by SFT rules. For this we use the following theorem which may be of independent interest.

A \emph{ray} in a graph is a one-way infinite simple path, i.e. a sequence $x_1, x_2, x_3, \dots$ of vertices of the graph with no repetitions. Two paths are disjoint if the sets of vertices they visit are disjoint.

\begin{lemma}
\label{lem:Rays}
Let $G$ be a finitely generated group, and $\mathcal{G}$ a right Cayley graph of $G$ with respect to a finite symmetric generating set $A \subset G$. The following are equivalent.
\begin{itemize}
\item $G$ is virtually cyclic.
\item there does not exist an infinite set of disjoint rays in $\mathcal{G}$.
\end{itemize}
\end{lemma}

\begin{proof}
If $G$ is virtually cyclic, then it has at most linear growth rate, and it is easy to see that it cannot contain an infinite set of disjoint rays with respect to any set of generators.

If $G$ is virtually free but not virtually cyclic, then it does contain an infinite set of disjoint rays, since the free group $F_2$ obviously has this property for any generating set.

If $G$ not virtually free, then every Cayley graph of $G$ has a \emph{thick end}, meaning an infinite set of disjoint one-sided paths which define the same end of the graph. See Lemmas~2.1 and 2.3 in \cite{BaSt18} for a proof following Woess \cite{Wo89}.
\end{proof}

In Appendix~\ref{thm:RaysApp}, we give an alternative deduction in the one-ended case using a result of Halin, and reduce to the one-ended case using Stallings theorem.

\begin{proof}[Proof of Theorem~\ref{thm:ManeGroups} (Theorem \ref{thm:ManeGroupsProof})]
Suppose first that $G$ is virtually cyclic. Let $H \leq G$ be a copy of $\Z$ of finite index. Then $H$ still acts expansively by the previous lemma, and \cite{Ma79} shows $X$ is finite-dimensional. By Theorem~\ref{thm:ExpansivityCharacterizationProof} the second and third conditions are equivalent, and clearly if the third condition is false, then the fourth is as well.

Suppose then that $G$ is not virtually cyclic, and let $A$ be any symmetric generating set not containing $1_G$. It is enough to show that the fourth item holds.

Let $Y$ be the following SFT: The alphabet is $\Sigma = \{\#\} \cup (\{0, 1\} \times A)$. The local rules state that if $x_g = (b, a)$ where $b \in \{0,1\}, a \in A$, then $x_{ag} \neq \#$.

Let $Z$ be the following SFT: The alphabet is of course $\Sigma^2$, and the SFT rules state that if $(x, y) \in Z$ then for all $g$ either $(x, y)_g = (\#, \#)$ or $(x, y)_g = ((b, a), (c, a))$ for some $b, c \in \{0,1\}, a \in A$, and that when $(x, y)_g = ((b, a), (c, a))$, we have that the three bits readable are among the allowed patterns of the ``same number'' relation, i.e.
$
\left( \begin{smallmatrix}
b & b' & b'' \\
c & c' & c''
\end{smallmatrix} \right) \in A_{\mathrm{req}}$
where $x_g = (b, a)$, $x_{ag} = (b', a')$, $x_{a'ag} = (b'', a'')$, and $y_g = (c, a)$, $y_{ag} = (c', a')$, and $y_{a'ag} = (c'', a'')$.

It is straightforward to show that $Z$ is an SFT equivalence relation in $Y$ and $(x, y) \in Z$ if and only if the graph drawn by the second components of non-$\#$ symbols is the same in both, and on each pair of corresponding paths in $x$ and $y$ the same point of the binary solenoid is represented.

By Lemma~\ref{lem:Rays}, $G$ admits some disjoint infinite set of one-way infinite rays. Then we have a map from the Hilbert torus $(S^1)^\omega$ to $Y$ which writes some binary expansion of the point on the $i$th circle of $(S^1)^\omega$ on the $i$th ray of $Y$. This function is continuous and bijective into $Y/K$, so $Y/K$ contains an embedded copy of $(S^1)^\omega$. Since $(S^1)^\omega$ contains a Hilbert cube, we have that $Y/Z$ is a \ttfrac{$G$-SFT}{SFT} which contains every compact metrizable space.
\end{proof}

In the case of f.g. monoids, the same construction gives infinite-dimensional systems on monoids, as soon as one can find infinitely many disjoint rays in the left Cayley graph. We do not know the characterization of such monoids.

\section{Metrization of \ttfrac{subshift}{subshift} systems}
\label{sec:Metrization}

\subsection{Hyperbolic graph systems}
\label{sec:Hyperbolic}

We define a family of topological spaces obtained as limits of systems of (finite) graphs. We prove that limits of systems of graphs produce precisely the compact metrizable spaces. We prove that for every graph system, there is a ``hyperbolic'' graph system that defines the same topological space. When the graph system is hyperbolic, we construct an explicit metric.

In this section, it is useful to allow (pseudo)metrics to take on infinite distances, so a (pseudo)metric is a function $d : X \times X \to \R_{\geq 0} \cup \{\infty\}$ satisfying the usual formulas for a (pseudo)metric. Infinite distances only appear in auxiliary (pseudo)metrics discussed in the proofs, and the final metrics we obtain for our $\ttfrac{subshift}{subshift}$ applications never take on infinite values.

The idea of approximating topological spaces by using systems of simpler spaces is certainly not new (see for example the section ``The theory of inverse limit spaces'' in the book~\cite{Naga85}) but we have not seen previous work where the metric has been constructed explicitly, which is needed later for generalizing Ma\~n\'e's theorem. The ideas of this section are also clearly similar to those of Nekrashevych~\cite{Nek05}, though we do not know the exact connection.

\begin{definition}
Let $G_i = (V_i, E_i)$ be an undirected graph with self-loops at each vertex for all $i \in \N$, and $\phi_i : G_{i+1} \to G_i$ a graph homomorphism surjective on vertices (but not necessarily edges). Then $(G_i, \phi_i)_{i \in \N}$ is called a \emph{system of graphs}. Write $d_n$ for the distance function in $G_n$ and define $G = \underset{\leftarrow} \lim^* \; G_i = \{x \in (V_i)_{i \in \N} \;|\; \phi_i(x_{i+1}) = x_i\}$. If $f : \N \to \N$ is an increasing
function satisfying $f(n) \geq n$,
we say a system of graphs is $f$-\emph{hyperbolic} if
\[ \forall n: \forall x, y \in G: d_{f(n)}(x_{f(n)}, y_{f(n)}) \geq 2d_n(x_n, y_n) - 1. \]
If a system is $f$-hyperbolic for some $f$, we say it is \emph{prehyperbolic}.
If a system is $f$-hyperbolic for $f(n) = n + c$ for some constant $c$, we say it is \emph{virtually hyperbolic}. If it is $f$-hyperbolic for $f(n) = n + 1$, we say it is \emph{hyperbolic}.
\end{definition}

\begin{lemma}
A system of graphs $(G_i, \phi_i)_{i \in \N}$ is prehyperbolic if and only if
\[ \forall n: \exists m: \forall x, y \in G: d_m(x_m, y_m) \leq 2 \implies d_n(x_n, y_n) \leq 1 \]
if and only if
\[ \forall n: \exists m: \forall m' \geq m: \forall x, y \in G: d_{m'}(x_{m'}, y_{m'}) \leq 2 \implies d_n(x_n, y_n) \leq 1. \]
Any increasing function $f(n) \geq n$ giving such $m$ for each $n$ can be picked as the hyperbolicity function.
\end{lemma}

\begin{proof}
We first consider the first equivalence. The contrapositive of the above, where the implication is replaced by $d_n(x_n, y_n) > 1 \implies d_m(x_m, y_m) > 2$, is a special case of prehyperbolicity.

For the other direction, suppose the condition of the statement holds. Suppose $d_m(x_m, y_m) = k$. Take any path between $x_m, y_m$ of length $k$, say $(z_0, z_1, \dots, z_k)$ with $x_m = z_0$, $y_m = z_k$. Let $z'_i = \phi_n \circ \cdots \circ \phi_{m-1}(z_i)$. If $k = 2\ell + 1$, $(z'_0, z'_2, \dots, z'_{2 \ell}, z'_k)$ is a path of length $\ell + 1$ from $x_n$ to $y_n$ by the assumption $d_m(x_m, y_m) \leq 2 \implies d_n(x_n, y_n) \leq 1$, and if $k = 2 \ell$ then $(z'_0, z'_2, \dots, z'_{2 \ell})$ is a path of length $\ell$ from $x_n$ to $y_n$. In either case $d_n(x_n, y_n) \leq \lceil d_m(x_m, y_m)/2 \rceil$, which implies $d_m(x_m, y_m) \geq 2d_n(x_n, y_n) - 1$.

For the second equivalence, observe that since the $\phi_i$ are graph homomorphisms, distances cannot decrease when we move further in the graph sequence.
\end{proof}

In particular by the above lemma, hyperbolicity is equal to
\[ \forall n: d_{n+1}(x_{n+1}, y_{n+1}) \leq 2 \implies d_n(x_n, y_n) \leq 1. \]

Clearly for any prehyperbolic system of graphs there exists an increasing sequence $j_0, j_1, \dots$ such that the graphs $G_{j_i}$ form a hyperbolic system of graphs when we compose the morphisms between them, and this is called a \emph{telescoping sequence}. For a hyperbolic system, $\N$ is such a sequence, and for a virtually hyperbolic system $c\N$ is such a sequence for some constant $c \in \N$.

Write $G^*$ for the space $\underset{\leftarrow} \lim^* \; G_i$ with the product topology. This space is always zero-dimensional, and does not take the graph structure into account -- it restricts the possible choices of maps $\phi_i$, but if the graphs are all complete or discrete, this is nothing but the general construction of a compact zero-dimensional space as an inverse limit of finite sets. We now quotient out the graph neighbor relations of the graphs to define a natural quotient space $\underset{\leftarrow} \lim \; G_i$:

\begin{definition}
Let $(G_i, \phi_i)_{i \in \N}$ be a system of finite graphs. Let $K \subset (G^*)^2$ be the relation $(x, y) \in K \iff \forall n: (x_n, y_n) \in E_n$. We define
\[ \underset{\leftarrow} \lim \; G_i \cong G^* / K. \]
\end{definition}

It is obvious that $K$ is closed in $(G^*)^2$ when $G^*$ has the product topology, and thus $G^* / K$ is a compact metrizable space, by \ref{lem:QuotientMapsMetrizable}. The converse is true (and essentially classical), and we explain it in the following.

\begin{lemma}
\label{lem:ClosedPerClosedPrehyperbolic}
Suppose $\Sigma_i$ are finite alphabets, $C \subset \prod_{i \in \N} \Sigma_i$ is closed and $R \subset C^2$ a closed equivalence relation. For each $n$ define the graph $G_n = (V_n, E_n)$ where $V_n = \{u \in \prod_{i = 0}^{n - 1} \Sigma_i \;|\; [u] \cap C \neq \emptyset\}$, and
\[ (u, v) \in E_n \iff \exists x, y \in C: x \in [u] \wedge y \in [v] \wedge (x, y) \in R. \]
Define $\phi_i (ua) = u$ for all $ua \in V_{n+1}, a \in \Sigma_n$.
Then $(G_n, \phi_n)_{n \in \N}$ form a prehyperbolic system of graphs, and $C/R \cong \underset{\leftarrow} \lim \; G_i$.
\end{lemma}

\begin{proof}
We check prehyperbolicity. Suppose the system is not prehyperbolic, so that for some $|u| = |v| = n$ we have $d_n(u, v) = 2$ but there exists arbitrarily long words $|w| = |w'| = m$ such that $d_{n+m}(uw, vw') = 2$, i.e. there are points $x^m \in [uw], y^m \in [vw']$ and $z^m, \hat z^m \in C$ such that $z^m_{[0,n+m-1]} = \hat z^m_{[0,n+m-1]}$, $(x^n, z^n), (\hat z^m, y^m) \in R$. Extract a limit point of $(x^m,z^m,\hat z^m,y^m)$ to get a $4$-tuple $(x, z, \hat z, y)$ with $(x, z), (\hat z, y) \in R$. Necessarily $z = \hat z$, so by transitivity of $K$ we have $(x, y) \in R$, a contradiction.

Let $K\subset (G^*)^2$ be the relation of the previous definition. By definition $C$ and $G^*$ are equal as sets of vertices. By definition also $R\subset K$ and a compactness argument shows that $K\subset R$, so $R=K$, so $C/R=G^*/K \cong \underset{\leftarrow} \lim \; G_i$.
\end{proof}

\begin{lemma}
Let $X$ be a topological space. The following are equivalent:
\begin{itemize}
\item $X$ is compact and metrizable, and
\item $X$ is homeomorphic to $\underset{\leftarrow} \lim \; G_i$ for some hyperbolic system of finite graphs.
\end{itemize}
\end{lemma}

\begin{proof}
If $X$ is homeomorphic to $\underset{\leftarrow} \lim \; G_i$ for some hyperbolic system of finite graphs, then it is compact and metrizable as a quotient of a compact metrizable space by a closed equivalence relation.

Every compact metrizable space is a quotient of Cantor space by a closed equivalence relation, say $Y/Z$. Pick words in $Y$ as the nodes of the graphs, and pick the edges to be those pairs whose corresponding clopens have Cartesian product intersecting $Z$. By the previous lemma, this is necessarily prehyperbolic, and we can telescope it to be hyperbolic. By the same lemma, the graph system is homeomorphic to $Y/Z$.
\end{proof}

The reason for introducing hyperbolic graph systems is that we can construct a metric explicitly using the doubling condition. If $d$ is a metric on a set $X$, then the \emph{$\epsilon$-itinerary restriction} $d'$ of a metric $d$ is obtained by defining
\[ d'(x, y) = \inf \{ \sum_i d(x_i, x_{i+1}) \;|\; (x_0,x_1,\dots,x_k) \in X^{k+1}, x = x_0, y = x_k, \forall i: d(x_i, x_{i+1}) \leq \epsilon \}. \]

In what follows, let $(G_i, \phi_i)_{i \in \N}$ be a hyperbolic system of graphs, and let $G^* = \underset{\leftarrow} \lim^* \; G_i$ and $G = G^*/K$ be as above. For $x, y \in G$ write $x \sim_n y \iff (x_n, y_n) \in E_n$, and in this case say the points are \emph{$n$-neighbors}. Note that by the homomorphism property $x \sim_{n+1} y \implies x \sim_n y$. Define $x \sim y \iff \forall n: x \sim_n y$, and in this case say the points are \emph{$\omega$-neighbors}.

\begin{definition}\label{def:Itinerary}
If $x,x'\in G^*$, then any pair
\[I=((n_0, \dots, n_{k-1}), (x^0, x^1, \dots, x^k))\]
with
\[k\in\N, (n_0, \dots, n_{k-1}) \in \N^k, (x^0, x^1, \dots, x^k) \in G^{k+1}\mbox{ such that } x^0 = x \wedge x^k = x' \wedge \forall i: x^i \sim_{n_i} x^{i+1}\]
is called an \emph{itinerary} and $c(I) = \sum_{i = 0}^{k-1} 2^{-n_i + 1}$ is called its \emph{cost}. The $n_i$ are referred to as \emph{pivot depths}, and we refer to the steps of an itinerary as \emph{pivot moves}.
\end{definition}

\begin{theorem}
\label{thm:MetrizingHyperbolics} 
Let $(G_i, \phi_i)_{i \in \N}$ be a hyperbolic system of graphs. Then
\[d(x, y) = \inf \{ c(I) \;|\; I\mbox{ is an itinerary from }x\mbox{ to }y\}\]
is the unique maximal pseudometric $d$ on $G^*$ satisfying $d(x, y) \leq d_n(x_n, y_n) / 2^{n - 1}$
for all $n \in \N$. Its $2^{-n+1}$-itinerary restriction $d^n$ satisfies
\[ (d_n(x_n, y_n) - 1) / 2^n \leq d^n(x, y) \leq d_n(x_n, y_n) / 2^{n - 1} \]
for all $n$, and its metric identification metrizes the topological space $G$.
\end{theorem}

\begin{proof}
We first prove that the formulas imply that the metric identification metrizes $G$. First let us show that the distance between two points $x, y \in G^*$ is nonzero if and only if the points are not in the $K$-relation. For this, consider $x,y \in G^*$ with $d(x, y) = 0$. In particular there must exist itineraries where all individual pivot moves are arbitrarily cheap (i.e.\ pivot depths are arbitrarily large), meaning $d(x,y) = d^n(x,y)$ for all $n$. Then $(d_n(x_n, y_n) - 1) / 2^n \leq d^n(x, y) = d(x,y) = 0$, so we must have $d_n(x_n, y_n) \leq 1$ for all $n$. Conversely, if $x_n \sim_n y_n$ for all $n$, then there are itineraries from $x$ to $y$ with a single pivot move of arbitrarily low cost.

To see that the topologies are also the same, note that the metric identification by $d$ is Hausdorff and $G$ is compact, so it is enough to show that the map taking $K$-equivalence classes to $d$-equivalence classes is continuous. The space $G$ is sequential (we already know it's metrizable by Lemma~\ref{lem:QuotientMapsMetrizable}) so it is enough to show sequential continuity. If $x^i \rightarrow x$ in $G$ then consider the $K$-closure $C$ of the set $\bigcup_{a \in V_n \\ (a, x_n) \notin E_n} [a]$. This corresponds to a closed set in $G$, and by the normality of $G$, we find an open set $U \subset G$ such that $x \in U$ and $U \cap C = \emptyset$. It follows that $x^i \in U$ for large enough $i$, meaning $x^i_n \sim_n x_n$ for large enough $i$. Clearly it follows that $d(x^i, x) \rightarrow 0$.

To see that $d$ is the maximal pseudometric $d'$ satisfying
\[ d'(x, y) \leq d_n(x_n, y_n) / 2^{n - 1}, \]
observe that $d$ satisfies this by using itineraries with only depth-$n$ pivot moves, and $d' \leq d$ for any $d'$ satisfying it by the triangle inequality.

Even if $d_n(x, y) = 0$ and $x \neq y$, we will have to pay $2^{-n+1}$ to move from $x$ to $y$ with a depth-$n$ pivot move (recall that we have self-loops in all nodes of each $G_n$, so we can always perform such a pivot move).

Note that by some simple limit arguments, $d^n(x, y)$ is precisely the infimum of costs of itineraries $x$ to $y$ satisfying $n_i \geq n$ for all~$i$. The upper bound $d^n(x, y) \leq d_n(x_n, y_n) / 2^{n - 1}$ then comes directly from the itinerary corresponding to a path of length $d_n(x_n, y_n)$. It now suffices to prove the lower bound.

Instead of
\[ d_n(x_n, y_n) / 2^n - 2^{-n} \leq d^n(x, y) \]
we show the following stronger technical claim (w.l.o.g. for $n = 0$):
suppose $d_0(x, x') \geq \ell \geq 1$ and $I$ is an itinerary from $x$ to $x'$. Then $c(I) \geq c_\alpha(I)/2 + \ell - 1 + c_\beta(I)/2$ where $c_\alpha, c_\beta$ are the costs of the first and last edges.

We prove this by induction, simultaneously for all pairs of points in all hyperbolic graph systems, first on length (as a tuple, i.e. the number $k+1$ in Definition~\ref{def:Itinerary}) then on depth (maximum pivot level $\max n_i$). Suppose $I$ has length $k = 2$ (minimal possible). Since $\ell - 1 = 0$, we just need to show its cost is at least half the sum of the costs of its first and last edges. Since $k = 2$, $I$ is a single edge and we have $c(I) = c_\alpha(I)/2 + c_\beta(I)/2$ as desired.

Now suppose $I$ has length $k > 2$. First suppose it contains a move with $n_i = 0$ (i.e. a pivot move of cost $2$). Suppose first $i \in [1, k-2]$ so it is not the first or last move. Split $I$ into $I_1 = (x^0, \dots, x^i)$ and $I_2 = (x^{i+1},\dots,x^k)$ and suppose $d_0(x^0, x^i) = m$ so $d_0(x^{i+1}, x^k) \geq \ell - m - 1$. Then by induction on length
\[ c(I_1) \geq c_{\alpha}(I)/2 + m - 1 + c_{\beta}(I_1)/2, \;\; c(I_2) \geq c_{\alpha}(I_2)/2 + \ell - m - 1 - 1 + c_\beta(I)/2 \]
so
\begin{align*}
c(I) &= c(I_1) + 2 + c(I_2) \\
& \geq c_{\alpha}(I)/2 + m - 1 + c_{\beta}(I_1)/2 + 2 + c_{\alpha}(I_2)/2 + \ell - m - 1 - 1 + c_\beta(I)/2 \\
& \geq c_{\alpha}(I)/2 + \ell - 1 + c_\beta(I)/2
\end{align*}

If $i = 0$, then let $I' = (x^1,\dots,x^k)$. We have $d_0(x^1, x^k) \geq \ell - 1$ so
\[ c(I') \geq c_\alpha(I')/2 + \ell - 2 + c_\beta(I')/2. \]
Then since $c_\alpha(I) = 2$ we have
\[ c(I) \geq 2 + c_\alpha(I')/2 + \ell - 2 + c_\beta(I')/2 \geq \ell + c_\beta(I)/2 = c_\alpha(I)/2 + \ell - 1 + c_\beta(I)/2. \]
The case where $i = k - 1$ is symmetric.

Now, suppose no moves with cost $2$ are used. Then $x^i \sim_1 x^{i+1}$ for all $i$ meaning $d_1(x^0, x^k) \geq 2\ell - 1$ by hyperbolicity. Observe that if we drop the graph $G_0$ we obtain another system of graphs which is obviously still hyperbolic, and the effect on costs of itineraries is that some itineraries are lost, and the remaining ones have their costs doubled. Let $y^i = \sigma(x^i)$ for each $i$ (that is, drop the $G_0$-component of each $x^i$) and $y = \sigma(x), y' = \sigma(x')$. Then $y^i$ form an itinerary $I'$ from $y$ to $y'$ with pivot sequence $(n_0-1,n_1-1,\dots,n_{k-1}-1)$, in this shifted hyperbolic graph system.

The maximal pivot depth $\max (n_i - 1)$ is one smaller, and the length of the itinerary is the same, so we can apply induction. We compute
\[ c(I') \geq c_{\alpha}(I')/2 + 2\ell - 1 - 1 + c_{\beta}(I')/2. \]
Since $c_{\alpha}(I') = 2c_{\alpha}(I)$ and $c_{\beta}(I') = 2c_{\beta}(I)$ and in general all costs are twice cheaper in $I$ than in $I'$. It follows that
\begin{align*}
c(I) &\geq (c_{\alpha}(I')/2 + 2\ell - 2 + c_{\beta}(I')/2)/2 \\
& = (c_{\alpha}(I')/4 + \ell - 1 + c_{\beta}(I')/4 \\
&= c_{\alpha}(I)/2 + \ell - 1 + c_{\beta}(I)/2
\end{align*}
as required.
\end{proof}

\begin{example}
The interval $[0,1]$ can be represented by repeated subdivision of paths. Take $G_i = P_{2^i}$, the path graph with $2^i$ nodes (with implicit self-loops), labeled in order along the path with the words $(0^i, 0^{i-1}1, \dots, 1^i)$. Let $\phi_i(w0) = \phi(w1) = w$. The following picture shows the first four levels $i = 0,1,2,3$. The vertical edges are edges of the graphs, and the horizontal edges (with arrow tips) depict the homomorphisms.
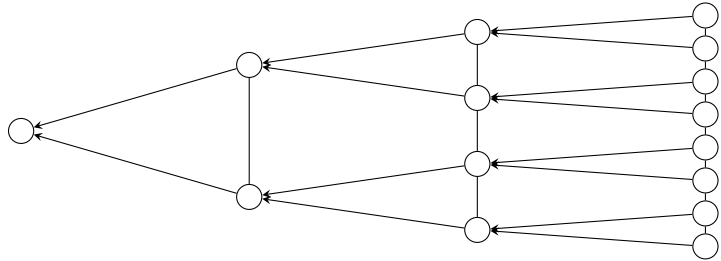
\begin{center}
\begin{tikzpicture}[>=stealth]
\node[draw,circle] (w) at (0,0.0) {};
\node[draw,circle] (w0) at (3,-0.875) {};
\node[draw,circle] (w1) at (3,0.875) {};
\node[draw,circle] (w00) at (6,-1.3125) {};
\node[draw,circle] (w01) at (6,-0.4375) {};
\node[draw,circle] (w10) at (6,0.4375) {};
\node[draw,circle] (w11) at (6,1.3125) {};
\node[draw,circle] (w000) at (9,-1.53125) {};
\node[draw,circle] (w001) at (9,-1.09375) {};
\node[draw,circle] (w010) at (9,-0.65625) {};
\node[draw,circle] (w011) at (9,-0.21875) {};
\node[draw,circle] (w100) at (9,0.21875) {};
\node[draw,circle] (w101) at (9,0.65625) {};
\node[draw,circle] (w110) at (9,1.09375) {};
\node[draw,circle] (w111) at (9,1.53125) {};
\draw (w0) -- (w1);
\draw (w00) -- (w01);
\draw (w01) -- (w10);
\draw (w10) -- (w11);
\draw (w000) -- (w001);
\draw (w001) -- (w010);
\draw (w010) -- (w011);
\draw (w011) -- (w100);
\draw (w100) -- (w101);
\draw (w101) -- (w110);
\draw (w110) -- (w111);
\draw[->] (w0) -- (w);
\draw[->] (w1) -- (w);
\draw[->] (w00) -- (w0);
\draw[->] (w01) -- (w0);
\draw[->] (w10) -- (w1);
\draw[->] (w11) -- (w1);
\draw[->] (w000) -- (w00);
\draw[->] (w001) -- (w00);
\draw[->] (w010) -- (w01);
\draw[->] (w011) -- (w01);
\draw[->] (w100) -- (w10);
\draw[->] (w101) -- (w10);
\draw[->] (w110) -- (w11);
\draw[->] (w111) -- (w11);
\end{tikzpicture}

\end{center}
Points of $\{0,1\}^\N$ are in natural correspondence with rays in this graph. The distance from the path $x$ corresponding to $0^\infty$ to the path $y$ corresponding to $1^\infty$, in the level-$i$ graph is $2^i-1$, and thus the cost on level $i = 0$ is $2$, and on all other levels the cost is $(2^i-1)/2^{i-1}=2-1/2^{i-1}$. This gives the defining upper bounds on the distance in the metric constructed in the theorem, and in particular we obtain that $d(x,y) \leq 1$.

Looking at this diagram only, the distance indeed seems to be $1$, as the optimal itinerary that uses no steps on levels $>3$ indeed only takes a single step on level $1$. Indeed, it is not hard to show that the metric in the graph system is precisely the Euclidean distance, interpreting $x$ and $y$ as binary representations of real numbers. The theorem does not allow us to deduce that $1$ is optimal from this prefix of the graph system: in fact even if we decide to only take steps of costs $1/4$ or less, i.e.\ consider the restriction $d^3$ (whose distances are at least as large as those of $d$), the lower bound obtained for the distance from $x$ to $y$ is $3/4$, since the formula gives
\[ d^3(x, y) \geq d_3(x_3, y_3) - 1) / 2^3 = 3/4. \]

We could not hope to have a better lower bound, without knowing exactly how the graph process continues: Consider the infinite periodic continuation beginning as follows:
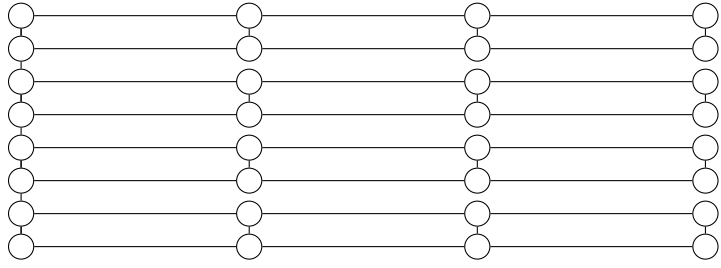
\begin{center}
\begin{tikzpicture}[>=stealth]
\node[draw,circle] (w0w000) at (0,-1.53125) {};
\node[draw,circle] (w0w001) at (0,-1.09375) {};
\node[draw,circle] (w0w010) at (0,-0.65625) {};
\node[draw,circle] (w0w011) at (0,-0.21875) {};
\node[draw,circle] (w0w100) at (0,0.21875) {};
\node[draw,circle] (w0w101) at (0,0.65625) {};
\node[draw,circle] (w0w110) at (0,1.09375) {};
\node[draw,circle] (w0w111) at (0,1.53125) {};
\node[draw,circle] (w1w000) at (3,-1.53125) {};
\node[draw,circle] (w1w001) at (3,-1.09375) {};
\node[draw,circle] (w1w010) at (3,-0.65625) {};
\node[draw,circle] (w1w011) at (3,-0.21875) {};
\node[draw,circle] (w1w100) at (3,0.21875) {};
\node[draw,circle] (w1w101) at (3,0.65625) {};
\node[draw,circle] (w1w110) at (3,1.09375) {};
\node[draw,circle] (w1w111) at (3,1.53125) {};
\node[draw,circle] (w2w000) at (6,-1.53125) {};
\node[draw,circle] (w2w001) at (6,-1.09375) {};
\node[draw,circle] (w2w010) at (6,-0.65625) {};
\node[draw,circle] (w2w011) at (6,-0.21875) {};
\node[draw,circle] (w2w100) at (6,0.21875) {};
\node[draw,circle] (w2w101) at (6,0.65625) {};
\node[draw,circle] (w2w110) at (6,1.09375) {};
\node[draw,circle] (w2w111) at (6,1.53125) {};
\node[draw,circle] (w3w000) at (9,-1.53125) {};
\node[draw,circle] (w3w001) at (9,-1.09375) {};
\node[draw,circle] (w3w010) at (9,-0.65625) {};
\node[draw,circle] (w3w011) at (9,-0.21875) {};
\node[draw,circle] (w3w100) at (9,0.21875) {};
\node[draw,circle] (w3w101) at (9,0.65625) {};
\node[draw,circle] (w3w110) at (9,1.09375) {};
\node[draw,circle] (w3w111) at (9,1.53125) {};
\draw (w0w000) -- (w0w001);
\draw (w0w010) -- (w0w011);
\draw (w0w100) -- (w0w101);
\draw (w0w110) -- (w0w111);
\draw (w1w000) -- (w1w001);
\draw (w1w010) -- (w1w011);
\draw (w1w100) -- (w1w101);
\draw (w1w110) -- (w1w111);
\draw (w2w000) -- (w2w001);
\draw (w2w010) -- (w2w011);
\draw (w2w100) -- (w2w101);
\draw (w2w110) -- (w2w111);
\draw (w3w000) -- (w3w001);
\draw (w3w010) -- (w3w011);
\draw (w3w100) -- (w3w101);
\draw (w3w110) -- (w3w111);
\draw (w1w000) -- (w0w000);
\draw (w1w001) -- (w0w001);
\draw (w1w010) -- (w0w010);
\draw (w1w011) -- (w0w011);
\draw (w1w100) -- (w0w100);
\draw (w1w101) -- (w0w101);
\draw (w1w110) -- (w0w110);
\draw (w1w111) -- (w0w111);
\draw (w2w000) -- (w1w000);
\draw (w2w001) -- (w1w001);
\draw (w2w010) -- (w1w010);
\draw (w2w011) -- (w1w011);
\draw (w2w100) -- (w1w100);
\draw (w2w101) -- (w1w101);
\draw (w2w110) -- (w1w110);
\draw (w2w111) -- (w1w111);
\draw (w3w000) -- (w2w000);
\draw (w3w001) -- (w2w001);
\draw (w3w010) -- (w2w010);
\draw (w3w011) -- (w2w011);
\draw (w3w100) -- (w2w100);
\draw (w3w101) -- (w2w101);
\draw (w3w110) -- (w2w110);
\draw (w3w111) -- (w2w111);
\draw (w0w000) -- (w0w001);
\draw (w0w001) -- (w0w010);
\draw (w0w010) -- (w0w011);
\draw (w0w011) -- (w0w100);
\draw (w0w100) -- (w0w101);
\draw (w0w101) -- (w0w110);
\draw (w0w110) -- (w0w111);
\end{tikzpicture}

\end{center}
In other words, suppose $G'_0 = G_0, G_1' = G_1, G_2' = G_2, G_3' = G_3$ are as before, but $G_n' = 4P_2$ for larger $n$, and the homomorphisms go as described in the diagram. This system is trivially hyperbolic, but the space is no longer the interval, but rather the discrete space on four points. Let us label the new nodes by repeating the last bit in the node at level $3$, and associate infinite words to infinite rays as before. In this new graph system, the distance from any path extending $x_{[0,3]}$ to any path extending $y_{[0,3]}$ is precisely $3/4$, realized by itineraries of the form
\[ ((m, 3, m, 3, m, 3, m), (x, 001^\infty, 010^\infty, 01^\infty, 10^\infty, 101^\infty, 110^\infty, y) \]
where $m \geq 3$ is arbitrary, corresponding to cost $4\cdot 2^{-m+1}+3\cdot 2^{-3+1}=3/4+8/2^m$. \qee
\end{example}

In the previous example, the lower and upper bounds given by Theorem~\ref{thm:MetrizingHyperbolics} converge (to the Euclidean distance) for any two points. A hyperbolic graph system and points in it can be given to an algorithm in the form of an oracle for the graphs $G_i$, homomorphisms $\phi_i$ and the graph elements $x_i \in G_i$ identifying the point. By definition, the metric is upper semicomputable in this data, i.e.\ one can obtain a converging sequence of upper approximations by inspecting larger and larger prefixes of the graph system, and in the previous example, one also obtains.

However, unlike in the previous example, it is difficult to tell in general how fast this convergence is, and indeed, we do not at present know whether the distance between given points is computable (uniformly in the points and the system), nor whether the information obtained from the formulas in Theorem~\ref{thm:MetrizingHyperbolics} always converges to the correct distance. We certainly believe this is the case for the sofic-per-sofic systems (and more generally automatic spaces) studied in the next section. However, an easy calculation shows that it at least gives a $5/2$-approximation, in the following sense:

\begin{theorem}
\label{thm:Computability}
Given a hyperbolic system of graphs $G^* = (G_i, \phi_i)_{i \in \N}$ and two points $x, y \in G^*$, one can effectively enumerate
\begin{itemize}
\item an increasing sequence $(r_m)_m$ of rationals such that $d(x, y) \in [r_m, \frac52 r_m + 2^{-m}]$ for all $j$, and
\item a decreasing sequence $(s_m)_m$ of rationals converging to $d(x, y)$.
\end{itemize}
\end{theorem}

\begin{proof}
The existence of the sequence of rational numbers converging to $d(x,y)$ from above is simply upper semicomputability. In what follows, we prove the first item.

Define an \emph{$m$-truncated itinerary} as a tuple
\[I=((n_0, \dots, n_{k-1}), (x^0, x^1, \dots, x^k))\]
with $n_i \leq n$ for all $i$, such that $x^0 = x$, $x^k = x'$ and
\[ \forall i \in [0,k-1]: (n_i < m \implies x^i \sim_{n_i} x^{i+1}) \wedge \forall i \in [0,k-2]: \min(n_i, n_{i+1}) < m \]
The \emph{lower cost} of an $m$-truncated itinerary is $c(I) = \sum_{i = 0}^{k-1} c(x^i, x^{i+1}, n_i)$, where
\[ c(x_i, x_{i+1}, n_i) = \left\{\begin{array}{ll}
2^{-n_i + 1}, & \mbox{if } n_i < m \\
d_m(x^i_m, y^{i+1}_m) / 2^m - 2^{-m}, & \mbox{otherwise.}
\end{array}\right.\]

The differences between an truncated itinerary and an actual itinerary are that an itinerary can take steps of arbitrarily large depths, while truncated itineraries cannot take steps beyond depth $m$ and cannot take two consecutive steps at that depth. On the other hand, in a truncated itinerary, there is no requirement for the moves on the level $m$ -- they may teleport between any two points.

Recall the inequalities
\[ d_n(x_n, y_n) / 2^n - 2^{-n} \leq d^n(x, y) \leq d_n(x_n, y_n)/2^{-n+1} \]
proved in \ref{thm:MetrizingHyperbolics}. Let $I = ((n_0, \dots, n_{k-1}), (x^0, x^1, \dots, x^k))$ be an arbitrary itinerary from $x$ to $y$. Replace all maximal subitineraries $((n_j, \dots, n_\ell), (x^j, \dots, x^{\ell+1}))$, where $n_i \geq m$ for all $i \in [j, \ell]$, by simply $((m), (x^j, x^{\ell+1}))$. What results is an $m$-truncated itinerary $I'$, and the first inequality implies that $c(I') \leq c(I)$

On the other hand, consider an arbitrary $m$-truncated itinerary $I'$ from $x$ to $y$. We can replace \emph{missing steps} $((m), (x_j, x_{j+1}))$ by pivot moves to obtain an itinerary from $x$ to $y$. The cost $d_m(x^i_m, y^{i+1}_m) / 2^m - 2^{-m}$ in the truncated itinerary is replaced by $d_m(x^i_m, y^{i+1}_m)/2^{m-1}$ in the actual itinerary. Thus, for each replaced missing step we add a cost of $2^{-m}$ and then (at most) multiply the entire cost by $2$ (the costs of steps on levels $p < m$ are not modified).

We get that for all $m$-truncated itineraries $I'$ there exists an itinerary $I$ such that $c(I) \leq 2c(I') + t2^{-m+1}$, where $t$ is the number of missing steps. Missing steps are separated by at least $t-1$ steps on levels $p < m$ each of whose cost is at least $2^{-m+2}$, so $c(I') \geq (t-1) 2^{-m+2}$, from which we get
\[ c(I')/2 + 2^{-m+1} \geq t 2^{-m+1}. \]
Substituting in $c(I) \leq 2c(I') + t2^{-m+1}$, we have $c(I) \leq \frac52c(I') + 2^{-m+1}$.

Now define
\[ r_m = \inf \{c(I') \;|\; I' \mbox{ is an $m$-truncated itinerary from $x$ to $y$}\}. \]
We have $r_j \leq c(I') \leq c(I)$ for any itinerary $I$ from $x$ to $y$ and $I'$ its $m$-truncation, so $r_j \leq \inf c(I) = d(x, y)$. On the other hand, for any truncated itinerary $I'$ there exists an itinerary $I$ such that $c(I) \leq \frac52 c(I') + 2^{-m+1}$, which implies $d(x, y) \leq \frac52 r_m + 2^{-m+1}$.

The $r_m$ now satisfy $d(x, y) \in [r_m, \frac52 r_m + 2^{-m+1}]$. Once we have such a sequence, we can make it nondecreasing by replacing $r_m$ with $\max_{i < m} r_i$. We can then replace $2^{-m+1}$ by $2^{-m}$ and can make the sequence strictly increasing, by offsetting the sequence and lowering the values artificially.
\end{proof}

\subsection{Application to subshifts}

It is basic topology that a \ttfrac{subshift}{subshift} dynamical system is metrizable as a topological space. As we want to obtain computable metrics in \ttfrac{subshift}{subshift} systems we understand well, we give an explicit construction of such a metric. In the \ttfrac{sofic}{sofic} case, we obtain an upper semicomputable metric that is given by a similar formula as the usual Cantor metric. For the terminology used in this section, see Section \ref{sec:Hyperbolic}.

\begin{definition}
\label{def:ShiftGraph}
Let $Y$ be a subshift and $Z \subset Y^2$ a subshift equivalence relation. If $u, u' \in \B_{2n+1}(Y)$, write $u \sim_n u'$ if $([u] \times [u']) \cap Z \neq \emptyset$, and define ${\sim} = \bigcup_n {\sim_n}$. Define $\mathcal{G}_{(Y,Z)}$ as the system of graphs $G_i$ where $V(G_n) = \B_{2n+1}(Y)$, $E(G_n) = \{(u, v) \;|\; u \sim v\}$, write $d_n$ for the graph ``metric'' of $G_n$ ($d_n(v,v)=1$ where the interpretation is that the shortest route from a vertex to itself is via a self-loop) and define the projections $\pi_n : G_{n+1} \to G_n$ by $\pi_n(u) = u_{[1,2n-1]}$.
\end{definition}

\begin{lemma}
\label{lem:SubshiftPerSubshiftPrehyperbolic}
The graph system $\mathcal{G}_{(Y,Z)}$ is $f$-hyperbolic for some $f$ and $\mathcal{G}_{(Y,Z)}$ is pseudometrized by the unique pseudometric $d$ such that
\[ \forall n: d(x, y) \leq d_{f(n)}(x_{[-f(n),f(n)]}, y_{[-f(n),f(n)]}) / 2^{n - 1} .\]
The induced metric topologizes $Y/Z$.
\end{lemma}

\begin{proof}
This is a direct application of Lemma~\ref{lem:ClosedPerClosedPrehyperbolic} and Theorem~\ref{thm:MetrizingHyperbolics} (after telescoping).
\end{proof}

Let $\ccfrac{Y}{Z}$ be \ttfrac{subshift}{subshift} and let $\mathcal{G}_{(Y,Z)}$ and the $G_n$ be as above.

\begin{theorem} 
\label{thm:SoficSoficHyperbolic}
If $\ccfrac{Y}{Z}$ is \ttfrac{sofic}{sofic}, then $\mathcal{G}_{(Y,Z)}$ is virtually hyperbolic, and for some $c \in \N$, 
\[ \forall n: d(x, y) \leq d_{cn}(x_{[-cn,cn]}, y_{[-cn,cn]}) / 2^{n - 1} \]
defines a unique maximal metric $d$ on $Y/Z$, and given two points in $Y$, one can effectively compute $\frac52$-approximations to $d(x, y)$ in the sense of Theorem~\ref{thm:Computability}.
\end{theorem}

\begin{proof}
Since $Y$ is sofic, we may assume $Z \subset Y^2$, and $Y$ is fully determined by $Z$. Let $q_Z$ be the number of states for a finite-state automaton accepting the language $Z$ (i.e. the number of vertices in an edge-labeled edge-shift cover). Let $m = q_Z^2+1$.

We use the notation of Definition \ref{def:ShiftGraph} and argue similarly as in Lemma~\ref{lem:ClosedPerClosedPrehyperbolic}. Suppose that $u, v \in \B_{2n+1}(Y)$ satisfy $d_n(u, v) = 2$ and suppose there exist words $s,t,s',t'$ with $|s| = |t| = |s'| = |t|' = m$ such that
$d_{n+m}(sut, s'vt') = 2$.
Then there exists a word $w \in \B_{2n+2m+1}(Y)$ such that $(sut, w), (w, s'vt') \in \B_{2(n+m) + 1}(Z)$. Since $m > q_Z^2$, by the pigeonhole principle there exist $\ell_1 < \ell_2, r_1 < r_2$ such that
\[
({^\infty}(s_{[\ell_1, \ell_2)}) s_{[\ell_2, |s|)} u t_{[0,r_1]} (t_{(r_1, r_2]})^\infty,
 {^\infty}(w_{[\ell_1, \ell_2)}) w_{[\ell_2, 2n+m+1+r_1)} (w_{(2n+m+1+r_1, 2n+m+1+r_2]})^\infty) \in Z, \]
\[
({^\infty}(w_{[\ell_1, \ell_2)}) w_{[\ell_2, 2n+m+1+r_1)} (w_{(2n+m+1+r_1, 2n+m+1+r_2]})^\infty,
{^\infty}(s'_{[\ell_1, \ell_2)}) s'_{[\ell_2, |s|)} v t'_{[0,r_1]} (t'_{(r_1, r_2]})^\infty) \in Z
\]
which by transitivity of $Z$ implies $[u] \cap [v] \neq \emptyset$, a contradiction with $d_n(u, v) = 2$. Since $m$ does not depend on $n$, this shows virtual hyperbolicity. The metric is then given by the previous lemma. The computability statement follows from Theorem~\ref{thm:Computability}, because we can clearly compute the graphs effectively from a finite description of the sofic shift.
\end{proof}

We give an example of a family of \ttfrac{subshift}{sofic} examples where the function describing prehyperbolicity grows arbitrarily slowly, and
\[ \forall n: d(x, y) \leq d_{cn}(x_{[-cn,cn]}, y_{[-cn,cn]}) / 2^{n - 1} \]
fails to give a metric for any $c$ (and even for any fixed function in place of $n \mapsto cn$).

\begin{example}
For every $f : \N \to \N$ there exists a \ttfrac{subshift}{sofic} $\ccfrac{Y}{Z}$ such that
\[ \forall n: d(x, y) \leq d_{f(n)}(x_{[-f(n),f(n)]}, y_{[-f(n),f(n)]}) / 2^{n - 1} \]
does not define a metric on $\mathcal{G}_{(Y,Z)}$.

For this, fix a sequence $(n_i)_{i \in \N}$ and let $Y$ be the subshift generated by the points
\[ \dots\#\#\# x_0 2^{n_0} x_1 2^{n_1} x_2 2^{n_2} x_3 2^{n_3} x_4 2^{n_4} \dots  \]
where $x_i \in \{0,1\}$ for all $i$. The sofic relation $Z$ verifies that the two points have their $\#$s and $2$s in the same positions, and that if $\#0$ or $\#1$ occurs in either of them (thus both), then the bit sequences represent the same number, i.e.\ the automaton verifies $A_{\mathrm{req}}$ (and that the first bits are either the same or one sequence starts with $100$ and the other with $011$), along the subsequence of bits.

Now, if $n_i$ grows much faster than $f(i)$, then it can be shown that the $d$-distance between
\[ \dots\#\#\# 0 2^{n_0} 0 2^{n_1} 0 2^{n_2} 0 2^{n_3} 0 2^{n_4} \dots  \]
and
\[ \dots\#\#\# 1 2^{n_0} 0 2^{n_1} 0 2^{n_2} 0 2^{n_3} 0 2^{n_4} \dots  \]
is zero, but they are not in the relation $Z$. \qee
\end{example}

While the above gives systems which are difficult to metrize, and thus the proof of Theorem~\ref{thm:TopDimFinite} does not give finite dimension for them, they in fact all have Lebesque covering dimension one. We do not have examples where dimension is infinite, see Question~\ref{q:SubshiftSoficFinDim}.

\section{Topological dimension of sofically presented systems}
\label{sec:SoficSoficMane}

The theorem of Ma\~n\'e says that every \ttfrac{subshift}{SFT} system has finite topological dimension. We show that the topological dimension of \ttfrac{sofic}{sofic} systems is also finite. The finite dimension comes from proving finite Hausdorff dimension for a suitably constructed metric, as in \cite{Fa89}.

\begin{theorem}
\label{thm:TopDimFiniteProof}
Suppose $M \in \{\N, \Z\}$ and $X$ is \ttfrac{$M$-sofic}{sofic}. Then the topological dimension of $X$ is finite.
\end{theorem}

\begin{proof}
We show the result for $\Z$-systems, but the proof for $\N$-systems is the same (though we use prefixes instead of centered words).

It is enough to show that there exists a metric such that $X$ has finite Hausdorff dimension with respect to it, as the topological dimension is always smaller than the Hausdorff dimension with respect to any metric inducing the topology, see p. 107 in \cite{HuWa48}. It is well-known that the Hausdorff dimension with respect to a metric $d$ is bounded from above by the \emph{(upper) box dimension}
\[ \bar C_d(X) = \limsup_{\epsilon \rightarrow 0} \frac{\log N(X, \epsilon)}{ - \log \epsilon } \]
where $N(X,\epsilon)$ is the number of balls of radius $\epsilon$ needed to cover $X$. Thus it is enough to show that with respect to some metric, the upper box dimension is finite.

Let $d$ be the metric given by Theorem~\ref{thm:SoficSoficHyperbolic}, i.e. for some $c \in \N$ we have
\[ d(x, y) \leq d_{cn}(x_{[-cn,cn]}, y_{[-cn,cn]}) / 2^{n - 1} \]
for all $n, x, y$ with $d_{cn}$ as in Definition \ref{def:ShiftGraph}. Then at most $|B_{2cn + 1}(Y)|$ balls of radius $1/2^{n - 1}$, i.e.
$|B_{2c(n + 1) + 1}(Y)|$ balls of radius $1/2^n$, are needed to cover the space, so
\begin{align*}
\bar C_d(X) &= \limsup_{\epsilon \rightarrow 0} \frac{\log N(X, \epsilon)}{ - \log \epsilon } \\
&= \limsup_{i \rightarrow \infty, a \in [0, 2^{-i})} \frac{\log N(X, 2^{-i} + a)}{ - \log (2^{-i} + a) } \\
&\leq \limsup_{i \rightarrow \infty} \frac{\log N(X, 2^{-i})}{ - \log (2^{-i+1}) } \\
&\leq \limsup_{i \rightarrow \infty} \frac{\log |B_{2c(i + 1) + 1}(Y)|}{ (i-1) \log 2 } \\
&= \limsup_{i \rightarrow \infty} \frac{\log (\exp(h(Y)(2c(i + 1) + 1) + \alpha(i) i))}{ (i-1) \log 2 } \mbox{ where $\alpha(i) \rightarrow 0$} \\
&= \limsup_{i \rightarrow \infty} \frac{(h(Y)(2c(i + 1) + 1) + \alpha(i) i)}{ (i-1) \log 2 } \mbox{ where $\alpha(i) \rightarrow 0$} \\
&\leq \frac{h(Y)2c}{\log 2}.
\end{align*}
\end{proof}

\section{Automata theory and morphisms}
\label{sec:AutomataTheory}

We outline some decidability results on $\N$ and $\Z$. See \cite{PePi04} for the basics of automata and logic on infinite words, and \cite{HoMoUl06} for automata theory on finite words. The decidability results here are meant as a counterpart to the undecidability results in Section~\ref{sec:ConjugacyUndecidability}, and thus we focus on quick and conceptually simple proofs of decidability; in many cases we can state on abstract grounds that a problem is decidable, but the present discussion does not provide any reasonable (e.g.\ polynomial time) algorithm.

\subsection{Deciding things about automata}

Since often our subshift relations arise from not a priori shift-invariant or closed sets, for stating decidability results it is useful to step into a more general framework than sofic shifts which has similar decidability properties, but does not have such restrictions.

We call $\omega$-automatic sets \emph{$\N$-automatic sets}.

By $MF_2(\Z, 0, \prec)$, we denote the monadic second-order logic of $\Z$, with the predicate $\prec$ on $\Z$ by $m \prec n$ iff $n = m+1$, and $0$ the constant $0 \in \Z$.

\begin{lemma}
\label{lem:MF2Decidable}
The logics $MF_2(\N, \prec)$ and $MF_2(\Z, 0, \prec)$ are decidable.
\end{lemma}

\begin{proof}
The case $MF_2(\N, \prec)$ is proved in \cite{PePi04}. By \cite{PePi04}, the logic $MF_2(\Z, \prec)$ is decidable. Given $\phi \in MF_2(\Z, 0, \prec)$, assume $\phi$ does not contain the variable $o$ and construct the following statement $\psi \in MF_2(\Z, \prec)$
\begin{align*}
\psi = \forall o: \phi[0 / o]
\end{align*}
where $\phi[0/o]$ is equal to $\phi$, except every occurrence of $0$ is replaced by $o$. If $\psi$ is satisfiable, it is satisfiable for the choice $o = 0$, thus $\phi$ is satisfiable by definition. If $\phi$ is satisfiable, then for any choice of $o$ we can conjugate all choices of (first and second order) variables by $x \mapsto x+o$ to obtain satisfiability of $\psi$.
\end{proof}

The above lemma can also be stated in the language of automata theory: the logic $MF_2(\N, \prec)$ can define exactly the $\omega$-regular languages, and the decidability of the logic corresponds to the decidability of emptiness of such a language (given a defining finite-state automaton). The logic $MF_2(\Z, \prec)$ can be characterized by automata as well, by $\zeta$-automaticity. We can give an automaton characterization of $MF_2(\Z, 0, \prec)$. These automata are nothing more than the usual $\zeta$-automatic languages and transducers, except that we fix an origin.

A \emph{$\Z$-automaton} is a tuple $T = (G, \Sigma, \mu, F^-, F^+)$ where $G$ is a finite graph, $\mu : V(G) \to \Sigma \times \{0,1\}$ a labeling function, $F^-, F^+ \subset 2^{V(G)}$. Its \emph{language} $L(T)$ is the set
\[ \{x \in \Sigma^\Z \;|\; T \mbox{ accepts } (x, \dots000.1000\dots) \} \]
where $T$ accepts a point $y$ if $y$ is the $\mu$-label of a path in $G$ whose set of states visited infinitely many times on the right is in $F^+$ and the set visited infinitely many times on the left is in $F^-$. A \emph{$\Z$-transducer} is an automaton whose alphabet is a product $\Sigma \times \Delta$, and its language is interpreted as a relation, or as the graph of a partial multivalued function. We call its language its \emph{graph}. A \emph{$\Z$-transducer} is \emph{deterministic} if it gives at most one value to every input.

The following is essentially by definition. 

\begin{lemma}
A language $L \subset \Sigma^\Z$ is $\Z$-automatic if and only if for some $\zeta$-automatic language $L'$ over alphabet $\Sigma \times \{0,1\}$, satisfying $\forall x \in L': |\pi_2(x)|_1 = 1$, we have $L = \{\pi_1(x) \;|\; x \in L', \pi_2(x)_0 = 1 \}$.
\end{lemma}

The following lemma links sofics and automatic sets.

\begin{lemma}
\label{lem:AutomaticSoficConnection}
Let $M = \Z$ or $M = \N$. Then an $M$-automatic set is a sofic subshift if and only if it is topologically closed and closed under the shift. A $\Z$-automatic set is $\zeta$-automatic if and only if it is shift-invariant.
\end{lemma}

\begin{proof}
A sofic shift is shift- and topologically closed, and is easily seen to be automatic.

For the other direction, consider first the case $M = \Z$. Suppose we have an $M$-automaton for $X$, and $X$ is topologically closed and shift-invariant. Remove from the automaton those states from which we cannot reach a cycle whose states form a set in $F^+$, those states from which, going backwards, we cannot reach a state marking the origin and then a cycle whose states form a set in $F^-$. Then all labels of finite paths in the automaton correspond to words that can appear after the origin, thus all words by the assumptions. The resulting automaton proves soficity.

When $M = \N$, remove those states that cannot be reached from an initial state and then again those states from which we cannot reach a cycle with all states accepting, and again we obtain an automaton proving soficity.
\end{proof}

\begin{lemma}
\label{lem:DecidableThings}
Let $M = \Z$ or $M = \N$, and let $X \subset \Sigma^M$ and $Y, Z \subset (\Sigma^2)^M$ be $M$-automatic sets. Then the following are algoritmically decidable:
\begin{itemize}
\item Is $X$ topologically closed? Is it shift-invariant? Is $X$ an SFT?
\item Is $Y$ reflexive/symmetric/antisymmetric/transitive on $X$?
\end{itemize}
The following can be effectively computed and are $M$-automatic:
\begin{itemize}
\item the minimal forbidden patterns of $X$, if it is an SFT,
\item the restriction of $Y$ to $X$, i.e. $Y \cap X^2$,
\item the topological closure and the shift-closure of $X$,
\item the composition $Y \circ Z$.
\end{itemize}
\end{lemma}

\begin{proof}
The proofs are basic automata theory. We outline some ideas.

Similarly as in the previous lemma (removing useless states, then making all states final), we can turn an $M$-automaton into one accepting its closure. Equality of $M$-automatic languages is decidable by Lemma~\ref{lem:MF2Decidable}. Similarly for shift-invariance. This shows the first and the penultimate item.

Given $X$, it is an SFT if and only if it is sofic and the set of first offenders of its colanguage is finite. Soficity is checked by checking topological and shift-closure, and the colanguage is computed by taking the complement. The first offenders of a regular language $L \subset \Sigma^*$ are the effectively regular language $L \setminus (L \Sigma \cup \Sigma L)$, and finiteness of a regular language is easy to verify.

Given a relation $Y \subset X^2$, in each case one can directly program each of the statements ``$Y$ is $P$'', where $P$ = reflexive/symmetric/antisymmetric/transitive, in $MF_2(\N, \prec)$ or $MF_2(\Z, 0, \prec)$. There are also easy automata-theoretic constructions: In the case of reflexivity, one can alternatively form an automaton for the diagonal relation and check inclusion, in the case of symmetry one can flip the relation and check equality, in the case of antisymmetry, one can flip the relation and check trivial intersection. In the case of transitivity one can compute an automaton for the composition: to compose $R$ and $R'$, given $(x, z)$, we simply guess $y$ such that $(x, y) \in R, (y, z) \in R'$.
\end{proof}

\begin{lemma}
Given a sofically presented system, it is decidable whether it is expansive, and whether it is finitely presented.
\end{lemma}

\begin{proof}
Expansivity is equivalent to being finitely presented by results of Fried. For finitely presentedness, by the same results, we just need to check whether the quotiented sofic is a relative SFT, which can be done using the previous lemma.
\end{proof}

Finding the transitive closure is only semidecidable, in the following sense.

\begin{lemma}
There exists a semialgorithm that, given $\Z$-automatic $R \subset (\Sigma^2)^M$, returns a $\Z$-automaton for the transitive closure $R^*$ in finite time if $\exists m: R^{\leq m} = R^*$. Given a symmetric reflexive sofic relation $R \subset (\Sigma^2)^\Z$, it is undecidable whether there exists $m$ such that $R^{\leq m} = R^*$.
\end{lemma}

\begin{proof}
Compute $R^{\leq m}$ for larger and larger $m$. We have $R^* = R^{\leq m}$ for some $m$ if and only if $R^{\leq m}$ is transitive, and this can be verified in finite time by the previous lemma.

For the undecidability claim, let $T$ be a reversible Turing machine, which we can see as acting in the moving tape model \cite{Ku97}, i.e. as an endomorphism of the sofic shift $((A \times Q) \sqcup A)^\Z$. Let $R$ be the relation containing $(x, x), (x, T(x)), (T(x), x)$ for all $x$ with one head, and $(x, x)$ for all $x \in ((A \times Q) \sqcup A)^\Z$.

Clearly $R^* = R^{\leq m}$ for some $m$ if and only if $T$ is periodic, which is undecidable by \cite{KaOl08}.
\end{proof}

We do not know whether, given a sofic relation $R \subset (A^2)^\Z$, it is undecidable whether $R^* = (A^2)^\Z$.

\begin{example}
\label{ex:SFTNotTrans}
SFT relations are not closed under composition.
Let $f : \{0,1\}^\Z \to \{0,1,2\}^\Z$ be the block map with neighborhood $\{0,1\}$ and local rule $f' : \{0,1\}^{\{0,1\}} \to \{0,1\}$ defined by $f'(a,b) = a \oplus b \in \{0,1\}$, where $\oplus$ is addition modulo $2$. Let $g : \{0,1\}^\Z \to \{0,1,2\}^\Z$ be the block map with local rule $g' : \{0,1\}^{\{0,1\}} \to \{0,1,2\}$ defined by $g'(a,b) = \left\{\begin{array}{ll}
0, & \mbox{if } a = b \\
1, & \mbox{if } ab = 01 \\
2, & \mbox{if } ab = 10
\end{array}\right.$. Seeing $f$ and $g$ as relations and extending their alphabets, we obtain SFT relations $R_f, R_g \subset \{0,1,2\}^\Z$ (say the domain of the function is on the left). The relation $R_f^R$ defined by $(x,y) \in R_f^R \iff (y,x) \in R_f$ is also SFT, but it can be shown that $R_f^R \circ R_g$ is proper sofic. \qee
\end{example}

\subsection{Morphisms between \ttfrac{subshift}{SFT} systems}
\label{sec:ExpansiveMorphisms}

In the case of expansive systems, i.e. \ttfrac{subshift}{SFT} systems, it is well-known that the number of morphisms between two systems is countable. We show that they can be presented by block maps, and at least in the finitely presented case over $\N$ or $\Z$, basic computational manipulation of such block maps is decidable. In the next section, we show that this fails for sofically presented systems. For concreteness, we consider $M = \Z$, but the case $M = \N$ works exactly the same.

Let $X, Y \subset \Sigma^\Z$ be subshifts and $K \subset X^2, L \subset Y^2$ subshift equivalence relations.

The \emph{domain} of a $\Z$-transducer is the domain of the relation it defines, and the \emph{codomain} the set of images. We use the usual function notation, writing $T(x)$ for the set of $T$-images of $x$. For a deterministic transducer, this is the singleton containing the unique image, which we often identify with the image directly. 

If a relation $f\subset X\times Y$ relates pairs of $K$-equivalent points in the domain of $f$ to pairs of $L$-equivalent points, it is called a $(K,L)$-morphism. If the domain of $f$ contains a representative from every $K$-equivalence class, then $f$ induces a function $f_{K,L}:X/K\to Y/L$ in a natural way (which we may also denote by $f$ if there is no risk of confusion).
Two $(K,L)$-morphisms $f,g : X \to Y$ are called \emph{$L$-equivalent} if they have the same domain and if their images are in $L$-relation for all $x$ in the domain. This means that $f_{K,L}$ and $g_{K,L}$ are the same function.

We show that transducers can be determinized. For practical purposes, this determinization step is not needed, but it simplifies certain things (like function composition) conceptually.

\begin{lemma}
If $T$ is a transducer, then there is a deterministic transducer $T'$ with the same domain as $T$ such that $T'(x) \subset T(x)$ for all $x$ in the domain.
\end{lemma}

\begin{proof}
Let $T : A^\Z \to B^\Z$. Given $x \in A^\Z$, define $T'(x)_{[0,\infty)}$ to be lexicographically minimal such that $T'(x)_{[0,\infty)} = y_{[0,\infty)}$ for some $y \in T(x)$, and then define $T'(x)_{(-\infty,-1]}$ to be lexicographically minimal among those left tails that are compatible with $T'(x)_{[0,\infty)}$ in $T(x)$. This defines a function whose graph is contained in that of $T(x)$, and the definition can be stated in $MF_2(\Z,0,\prec)$, so we obtain a transducer.
\end{proof}

\begin{example}
Consider $X = \{0,1\}^\Z$ and let $T = A_{req}$, i.e.\ think of the equivalence relation as a $\Z$-transducer that simply verifies the relation $A_{req}$ everywhere (and keeps no state). The determinization $T'$ constructed in the previous lemma is a $\Z$-automaton that, given a configuration $x \in \{0,1\}^\Z$, chooses the minimal representative in the specified ordering. Concretely the choices are as follows: If the right tail of $x$ contains infinitely many $0$s and infinitely many $1$s, pick $x$. If the right tail is eventually constant, then
\[ T'(0^\Z) = T'(1^\Z) = 0^\Z, \]
\[ T'(y.w10^\omega) = T'(y.w01^\omega) = y.w01^\omega, \]
\[ T'(y10^k.0^\omega) = T'(y01^k.1^\omega) = y10^k.0^\omega. \]
Note that since $E_{req}$ is an equivalence relation, $(x, y) \in E_{req}$ implies $T'(x) = T'(y)$, so we have constructed an automaton that turns the abstract identity function into a concrete one by picking a representative for each point, by an automaton. \qee
\end{example}

Note that the representative is not picked in a shift-invariant way in the above example. For the solenoid, there is also a deterministic shift-invariant transducer that picks one, but we do not know whether there is a general construction that would yield one. By classical determinization results in $\omega$-automata theory, the graph of a deterministic $\Z$-transducer is of the form $A \cap B$ where $A \in \Sigma^0_2, B \in \Pi^0_2$ in the Borel hierarchy. We do not even know whether nondeterministic transducers admit shift-invariant Borel sections at this (or any other) level.

If $X \subset A^\Z$ is a subshift, its \emph{$k$th SFT approximation} is the SFT with allowed patterns $\B_k(X)$.

\begin{lemma}
Let $f_{K,L} : X/K \to Y/L$ be a morphism induced by a fixed $(K,L)$-morphism $f\subset X\times Y$, $X \subset A^\Z, Y \subset B^\Z$. For every $\epsilon > 0$ there exists a deterministic $\Z$-transducer $T$ which interpreted as a relation $T\subset (A \times B)^\Z$ is $\epsilon$-close to the graph of $L\circ f\circ K$ and whose domain contains $X$.  If $X$ (resp. $Y$) is sofic, we may restrict the domain (resp. codomain) of $T$ to be equal to $X$ (resp. to be included in $Y$). If $Y$ is sofic and $L$ is an SFT relation, then $T$, restricted to $X$, can be made to be $L$-equivalent to $L\circ f\circ K$. Then in particular $T_{K,L}$ and $f_{K,L}$ are the same function.
\end{lemma}

\begin{proof}
First assume only that $f_{K,L} : X/K \to Y/L$ is a morphism, i.e. $f_{K_L}: X/K \to Y/L$ is continuous, and commutes with the shift, i.e. $(\sigma(f(x)), f(\sigma(x)) \subset L$.

Take the full graph
\[ R = L\circ f\circ K= \{(x, y) \;|\; \exists x', y': (x, x') \in K, (y, y') \in L, y'\in f(x') \} \]
as a subset of $X \times Y$. This is a subshift: shift-invariance is obvious, and if $(x, y) \notin R$, then $f_{K,L}(K(x)) \neq L(y)$, and since the maps $\pi_K:X\to X/K$, $\pi_L:Y\to Y/K$, $f_{K,L}$ are continuous, $f_{K,L}(\pi_K(B_\epsilon(x))) \cap \pi_Y(B_{\epsilon}(y)) = \emptyset$ for some $\epsilon > 0$. Then in particular $f(K(B_{\epsilon}(x)))\cap L(B_\epsilon(y))=\emptyset$.

The $m$th SFT approximation of $R$ in $(A \times B)^\Z$ for any $m$ is trivially a $\Z$-transducer whose domain contains $X$. By the previous lemma, it can be determinized to a transducer with the same domain.

The second claim follows because sofic shifts are $\Z$-automatic, and $\Z$-automatic subshifts are closed under Cartesian product and intersection. For the last claim, if $L$ is an SFT relation in the sofic $Y$, the image of $X$ under $T$ is contained in $Y$, and if $m$ of the previous paragraph is large enough, then for every $x$, $T(x)$ is $L$-equivalent with the elements of $f(x)$, and thus the induced map $T_{K,L} : X/K \to Y/L$ must be well-defined and equal to $f_{K,L}$.
\end{proof}

The following shows that, in some sense, all morphisms come from ``almost shift-commuting'' transducers.

\begin{lemma}
Let $X, Y$ be sofic and $K, L$ be SFT relations. Then $f : X/K \to Y/L$ is a morphism if and only if there exists a deterministic transducer $T : X \to Y$ such that $\forall (x, x') \in K: (T(x), T(x')) \in L$ and $\forall x \in X: (\sigma(T(x)), T(\sigma(x))) \in L$.
\end{lemma}

\begin{proof}
It is easy to verify that any $f$ defined by such a deterministic transducer $T$ is indeed a morphism. For the converse, by the previous lemma, for any morphism $f$ there exists a deterministic transducer $T$ from $X$ to $Y$ defining the same function from $X/K$ to $Y/L$. In particular, it is well-defined and shift-commuting, thus the stated formulas must hold.
\end{proof}

\begin{theorem}\label{lem:SFTSFTConjugacySemidecidable}
Let $X, Y$ be sofic and $K, L$ be SFT equivalence relations. Then it is semidecidable whether $X/K \cong Y/L$.
\end{theorem}

\begin{proof}
By the previous lemma, conjugacy is equivalent to the existence of deterministic transducers $T : X \to Y$, $T' : Y \to X$ such that
\[ \forall (x, x') \in K: (T(x), T(x')) \in L, \;\; \forall (y, y') \in L: (T'(y), T'(y')) \in K, \]
\[ \forall x \in X: (\sigma(T(x)), T(\sigma(x))) \in L, \;\; \forall y \in Y: (\sigma(T'(y)), T'(\sigma(y))) \in K, \]
\[ \forall x \in X: (T'(T(x)), x) \in K, \;\; \forall y \in Y: (T(T'(y)), y) \in L. \]
These statements can be easily programmed in $MF_2(\Z, 0, \prec)$.
\end{proof}

\begin{corollary}
Conjugacy of finitely presented systems is semidecidable.
\end{corollary}

\subsection{Automorphism groups}
\label{sec:AutomorphismGroups}

We show that the techniques from the previous section do not generalize to sofically presented systems -- there is no way to present morphisms between such systems combinatorially. For this, we present a pathological example of an automorphism group of a sofically presented system.

It is natural to consider the automorphism group of a \ttfrac{subshift}{subshift} system as a topological group, and there are several topologies one can choose.

For automorphism groups of subshifts, some that have been studied are 
the discrete, profinite and periodic point topologies (see \cite{BoLiRu88}), and two more (families of) topologies are also defined in \cite{SaTo12d}. The compact-open topology appears on this list as $\mathcal{T}_{\mathrm{dis}}$ \cite{BoLiRu88}.

In this section, by default we study the automorphism group of a dynamical system in the compact-open topology, which outside the expansive case is no longer necessarily discrete. To specify the discrete topology, we'll say ``the discrete automorphism group'' or write ``dis'' as subscript.

\begin{proposition}
\label{prop:DiscreteInCompactOpen}
Let $X$ be \ttfrac{subshift}{SFT}. Then $\Aut(X)$ is a countable discrete group under the compact-open topology.
\end{proposition}

\begin{proof}
The countability of $\Aut(X)$ is due to Hedlund (at least when $X$ is a subshift~\cite{Hed69}). We recall the argument. Let $\epsilon > 0$ be some expansivity constant. For any $f \in \Aut(X)$ there exists $n \in \N$ such that $1/n<\epsilon/5$ and $d(x, y) \leq 1/n \implies d(f(x), f(y)) < \epsilon/5$. Let $\Aut_n(X) \subset \Aut(X)$ be the set of all such $f$. Pick any $1/n$-spanning set $S_n$ (i.e. a finite subset of $X$ such that every element of $X$ has distance less than $1/n$ from $S_n$) and to each $f \in \Aut_n(X)$, using the spanning property of $S_n$ associate $F_f : S_n \to S_n$ such that $d(F_f(x), f(x)) \leq 1/n <\epsilon/5$ for all $x \in S_n$. Then if $f, g \in \Aut_n(X)$ and $F_f = F_g$, then for all $x \in X, m \in M$ we have
\begin{flalign*}
&d(mf(x), mg(x)) = d(f(mx), g(mx)) \leq d(f(mx), f(y)) + d(f(y), F_f(y)) \\
&+d(F_f(y), F_g(y)) + d(F_g(y),g(y)) + d(g(y), g(mx)) < \epsilon
\end{flalign*}
for a $y \in S_n$ such that $d(y,mx)\leq 1/n$, which implies $f(x) = g(x)$ for any $x$ by expansivity, implying $f = g$, thus countability since each $S_n$ is finite.

Discreteness is proved similarly: Let $\Aut_n(X), S_n$ be as above and let $f \in \Aut_n(X)$. We need to show $\{f\}$ is open in the compact-open topology. To see this, consider the set
\[ U' = \bigcap_{y \in S_n} V(\overline{B_{1/n}(y)}, B_{\epsilon/5}(f(y))) \]
where for compact $K$ and open $U$, $V(K, U)$ is an element of the subbasis, namely the set of functions $g$ such that $g(K) \subset U$.

The set $U'$ is open in the compact-open topology, since it is a finite intersection of open sets. It contains $f$ by the assumption $f \in \Aut_n(X)$, and if $g \in U'$, then for all $x \in X, m \in M$ we have
\[d(mf(x), mg(x)) = d(f(mx), g(mx)) < d(f(mx), f(y)) + d(f(y), g(y)) + d(g(y), g(mx)) < \epsilon \]
for a $y \in S_n$ such that $d(y,mx)\leq 1/n$, which implies $f(x) = g(x)$ for any $x$ by expansivity, implying $f = g$.
\end{proof}

The above proof gives another point of view\footnote{Formally, this is essentially the same thing, except that the threading is performed through interpretation rather than by the automaton.} to morphisms than the transducer approach of Section~\ref{sec:ExpansiveMorphisms}. Namely, morphisms of expansive systems are (up to recoding) given by block maps, but with an additional ``threading'' step. We explain this in the following example.

\begin{example}
Let $\ccfrac{X}{K}$ be $\ttfrac{$\Z$-subshift}{SFT}$, $X \subset \Sigma^\Z$. Let $f_{K,K} : X/K \to X/K$ be a morphism induced by a fixed $(K,K)$-morphism $f\subset X\times X$. Pick a large $k$, and a large $n$, and for every word $w \sqsubset X$ of length $2n+1$ pick a point $x_w \in [w]_{-n}$, and fix $F(w)=y_{[-k,k]}$ for some $y\in f(x_w)$. Define the \emph{threading relation} $T \subset (\Sigma^{[-k,k]})^\Z \times \Sigma^\Z$ by
\[ (y,x) \in T \iff x \in X \wedge \forall i \in \Z: (x_{[i-k,i+k]},y_i)\in L(K). \]
Observe that there can be at most one threading if $k$ is large enough, by expansivity (i.e. every $y$ is in relation with at most one $x$). Now, given $x \in X$, define a relation $\hat F \subset X^2$ by
\[ (x, y) \in \hat F \iff ((F(x_{[i-n,i+n]}))_{i \in \Z}, y) \in T. \]
If $n$ and $k$ are large enough, we necessarily have
\[ (x, y), (x', y') \in \hat F \wedge (x,x') \in K \implies (y,y') \in K \]
by expansivity, $\hat F$ is uniquely specified once we know $F$, and $(x,y) \in \hat F \iff f_{K,K}([x]) = [y]$. Thus, we can say that $F$ is a local rule for $f_{K,K}$. By recoding, we may assume $\hat F$ is the usual kind of block map, i.e. the image is always a single letter. \qee
\end{example}

Since Proposition~\ref{prop:DiscreteInCompactOpen} shows that the automorphism group of a \ttfrac{subshift}{SFT} is locally compact (under the compact-open, equivalently the discrete topology), one may ask basic questions about these groups. Automorphism groups of subshifts are studied in several works~\cite{BoLiRu88,Hed69,KR91,FiFi96}, and it is known that many phenomena can appear -- already for SFTs, the Tits alternative can fail~\cite{Salo19}, residual finiteness can fail~\cite{SaSc16a}, the torsion problem can be undecidable for f.g. subgroups~\cite{BaCaSa16}, and one can find f.g. infinite simple subgroups~\cite{SaSc16a}. For general subshifts, one can find f.g. subgroups of intermediate growth and infinite f.g. torsion groups~\cite{SaSc16a}. At the moment, with finitely presented systems whose underlying space is connected, the authors do not know how to do any of these things. However, interestingly enough, unlike in the case of automorphism groups where distortion elements currently seem difficult to produce, there is a trivial way to obtain them for connected finitely presented systems:

\begin{example}
Let $\ccfrac{X}{K}$ be the solenoid system, $X \subset \{0,1\}^\Z$. The product system with diagonal action is $\ccfrac{X^k}{K^k}$ where $K^k$ is interpreted as comparing configurations trackwise. Cartesian products of the solenoid with itself are connected, and the automorphism group of $\ccfrac{X^k}{K^k}$ contains an obvious\footnote{It is obvious from the abstract characterization of the solenoid system that adding a track to another gives an automorphism. The transducer performing this is also obvious, but seeing that it is continuous requires a bit of thinking. A concrete local rule for the morphism $(x,y) \mapsto x+y$ is obtained by summing runs of length $3$ in $x$ and $y$ together pairwise, with an inserted carry of $0$ in each sum. The threading takes care of incorrect carries.} copy of the group $\mathrm{GL}(k,\Z)$ of full rank $k$-by-$k$ matrices over $\Z$. Doing the same with the circle $X \subset \{0,1\}^\N$ gives a path-connected finitely presented $\N$-system with a copy of $\mathrm{GL}(k,\Z)$ in the automorphism group. \qee
\end{example}

The example also shows that connected \ttfrac{$\Z$-subshift}{SFT} systems can have nonamenable automorphism groups. We now define and study a pathological example where amenability cannot even be studied in the standard sense, because already local compactness fails. Fix $M = \Z$, and let us first consider the example with the discrete topology.

\begin{example}
\label{ex:CountableSoficSoficUncountableAut}
There is a \ttfrac{countable sofic}{countable sofic} whose discrete automorphism group is uncountable and contains every countable group, namely $\Aut_{\mathrm{dis}}(X_{\leq 2}/(X_{\leq 1}^2 \cup \Delta_{X_{\leq 2}}))$ is such. To see this, observe that for any permutation $\pi$ of $\N$, the unique shift-commuting map that satisfies
\[ \dots 000.10^n1000\dots \mapsto \dots 000.10^{\pi(n)}1000\dots \]
for all $n \in \N$ is an automorphism and this gives an embedding of the group of bijections of $\N$ into $\Aut_{\mathrm{dis}}(X_{\leq 2}/X_{\leq 1}^2)$. This group contains every countable group by Cayley's theorem. \qee
\end{example}

It follows $\Aut_{\mathrm{dis}}(X)$ can be nonamenable even if $X$ is countable. When $X$ is a subshift, this is impossible \cite{SaSc16a}.

Now let us consider the compact-open topology. First, it is known that indeed $\Aut(X)$ is a topological group for any dynamical system $X$ \cite{Ar46}. Since $X$ is compact and metrizable, $\Homeo(X)$ is metrizable \cite{McIb06} and thus its subspace $\Aut(X)$ is also metrizable. Indeed, it is an exercise that if $X$ is a compact space and $Y$ is a metric space with metric $d$, then the compact-open topology on $C(X, Y)$ is metrized by $e(f, g) = \sup\{d(f(x), g(x)) \;|\; x \in X\}$.

It would be nice if our groups were locally compact (which is the case on subshifts trivially since they are discrete), as much of group theory concentrates on the locally compact case. For example, we do not know whether there are analogues of the Haar measure and amenability outside this framework. Unfortunately, the same example as above shows that the groups are not locally compact even in the sofically presented case. 

\begin{example}
The group $\Aut(X_{\leq 2}/(X_{\leq 1}^2 \cup \Delta_{X_{\leq 2}}))$ is not locally compact in the compact-open topology. We show that $\id$ does not have a compact neighborhood. Suppose $\id \in U \subset K$ where $U$ is open and $K$ is compact.

Now let $\sigma_n$ be the map on $X_{\leq 2}$ that shifts all configurations where the distance of $1$s is at least $n$, as well as $X_{\leq 1}$, one step to the left, and fixes the remaining points (which form finitely many $\sigma$-orbits). It gives an automorphism of $X_{\leq 2}/(X_{\leq 1}^2 \cup \Delta_{X_{\leq 2}})$. It is easy to show that as $n \rightarrow \infty$, we have $\sigma_n^m \rightarrow \id$ uniformly in $m \in \Z$. In particular, for any large enough $n$ we have $\sigma_n^m \in U$ for all $m$.

Since our space $\Aut(X_{\leq 2}/(X_{\leq 1}^2 \cup \Delta_{X_{\leq 2}}))$ is metrizable, $K$ is compact if and only if it is sequentially compact. Since $U \subset K$, it follows that $\sigma_n^m$, $m \rightarrow \infty$ has a converging subsequence. For any large enough $n$ let $f$ be a limit point of this sequence, and observe that $f$ maps $\dots 000.10^n1000\dots$ into the equivalence class of $0^\Z$. Thus it cannot be an automorphism. \qee
\end{example}

\section{Automatic spaces}
\label{sec:ConjugacyUndecidability}

\subsection{Definition and examples of automatic spaces}
\label{sec:AutomaticSpaces}

Let $S, Q$ be classes of $\omega$-automatic sets. An \emph{\ttfrac{S}{Q}-automatic space} is a topological space of the form $Y/Z$ where for some finite alphabet $A$, $Y \subset A^\N$ is $\omega$-automatic and $Z \subset Y^2$ is a closed $\omega$-automatic equivalence relation. We are mainly interested in two classes of $\omega$-automatic sets: the class ``any'' of all $\omega$-automatic sets, and the class ``closed'' of all closed $\omega$-automatic sets. An \emph{automatic space} is by default an \ttfrac{any}{closed}-automatic space. The \emph{compact automatic spaces} are precisely the \ttfrac{closed}{closed}-automatic spaces. 

\begin{theorem}
\label{thm:SimplicialComplexesAutomatic}
Every finite simplicial complex is automatic.
\end{theorem}

\begin{proof}
Let $S \subset [0,1]^d$ be a standard embedding of an abstract simplicial complex $S'$, i.e.
\[ S = \{\vec v \in [0,1]^d \;|\; \sum_{j = 1}^d \vec v_j = 1 \wedge \{j \;|\; \vec v_j > 0\} \in S'\}. \]
i.e. $S$ is obtained by linearly spanning abstract simplices in $S'$ into concrete ones.

Take the alphabet $\{0,1\}^d$ and define the following closed $\omega$-automatic set:
\[ Y = \{y \in (\{0,1\}^d)^\N \;|\; (\bin(y_j))_{1 \leq j \leq d} \in S \} \]
where $\bin(y) = \sum_i 2^{-i} y_i$. To see that this is indeed $\omega$-automatic, observe that we simply need to verify that $\sum_{j = 1}^d \bin(y_j) = 1$, 
and then verify that the set of coordinates with nonzero value is in $S'$. The computation of the sum can be done by an automaton implementation of the standard school algorithm, by adding digits together and keeping the (bounded) carry in the state of the automaton. Checking that the set of coordinates with nonzero value is in $S'$ can be done by keeping track of nonzero coordinates seen so far in the state of the automaton.

Now let $Z'$ be the relation that identifies two configurations $(x,y)$ in $\{0,1\}^\N$ if they represent the same number. This is closed $\omega$-automatic since $(x,y)$ is in this relation iff $(x, y)_{[i,i+2]} \in A_{\mathrm{req}}$ for all $i$, and $(x, y)_{[0,2]} \neq \left( \begin{smallmatrix} (1-a) & (1-a) & (1-a) \\ a & a & a \end{smallmatrix} \right)$. 

Let $Z = (Z')^d$. We claim that $Y/Z$ is homeomorphic to $S$. To see this, observe that the map from $Y$ to $[0,1]^d$ defined by $y \mapsto (\bin(y_1), \bin(y_2), \dots, \bin(y_d))$ by definition maps onto $S$, and is continuous by the formula defining $\bin$. This gives a map from $Y/Z$ to $S$, which is still continuous by the definition of the quotient topology. Since $Y/Z$ is compact and $S$ is Hausdorff, this gives a homeomorphism.
\end{proof}

\subsection{Underlying spaces of sofically presented systems}

We prove that the underlying spaces of sofically presented systems are automatic. 

\begin{theorem}
\label{thm:soficsoficautomatic}
The underlying topological space of every sofically presented system with $M \in \{\N, \Z\}$ is (effectively) automatic.
\end{theorem}

\begin{proof}
A sofic shift $Y \subset \Sigma^M$ is precisely a subshift whose language $L$ is regular. If $M = \N$, then we can construct an automaton that, given $x \in \Sigma^M$, checks that every prefix of $x$ is in $L$, and similarly for the kernel subshift.

If $M = \Z$, we observe that zipping the left and right end into a one-sided sequence is a homeomorphism. Thus, we construct an automaton over the alphabet $\Sigma^2$ that, given $(x, y) \in (\Sigma^2)^\N$, checks that for all $n$, $u^R v$ is in $L$ for all prefixes $u, v$ of $x, y$ respectively, with $|u| = |v| = n$, and similarly for the kernel subshift.
\end{proof}

\begin{definition}
Let $X$ be any topological space. Let $\overset{\rightarrow}X = (X\times\N) \cup \{\infty\}$ where $\infty$ is the one-point compactification of $X\times\N$. Define $\overset{\rightarrow}s(\infty) = \overset{\rightarrow}s(x, 0) = \infty, \overset{\rightarrow}s(x, n+1) = (x, n)$. Let $\overset{\leftrightarrow}X = (X\times\Z) \cup \{\infty\}$ where $\infty$ is the one-point compactification of $X\times\Z$. Define again $\overset{\leftrightarrow}s(\infty) = \infty, \overset{\leftrightarrow}s(x, n) = (x, n-1)$.
\end{definition}

See Figure~\ref{fig:SphereCarousel} for an illustration.

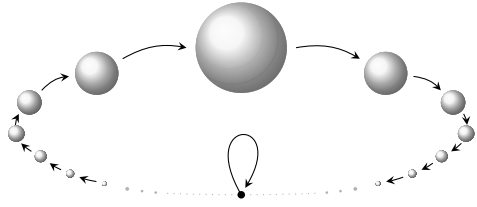
\begin{figure}[h]
\begin{center}
\begin{tikzpicture}[scale = 1.5];
\def\ys{0.65}
\shade[ball color = gray!10, opacity = 0.6] (1.2,-0.85*\ys) circle (0.02);
\shade[ball color = gray!10, opacity = 0.6] (1.5,-0.7125168290796722*\ys) circle (0.035);
\shade[ball color = gray!10, opacity = 0.6] (1.7581665796442194,-0.4766681965010218*\ys) circle (0.052);
\shade[ball color = gray!10, opacity = 0.6] (1.9714841972416093,-0.16826323129742857*\ys) circle (0.072);
\shade[ball color = gray!10, opacity = 0.6] (1.8574469507609722,0.25518371947877134*\ys) circle (0.107);
\shade[ball color = gray!10, opacity = 0.6] (1.2666538192549668,0.6516099682706644*\ys) circle (0.19);

\shade[ball color = gray!10, opacity = 0.6] (0, 1.0*\ys) circle (0.4);

\shade[ball color = gray!10, opacity = 0.6] (-1.2666538192549666,0.6516099682706646*\ys) circle (0.19);
\shade[ball color = gray!10, opacity = 0.6] (-1.857446950760972,0.25518371947877168*\ys) circle (0.107);
\shade[ball color = gray!10, opacity = 0.6] (-1.9714841972416093,-0.16826323129742846*\ys) circle (0.072);
\shade[ball color = gray!10, opacity = 0.6] (-1.7581665796442194,-0.4766681965010217*\ys) circle (0.052);
\shade[ball color = gray!10, opacity = 0.6] (-1.5,-0.7125168290796721*\ys) circle (0.035);
\shade[ball color = gray!10, opacity = 0.6] (-1.2,-0.85*\ys) circle (0.02);

\node[circle,fill=black!30!white,inner sep=0.5pt] at (-1,-0.92*\ys) {};
\node[circle,fill=black!30!white,inner sep=0.35pt] at (-0.87,-0.95*\ys) {};
\node[circle,fill=black!30!white,inner sep=0.2pt] at (-0.75,-0.97*\ys) {};

\node[circle,fill=black!30!white,inner sep=0.5pt] at (1,-0.92*\ys) {};
\node[circle,fill=black!30!white,inner sep=0.35pt] at (0.87,-0.95*\ys) {};
\node[circle,fill=black!30!white,inner sep=0.2pt] at (0.75,-0.97*\ys) {};

\draw[dotted,color=black!30!white] (-0.65,-0.98*\ys) edge[bend right = 1] (0,-1*\ys);
\draw[dotted,color=black!30!white] (0.65,-0.98*\ys) edge[bend left = 1] (0,-1*\ys);

\node[circle,fill=black,inner sep=1pt] at (0,-1*\ys) {};

\draw (0.48,1*\ys) edge[bend left=20, -stealth] (1.04,0.85*\ys);
\draw (1.51,0.61*\ys) edge[bend left=20, -stealth] (1.75,0.42*\ys);
\draw (1.95,0.1*\ys) edge[bend left=10, -stealth] (1.98,-0.05*\ys);
\draw (1.93,-0.3*\ys) edge[bend left=5, -stealth] (1.84,-0.41*\ys);
\draw (1.68,-0.57*\ys) edge[bend left=5, -stealth] (1.58,-0.66*\ys);
\draw (1.42,-0.76*\ys) edge[bend left=5, -stealth] (1.27,-0.82*\ys);

\draw (-0.48,1*\ys) edge[bend right=20, stealth-] (-1.04,0.85*\ys);
\draw (-1.51,0.61*\ys) edge[bend right=20, stealth-] (-1.75,0.42*\ys);
\draw (-1.95,0.1*\ys) edge[bend right=10, stealth-] (-1.98,-0.05*\ys);
\draw (-1.93,-0.3*\ys) edge[bend right=5, stealth-] (-1.84,-0.41*\ys);
\draw (-1.68,-0.57*\ys) edge[bend right=5, stealth-] (-1.58,-0.66*\ys);
\draw (-1.42,-0.76*\ys) edge[bend right=5, stealth-] (-1.27,-0.82*\ys);

\tikzset{every loop/.style={shorten >=3,min distance=8mm,in=60,out=120,looseness=10}}
\path[->] (0,-1*\ys) edge[loop above, -stealth] node {} ();

\end{tikzpicture}
\end{center}
\vspace{-10pt}
\caption{Depiction of the $\Z$-system $\protect\overset{\leftrightarrow}{S^2}$
. The dynamics follows the arrows. The thick black dot denotes the one-point compactification fixed by the dynamics. By Theorem~\ref{thm:leftrightsystem}, this system is \ttfrac{SFT}{sofic}.} 
\label{fig:SphereCarousel}
\end{figure}

\begin{theorem}
\label{thm:leftrightsystem}
Let $X$ be a compact automatic space. Then $(\overset{\leftrightarrow}X, \overset{\leftrightarrow}s)$ and $(\overset{\rightarrow}X, \overset{\rightarrow}s)$ are sofically presented. 
\end{theorem}

\begin{proof}
Consider first the case $M = \Z$. Let $X \subset \Sigma^\N$ be an $\omega$-regular closed language and $K \subset (\Sigma^2)^\N$ the corresponding $\omega$-automatic closed relation.

The sofic shift $Y$ we construct checks the following:
\begin{enumerate}
\item Every point is of the form
${^\infty}\# x$ for some $x \in \Sigma^\N$, or a degenerate version (i.e. all the points $x \in \Sigma^\Z$ and $\#^\Z$ are included).
\item The word $\# w$ is forbidden if $w_0 \in \Sigma$ and there is no point in $X$ with prefix~$w$.
\end{enumerate}
Note that in particular all of $\Sigma^\Z$ is included in $Y$.

The sofic shift $Z \subset Y^2$ we construct checks the following:
\begin{enumerate}
\item given a pair $(y, y')$, if either $y$ or $y'$ contains both $\#$ and some symbol from $\Sigma$, then every symbol in $(y, y')$ is $(\#, \#)$ or in $\Sigma^2$,
\item if $y_{i-1} = \#, y_i \neq \#$, then $(y, y')_{[i,\infty)} \in K$.
\end{enumerate}

Thus $Z$ has all of $(\#^\Z, \#^\Z)$ and $(\Sigma^2)^\Z \cap Y^2$ in the same equivalence class, and otherwise identifies points $y, y'$ if and only if both are of the non-degenerate form
$(y, y') = ({^\infty}\# x, {^\infty}\# x'),$
both $x, x' \in \Sigma^\N$ start from the same coordinate, and $(x, x') \in K$.

Now, let $Y_i$ be the set of points $\dots\#\#\# x$ where $x \in \Sigma^\N$ begins at the $i$th coordinate. Clearly $Y_i/K$ is homeomorphic to $X$. Since the degenerate points form only one point $Y_\infty$ in $Y/K$, and clearly all sequences of points from $Y_i/K$ with $|i| \rightarrow \infty$ converge to $Y_\infty$, it follows that $\overset{\leftrightarrow}X$ is homeomorphic to $Y$. The natural homeomorphism conjugates $s$ and $\sigma$.

Consider then $M = \N$. We use exactly the same construction: $Y$ contains all of $\Sigma^\N$, and after $\#$-symbols, we verify that the continuation is in $X$. The kernel relation $Z$ again identifies all points $\#^\N$ and all points in $\Sigma^\N$, and verifies that when both $\#$ and a letter of $\Sigma$ occur in one of the points $(y, y')$, the last $\#$ is in the same coordinate in both, and the continuation is a pair of points in $K$.
\end{proof}

Supposing that $A, B$ are subsets of some set $X$, and $A \cap B = \emptyset$, we say $A$ and $B$ are \emph{recursively inseparable} if there is no recursive set $C \subset X$ such that $A \subset C \subset X \setminus B$.

\begin{theorem}
\label{thm:ConjugacyUndecidableProof}
Let $M = \N$ or $M = \Z$. Given two sofically presented systems $X, Y$, the following properties are recursively inseparable:
\begin{itemize}
\item $X$ and $Y$ are conjugate,
\item $X$ and $Y$ are not homeomorphic.
\end{itemize}
\end{theorem}

\begin{proof}
We prove this for $M = \Z$. Given two connected simplicial complexes $S_1, S_2$, construct using Theorem \ref{thm:SimplicialComplexesAutomatic} automatic spaces $A_1, A_2$ such that $A_1 \cong S_1, A_2 \cong S_2$, and construct the \ttfrac{SFT}{sofic} shifts $\overset{\leftrightarrow}{A_1}$ and $\overset{\leftrightarrow}{A_2}$ using Theorem~\ref{thm:leftrightsystem}. If $A_1$ and $A_2$ are homeomorphic, the systems are obviously conjugate. If $\overset{\leftrightarrow}{A_1}$ and $\overset{\leftrightarrow}{A_2}$ are homeomorphic, then the infinite connected components are pairwise homeomorphic (with some correspondence). The infinite connected components are obviously all homeomorphic to $A_1$ in $\overset{\leftrightarrow}{A_1}$ and to $A_2$ in $\overset{\leftrightarrow}{A_2}$, so we must have $A_1 \cong A_2$. Since homeomorphism of connected simplicial complexes is undecidable~\cite{Mar58}, we have proved the result.
\end{proof}

\begin{example}
It is known to even be undecidable whether a given simplicial complex is homeomorphic to the $5$-sphere $S^5$~\cite{VKF74}, and the previous theorem gives a many-one reduction of simplicial complex homeomorphism simultaneously to the homeomorphism and conjugacy problems of sofically presented systems. 
Thus, by constructing a \ttfrac{$\Z$-sofic}{sofic} representation of $\overset{\leftrightarrow}{S^5}$ we obtain a single \ttfrac{sofic}{sofic} such that conjugacy (and homeomorphism) to it is undecidable.

Concretely, a \ttfrac{$\Z$-sofic}{sofic} system $\ccfrac{X}{K}$ is obtained as follows: $X$ is generated by configurations $\dots\#\#\# x$
where $x=(x_0,x_1,\dots,x_6) \in (\{0,1\}^7)^\N$, and we check that
\begin{enumerate}
\item $\sum_{i = 0}^6 r(x_i) = 1$,
\item $\exists i: x_i = 0^\N$.
\end{enumerate}
Namely, then the $x_i$ correspond to binary representations of points in the standard representation of the boundary of the $6$-simplex, i.e. $\{A \subset \{0,1,2,3,4,5,6\} \;|\; A \neq \{0,1,2,3,4,5,6\}\}$ (which is homeomorphic to $S^5$).

For two points of this form, the relation $K$ verifies that the rightmost $\#$ is in the same position, after which the corresponding binary tracks have the same real value (i.e. are in the relation $A_{\mathrm{req}}$ and begin with the same symbol or with $(100, 011)$). All configurations not of this form, i.e. the ones in $(\{0,1\}^7)^\Z \cup \{\#^\Z\}$, are simply identified by $K$. In other words,
\[ K = \cdots\#\#\# S \cup ((\{0,1\}^7)^\Z \cup \{\#^\Z\})^2 \cup \Delta_X, \]
where $S$ is the relation described above. \qee
\end{example}

\begin{example}
Connected components of \ttfrac{SFT}{SFT} systems can be arbitrary simplicial complexes. For this, perform the same construction as in Theorem~\ref{thm:ConjugacyUndecidableProof} for any connected simplicial complex $S$, but do not identify all of $\{\#^\Z\} \cup Y^2$ as one point. Then the homeomorphism type is the same, except that the one-point compactification is replaced by some quotient of $\{\#^\Z\} \cup Y$. Clearly, the points with a rightmost $\#$-symbol at $i \in \Z$ form a clopen set homeomorphic to $S$. \qee
\end{example}

To prove the undecidability of homeomorphism or conjugacy for \ttfrac{SFT}{SFT} systems, one would need more control on the $\{\#^\Z\} \cup Y$ component.

\section{Systems with trivial dynamics}
\label{sec:TrivDyn}

It is known that expansive systems that have only periodic points are finite (see \cite{MeSa19} for a general result). The following example shows this is no longer true for subshifts over sofics.

\begin{example}
There exists a \ttfrac{subshift}{countable sofic} dynamical system $X$ such that every point in $X$ has finite orbit, and for every $n$, there exists one periodic orbit of size $n$. Namely, let $Y$ be the orbit closure of the points $\{(0^n1)^\Z \;|\; n \geq 1\}$ and let $Z=X_{\leq 1}^2$. Then $X = Y/Z$ has the desired property. \qee
\end{example}

There are even subshifts over sofics that consist of infinitely many fixed points.

\begin{example}
\label{ex:Boyle}
There exists an infinite \ttfrac{subshift}{sofic} dynamical system $X$ such that every point in $X$ is a fixed point. Namely, let $Y$ be the orbit closure of the points $\{(0^n1)^\Z \;|\; n \geq 1\}$ and let $Z\subseteq(\{0,1\}^2)^\Z$ be the subshift whose language consists of all the subwords of $((1,0)(0,0)^*(0,1)(0,0)^*)^*\cup\{(0,0),(1,1)\}^*$, so $Z$ is sofic. The restriction of $Z$ to $Y$ identifies pairs of configurations $(y_1,y_2)$ where consecutive occurrences of $1$ alternate between $y_1$ and $y_2$. In other words, for every $n\geq1$ all the points in the orbit of $(0^n1)^\Z$ are identified with each other and all the points of $X_{\leq 1}$ are identified with each other, and no other identifications are made. The resulting system $Y/Z$ is conjugate to $(X,\id_X)$ where $X=\{1/n\mid n\in\Z, n\geq 1\}\cup\{0\}$. \qee
\end{example}

Note that $X$ is isomorphic to the ordinal $\omega+1$ with the order topology.

In the class of \ttfrac{subshift}{subshift} not only is it possible to have systems consisting of infinitely many fixed points, but the set of fixed points can even form any compact metric space. We present a straightforward proof, but the following theorem would also follow from a more general result of~\cite{BoFiFi02} saying that every dynamical system with zero entropy is a factor of a subshift.

\begin{theorem}
\label{thm:TrivialDynamics}
Let $X$ be any compact metric space. Then the trivial $\Z$-system $(X, \id_X)$ is \ttfrac{subshift}{subshift}.
\end{theorem}

\begin{proof}
Let $X \cong X'/K$ where $X' \subset \{0,1\}^\N, K \subset (X')^2$ are closed. We define $Y$ and $K'$ such that $\ttfrac{Y}{K'} \cong (X, \id)$ as a union of Toeplitz subshifts with a shared maximal equicontinuous factor.

Let $x \in \{0,1\}^\N$ be arbitrary and define the partial configuration
\[ y_0 = \dots \# x_0 \bla \# x_0 \bla . \# x_0 \bla \# x_0 \bla \# x_0 \bla \dots \]
where $\bla$ denotes an unknown value. For such a point, the sequence of $\#$s is referred to as the \emph{$\#$-skeleton}. Observe that if the holes $\bla$ are filled by a configuration $z$ such that $\#^m \not\sqsubset z$ for some $m$, then in every point of the orbit closure of the resulting configuration, we can locally determine a unique $\#$-skeleton (by a block map), and thus can determine the bit $x_0$.

Now, define $y_1$ by filling the sequence of $\bla$s in $y_0$ by
\[ \dots \# x_1 \bla \# x_1 \bla . \# x_1 \bla \# x_1 \bla \# x_1 \bla \dots \]
so that the occurrence of $\bla$ in $y_0$ closest to the origin (breaking ties arbitrarily) is replaced by a non-$\bla$ symbol. Note that no matter what the sequence of $\bla$s in $y_1$ is replaced with, $\#^3$ will not occur in it, and thus after filling the $\bla$s of $y_0$, we have a local rule detecting the $\#$-skeleton. If further the sequence of $\bla$s in $y_1$ is filled by any sequence not containing $\#^m$ for some $m$, its $\#$-skeleton, and thus the bit $x_1$, will be locally detectable. This process can be continued inductively to construct $y_n$ from $y_{n-1}$ and $x_n$, and $y_n$ converge to some configuration $y$ without occurrences of $\bla$.
 
To define the subshift $Y$, we require that $y\in Y$ is constructed by the process described above using some sequence $x_0 x_1 \dots$ belonging to $X'$. This is a shift-invariant condition since the bits $x_0 x_1 \dots x_i$ can be deduced from any long enough subword of any point obtained by the process, and it is closed since $X'$ is closed.

To define $K'$, we identify two points $y^1, y^2$ if their associated sequences are
$x^1_0 x^1_1 x^1_2 \dots$ and $x^2_0 x^2_1 x^2_2 \dots$, respectively, and $(x^1_0 x^1_1 x^1_2 \dots, x^2_0 x^2_1 x^2_2 \dots) \in K$. Again since $K$ is closed and the bits $x_0 x_1 \dots x_i$ can be locally deduced from any long enough subword, this is a closed shift-invariant condition, and it is an equivalence relation since $K$ is.

The map that computes $x_0 x_1 x_2 \dots$ from a point of $Y$ is continuous from $Y$ to $X$, and is clearly bijective between $K'$-classes and points of $X$, thus it is a homeomorphism between $Y/K'$ and $X$. Clearly the shift dynamics on $Y/K'$ is trivial, concluding the proof.
\end{proof}

In such examples, the numerator cannot be sofic:

\begin{lemma}
If $(X,S)$ is a transitive dynamical system, $(Y,T)$ is an infinite dynamical system and $Y$ is a factor of $X$, then some point in $Y$ has an infinite orbit.
\end{lemma}

\begin{proof}
Let $f:(X,S)\to(Y,T)$ be a factor map and choose some transitive point $x\in X$. Then $y=f(x)$ a transitive point in $Y$. But the system $Y$ is infinite, so the orbit of $y$ is infinite.
\end{proof}

\begin{proposition}
Let $(X,T)$ be a \ttfrac{$\Z$-sofic}{subshift} dynamical system. If $X$ is infinite, then some point in it has an infinite orbit.
\end{proposition}

\begin{proof}
Let $Z$ be a $\Z$-SFT and let $f:(Z,\sigma)\to (X,T)$ be a factor map and let $Z_1,\dots,Z_n$ be the transitive components of $Z$. Assume to the contrary that all points of $X$ are periodic; then by the previous lemma the systems $f(Z_i)$ are all finite. Now let $z\in Z$ be arbitrary and let $x=f(z)$. Choose a strictly increasing sequence of $n_t\in\N$ such that $(\sigma^{n_t}(z))_{t\in\N}$ converges to some point $z'$, which necessarily belongs to some transitive component $Z_i$. By continuity of $f$ it follows that $(f(\sigma^{n_t}(z)))_{t\in\N}=(T^{n_t}(f(z)))_{t\in\N}=(T^{n_t}(x))_{i\in\N}$ converges to $f(z')$. We assumed that all points of $X$ are periodic, so there are only finitely many distinct $T^{n_t}(x)$ and therefore $T^{n_t}(x)=f(z')$ for all sufficiently large $t$. Thus $x\in f(Z_i)$ and we may conclude that $X=\bigcup_{i=1}^n f(Z_i)$. This is a finite set, a contradiction.
\end{proof}

Theorem~\ref{thm:TrivialDynamics} has the following obvious corollary on automorphism groups.

\begin{proposition}
Let $X$ be any compact metrizable space. Then $\Homeo(X)$ is isomorphic to the automorphism group of some \ttfrac{subshift}{subshift}, as topological groups with the compact-open topologies.
\end{proposition}

\begin{proof}
This is direct from Theorem~\ref{thm:TrivialDynamics}, as commutation with the identity map is automatic, so $\Aut(X, \id) = \Homeo(X)$.
\end{proof}

It is not known whether $\Homeo(X)$ for a compact metrizable space $X$ can have dimension other than $0$, $1$ or $\infty$~\cite{Br01}. Can $\Aut(X)$ (in the compact-open topology) have a dimension other than $0$, $1$ or $\infty$ for a \ttfrac{subshift}{subshift} $X$? What about the automorphism group of a general dynamical $\Z$-system $X$?

\section{Case study 1: $\beta$-shifts}
\label{sec:BetaShifts}

In this section we consider a natural class of \ttfrac{$\N$-subshift}{subshift}-systems associated to $\beta$-shifts. In Definition~\ref{def:BetaShift} we define the usual $\beta$-shift $S_\beta$ associated to a real number $\beta > 1$. The $\beta$-shift codes the greedy $\beta$-adic expansions of real numbers. In Definition~\ref{def:BetaKernel}, we define the $\beta$-kernel $K_\beta$. Dividing a $\beta$-shift by the $\beta$-kernel gives, in a precise sense, the most natural map to the real numbers modulo $1$ under multiplication by $\beta$: this naturalness will be explained in more detail and in greater generality in Subsection~\ref{sec:TDSification}. The main theorem of this section gives a characterization of the soficity/SFTness types of the natural formal quotients $\ccfrac{S_\beta}{K_\beta}$.

\begin{theorem}
\label{thm:BetaShiftClassificationProof}
Let $\beta > 1$. Then the following hold:
\begin{enumerate}
\item $\ccfrac{S_\beta}{K_\beta}$ is \ttfrac{SFT}{SFT} if and only if $\beta \in \N$,
\item $\ccfrac{S_\beta}{K_\beta}$ is \ttfrac{SFT}{proper sofic} if and only if $\beta \notin \N$ and $d^*_\beta(1)$ is totally periodic,
\item $\ccfrac{S_\beta}{K_\beta}$ is \ttfrac{proper sofic}{proper sofic} if and only if $d^*_\beta(1)$ is strictly eventually periodic,
\item $\ccfrac{S_\beta}{K_\beta}$ is \ttfrac{non-sofic}{SFT}
if and only if $\{\sigma^i(d^*_\beta(1)) \;|\; i\in\N\}$ is dense in $S_\beta$.
\item $\ccfrac{S_\beta}{K_\beta}$ is \ttfrac{non-sofic}{non-sofic} in other cases.
\end{enumerate}
In particular, $\ccfrac{S_\beta}{K_\beta}$ is never \ttfrac{SFT}{non-sofic}, \ttfrac{non-sofic}{proper sofic} or \ttfrac{proper sofic}{SFT}.
\end{theorem}

\begin{proof}
By Lemma~\ref{lem:DenseNotSP}, if $\{\sigma^i(d^*_\beta(1)) \;|\; i\in\N\}$ is dense then $d^*_\beta(1)$ is not eventually periodic, so it is clear that any possible $\beta$ is covered by exactly one of the cases. Since the types \ttfrac{X}{Y} occurring in the antecedent are also mutually exclusive, it is enough to show the ``if'' directions.

Suppose first that $\beta \in \N$. Then by Lemma~\ref{Kint}, $d^*_\beta(1)$ is totally periodic, so by Theorem~\ref{thm:LothaireSFTSofic}, $S_\beta$ is SFT. By Theorem~\ref{thm:Kint}, $K_\beta$ is also SFT in this case.

Suppose then $\beta \notin \N$ and $d^*_\beta(1)$ totally periodic. Again by Theorem~\ref{thm:LothaireSFTSofic}, $S_\beta$ is SFT. By Theorem~\ref{thm:KIfSSFT}, $K_\beta$ is proper sofic.

Similarly, if $\beta$ is strictly eventually periodic, then Theorem~\ref{thm:LothaireSFTSofic} shows $S_\beta$ is proper sofic, and then by Theorem~\ref{thm:KIfSSofic} $K_\beta$ is proper sofic.

Suppose then that the orbit of $d^*_\beta(1)$ is dense in $S_\beta$. Then Lemma~\ref{lem:DenseNotSP} and Theorem~\ref{thm:LothaireSFTSofic} show that $S_\beta$ is non-sofic, and Theorem~\ref{thm:DenseCase} shows that $K_\beta$ is SFT, thus also sofic.

Suppose finally that $d^*_\beta(1)$ does not satisfy any of the above, i.e. it is aperiodic and its orbit is not dense. Then Theorem~\ref{thm:LothaireSFTSofic} shows that $S_\beta$ is non-sofic and Theorem~\ref{thm:NotDenseNotSofic} shows $K_\beta$ is non-sofic.
\end{proof}

\begin{remark}
Note that this theorem talks about the soficity/SFTness type of the formal quotient \ccfrac{S_\beta}{K_\beta}, not the abstract system \ttfrac{$S_\beta$}{$K_\beta$}. The ``if'' directions of course hold also for the abstract systems.
\end{remark}

\subsection{Topological dynamical system -ification}
\label{sec:TDSification}

In this section, because the application is $\beta$-shifts, we restrict to $\N$-actions for simplicity. We explain the general construction that attaches to a $\beta$-shift a corresponding multiplication map on the interval (or, equivalently, the circle). The idea is that if a compact metric space has non-continuous dynamics, then there is a unique best way to ``glue'' the underlying space so that the dynamics becomes continuous. This is expressed formally as a universal property in Theorem~\ref{maximal}.

\begin{definition}
A \emph{(non-topological) dynamical system (nTDS)} is a pair $(X,T)$, where $X$ is a non-empty set and $T:X\to X$ is a self-map on $X$.\end{definition}

If $R$ is a $T$-invariant equivalence relation for $(X,T)$, there is a unique map $T_R:X/R\to X/R$ such that $\pi_R\circ T=T_R\circ\pi_R$. If $(X,T)$ is a TDS and $R\subset X\times X$ is a topologically closed $T$-invariant equivalence relation on $X$, then the topological quotient space $X/R$ is a compact metric space and $T_R$ is a continuous map on $X/R$, so $(X/R,T_R)$ becomes a TDS.

If $(X,T)$ and $(Y,S)$ are dynamical systems and $\pi:X\to Y$ is a map such that the relation $S\circ\pi=\pi\circ T$ is satisfied, we write $\pi:(X,T)\to(Y,S)$. If $\pi:(X,T)\to(Y,S)$ is surjective, we say that $(Y,S)$ is a factor of $(X,T)$ (via the factor map $\pi$). If moreover $(X,T)$ and $(Y,S)$ are TDSs and $\pi$ is a continuous bijection, we say that $(X,T)$ and $(Y,S)$ are conjugate (via the conjugacy $\pi$).

\begin{definition}
We say that $\pi:(X,T)\to(Y,S)$ is a \emph{TDS factor} if $(X,T)$ is an nTDS, $X$ is a compact metric space, $(Y,S)$ is a TDS and $\pi$ is a continuous factor map.
\end{definition}

\begin{theorem}\label{maximal}Any nTDS $(X,F)$ where $X$ is a compact metric space possesses a maximal TDS factor $\pi:(X,F)\to(Y,G)$ such that for any TDS factor $\phi:(X,F)\to(Z,H)$ there is a unique factor map $\psi:(Y,G)\to(Z,H)$ with $\psi\circ\pi=\phi$. The TDS $(Y,G)$ is unique up to conjugacy. One maximal TDS factor is $\pi_R:(X,F)\to(Y/R,F_R)$, where $R\subset X\times X$ is the minimal closed $F$-invariant equivalence relation such that $F_R$ is continuous.
\end{theorem}
\[
\begin{tikzcd}
(X,F) \arrow{r}{\pi} \arrow[swap]{dr}{\phi} & (Y,G) \arrow{d}{\psi}\\
&(Z,H)
\end{tikzcd}
\]
\begin{proof}Let $\mathcal{R}$ be the family of all closed $F$-invariant equivalence relations $R$ on $X$ such that $F_R$ is a continuous map on $X/R$. Let $R$ be the intersection of all relations in $\mathcal{R}$. $R$ is also a closed $F$-invariant equivalence relation, so $X/R$ is a compact metric space and $F_R$ is well defined. We claim that $F_R$ is a continuous map on $X/R$. Indeed, if this were not so, by compactness there would be $x,x_i,y\in X$ ($i\in\N$) such that $([x_i])_{i\in\N}$ converges to $[x]_R\in X/R$ and $(F_R([x_i]_R))_{i\in\N}$ converges to $[y]_R$ but $[y]_R\neq [x]_R$. By passing to a suitable subsequence of $(x_i)_{i\in\N}$ and choosing $x,y$ properly we may also assume without loss of generality that $(x_i)_{i\in\N}$ converges to $x$ and $(F(x_i))_{i\in\N}$ converges to $y\neq x$. Choose a relation $R'\in\mathcal{R}$ such that $(x,y)\notin R'$; such a relation exists, because otherwise $(x,y)\in R$. Now $([x_i]_{R'})_{i\in\N}$ converges to $[x]_{R'}$ and $(F_{R'}([x_i]_{R'}))_{i\in\N}$ converges to $[y]_{R'}\neq [x]_{R'}$ contradicting the continuity of $F_{R'}$. Thus $\pi_R:(X,F)\to (X/R,F_R)$ is a TDS factor.

Denote $(Y,G)=(X/R,F_R)$. We claim that $\pi_R:(X,F)\to(Y,G)$ is a maximal TDS factor. Let therefore $\phi:(X,F)\to(Z,H)$ be an arbitrary TDS factor and consider the relation
\[R'=\{(x_1,x_2)\mid \phi(x_1)=\phi(x_2)\}.\]
$R'$ is an equivalence relation, and as the preimage of the closed set $\Delta_Z\subset Z\times Z$ under the continuous map $\phi\times\phi:X\times X\to Z\times Z$ it is closed. The map $\phi$ induces a homeomorphism between spaces $X/R'$ and $Z$, and via this the continuous map $H$ can be lifted to a continuous map on $X/R'$, which is equal to $F_{R'}$. Because $F_{R'}$ is continuous, it follows that $R'\in\mathcal{R}$ and instead of the factor $\phi:(X,F)\to(Z,H)$ it suffices to consider a factor of the form $\pi_{R'}:(X,F)\to (X/R',F_{R'})$, where $R'\in\mathcal{R}$. Define a map $\psi:X/R\to X/R'$ by $\psi([x]_{R})=[x]_{R'}$ for every $x\in X$; because $R$ is a refinement of $R'$, $\psi$ is well defined and it is the unique map which satisfies $\psi\circ\pi=\phi$. It is clear that $\psi$ is a continuous surjection satisfying $F_R'\circ\psi=\psi\circ F_{R}$, so $\psi$ is a factor map.

The uniqueness of $(Y,G)$ up to conjugacy follows by the standard argument.
\end{proof}

Occasionally we can construct a maximal TDS factor in a finite number of steps.

\begin{lemma}
If $R$ is an equivalence relation on a compact metric space $X$ that contains only one nontrivial equivalence class, then also $\cl{R}$ is an equivalence relation on $X$.
\end{lemma}

\begin{proof}
Clearly $\cl{R}$ is reflexive and symmetric, so it remains to prove transitivity. Let therefore $(x,y),(y,z)\in\cl{R}$. We need to prove that $(x,z)\in\cl{R}$ and we may assume that $x\neq y$, $y\neq z$. There exist $(x_i,y_i)\in R$, $(y_i',z_i)\in R$ ($i\in\N$) such that $(x_i,y_i)_{i\in\N}$ converges to $(x,y)$ and $(y_i',z_i)_{i\in\N}$ converges to $(y,z)$. Since $x\neq y$, $y\neq z$, we may assume that $x_i\neq y_i$ and $y_i'\neq z_i$ for $i\in\N$. Therefore all $y_i$ and $y_i'$ belong to the only nontrivial equivalence class of $R$ and $(y_i,y_i')\in R$. By transitivity $(x_i,z_i)\in R$ for $i\in\N$ and $(x,z)\in\cl{R}$ because it is the limit of $(x_i,z_i)_{i\in\N}$.
\end{proof}

\begin{theorem}
\label{construct}
Let $(X,F)$ be an nTDS where $X$ is a compact metric space. Define the following relations on $X$.
\begin{itemize}

\item $R_1$ contains the whole $\Delta_X$ and a point $(z_1,z_2)\in X\times X$ if there are sequences $(x_i)_{i\in\N}$ and $(y_i)_{i\in\N}$ on $X$ that converge to a common point such that $(F(x_i),F(y_i))_{i\in\N}$ converges to $(z_1,z_2)$.
\item $R_2$ is the $F$-invariant closure of $R_1$, i.e.\ $R_2=\{(F^t(x),F^t(y))\mid(x,y)\in R_1,t\in\N\}$.
\item $R_3$ is the transitive closure of $R_2$, i.e.\ $(x_1,x_n)\in R_3$ if and only if there exist $x_1,\dots,x_n\in X$ such that $(x_i,x_{i+1})\in R_2$ for $1\leq i<n$.
\item $R_4$ is the topological closure of $R_3$ in $X\times X$.
\end{itemize}
If $R_3$ has only one nontrivial equivalence class, then $R_4$ is the minimal closed $F$-invariant equivalence relation such that $F_{R_4}$ is continuous and $\pi_{R_4}:(X,F)\to(X/R_4,F_{R_4})$ is a maximal TDS factor.
\end{theorem}

\begin{proof}
It is easy to verify that $R_1$ is reflexive and symmetric, that therefore $R_2$ is reflexive, symmetric and $F$-invariant and that therefore $R_3$ is an $F$-invariant equivalence relation. By the previous lemma $R_4$ is a closed equivalence relation. To show that $R_4$ is $F$-invariant, let $(x,y)\in R_4$ and let $(x_i,y_i)\in R_3$ ($i\in\N$) be such that $(x_i,y_i)_{i\in\N}$ converges to $(x,y)$. By passing to a suitable subsequence if necessary, we may assume without loss of generality that $(F(x_i),F(y_i))_{i\in\N}$ converges to some point $(z_1,z_2)$. Since all $(F(x_i),F(y_i))$ belong to $R_3$, it follows that $(z_1,z_2)\in R_4$. Since $(x_i)_{i\in\N}$, $(y_i)_{i\in\N}$ converge to the same points as the constant sequences $(x)_{i\in\N}$, $(y)_{i\in\N}$ respectively, then the limits $(F(x),z_1)$ and $(z_2,F(y))$ of the sequences $(F(x),F(x_i))_{i\in\N}$ and $(F(y_i),F(y))_{i\in\N}$ belong to $R_1\subset R_4$. By transitivity of $R_4$ also $(F(x),F(y))$ belongs to $R_4$, which proves that $R_4$ is $F$-invariant.

To show that $F_{R_4}$ is continuous, assume to the contrary (and using the compactness of $X/R_4$) that there is a sequence $([x_i]_{R_4})_{i\in\N}$ in $X/R_4$ converging to $[x]_{R_4}$ such that $(F_{R_4}([x_i]_{R_4}))_{i\in\N}$ converges to a point $[y]_{R_4}$ with $[y]_{R_4}\neq F_{R_4}([x]_{R_4})$. By passing to a suitable subsequence and choosing $x,y\in X$ properly we may also assume that $(x_i)_{i\in\N}$ converges to $x$ and $(F(x_i))_{i\in\N}$ converges to $y$. But because $(x_i)_{i\in\N}$ and the constant sequence $(x)_{i\in\N}$ converge to a common point, it follows that the limit $(F(x),y)$ of the sequence $(F(x),F(x_i))_{i\in\N}$ belongs to $R_1\subset R_4$, a contradiction with $[y]_{R_4}\neq F_{R_4}([x]_{R_4})$.

Now let $R$ be the minimal closed $F$-invariant equivalence relation such that $F_{R}$ is continuous. Because $R_4$ is a closed $F$-invariant equivalence relation such that $F_{R_4}$ is continuous, it follows that $R\subset R_4$. To show that $R_4\subset R$, it is sufficient to prove that $R_1\subset R$. Let therefore $(z_1,z_2)\in R_1\setminus\Delta_X$ and let $(x_i)_{i\in\N},(y_i)_{i\in\N}$ be sequences converging to a common point such that $(F(x_i),F(y_i))_{i\in\N}$ converges to $(z_1,z_2)$. Since $([x_i]_R)_{i\in\N}$ and $([y_i]_R)_{i\in\N}$ converge to a common point in $X/R$, by continuity of $F_R$ it follows that the limits $[z_1]_R$ and $[z_2]_R$ of $(F_R([x_i]_R))_{i\in\N}$ and $(F_R([y_i]_R))_{i\in\N}$ are equal and $(z_1,z_2)\in R$.
\end{proof}

\begin{example}
If $(X,F)$ is a TDS, then the trivial equivalence relation $R$ on $X$ is closed, $F$-invariant and $F_R=F$ is a continuous map on $X/R=X$. Thus $\id:(X,F)\to(X,F)$ is a maximal TDS factor of $(X,F)$. \qee
\end{example} 

\begin{example}
If $(X,F)$ is an nTDS with a maximal TDS factor $\pi:(X,F)\to(Y,G)$, then even if $F$ is discontinuous at every point of $X$, it does not follow that $(Y,G)$ consists of the trivial map on the one-point space. Consider for example the system $(X,F)$ where $X=[0,1]$ and $F$ maps rational numbers to $0$ and irrational numbers to $1$, in which case $(Y,G)$ consists of the one-dimensional torus and $G$ maps every point to $0$. \qee
\end{example}

Given a sufficiently good symbolic representation of a sufficiently nice nTDS there is a simple direct construction of the maximal TDS factor as a factor of a subshift. This follows from the following theorem when $(Z,T)$ is a subshift.

\begin{theorem}
\label{symbconstruct}
Let $(X,F)$ be an nTDS with $X$ a compact metrizable space and $F$ such that for every open $U\subset X$ the set $F^{-1}(U)$ is the closure of its interior. Let $(Z,T)$ be a TDS and let $\phi:Z\to X$ be a surjective continuous map such that $F(\phi(z))=\phi(T(z))$ for $z$ in a dense set $Z'\subset Z$. If $K$ is the minimal closed $T$-invariant equivalence relation on $Z$ containing $R_\phi=\{(z_1,z_2)\in Z^2\mid \phi(z_1)=\phi(z_2)\}$ and $\pi$ is the unique map such that $\pi_K=\pi\circ\phi$, then $\pi:(X,F)\to (Z/K,T_K)$ is a maximal TDS factor of $(X,F)$.
\end{theorem}

\begin{proof}
Let $C_Z$ be the set of those TDS factors $\pi_R:(Z,T)\to(Z/R,T_R)$ such that $[z]_R=[w]_R$ for $(z,w)\in R_\phi$ and let $C_X$ be the set of TDS factors of $X$ (up to conjugacy). We will show that there is a natural one-to-one correspondence between the elements of these sets.

For one direction, let $\pi:(X,F)\to(Y,G)$ be a TDS factor of $(X,F)$ and $\psi=\pi\circ\phi$. Clearly $\psi$ is continuous and $\psi(z)=\psi(w)$ for $(z,w)\in R_\phi$, so to see that $\psi:(Z,T)=(Y,G)$ is in $C_Z$ it remains to show that $G\circ\psi=\psi\circ T$. Let $z\in Z$ and let $(z_i)_{i\in\N}$ be a sequence in $Z'$ that converges to $z$. Then
\[(G(\psi(z))=\lim_{i\to\infty}G(\psi(z_i))=\lim_{i\to\infty}\pi(F(\phi(z_i))))=\lim_{i\to\infty}\psi(T(z_i))=\psi(T(z)).\]

For the other direction, let $\pi_R:(Z,T)\to(Z/R,T_R)$ be from $C_Z$. It is easy to check that there is a unique map $\pi:X\to Z/R$ such that $\pi_R=\pi\circ\phi$. If $U\subset Z/R$ is open, the set $\phi^{-1}(\pi^{-1}(U))=\pi_R^{-1}(U)$ is also open. As a continuous map between compact metric spaces $\phi$ is a quotient map, so $\pi^{-1}(U)$ is open and $\pi$ is continuous. To see that $T_R\circ \pi=\pi\circ F$, let $x\in X$ be arbitrary and let $(U_i)_{i\in\N}$ be a sequence of open sets in $X$ whose intersection equals $F(x)$. Let $V_i$ be the interior of $F^{-1}(U_i)$. By the assumption on $F$ the closure of $V_i$ contains $x$. Let $z_i\in\phi^{-1}(V_i)\cap Z'$ such that $(\phi(z_i))_{i\in\N}$ converges to $x$. Then $(F(\phi(z_i)))_{i\in\N}$ converges to $F(x)$ and
\begin{flalign*}
&\pi(F(x))=\lim_{i\to\infty}\pi(F(\phi(z_i)))=\lim_{i\to\infty}\pi(\phi(T(z_i))) \\
&=\lim_{i\to\infty}\pi_R(T(z_i))=\lim_{i\to\infty}T_R(\pi_R(z_i))=\lim_{i\to\infty}T_R(\pi(\phi(z_i)))=T_R(\pi(x)).
\end{flalign*}s
Thus $\pi_R:(Z,T)\to(Z/R,T_R)$ is a TDS factor of $(X,F)$ and belongs to $C_Z$.

The correspondence between $C_Z$ and $C_X$ yields, starting from $\pi_K:(Z,T)\to(Z/K,T_K)$ with $K$ minimal w.r.t $C_Z$, the factor $\pi:(X,F)\to (Z/K,T_K)$ with $K$ minimal w.r.t $C_X$, i.e. this is a maximal TDS factor of $(X,F)$.
\end{proof}

\subsection{Application to $\beta$-shifts}

In this subsection we apply the results of the previous subsection to $\beta$-shifts and the corresponding multiplication map on the interval and prove the lemmas and theorems leading to the main result, Theorem~\ref{thm:BetaShiftClassificationProof}.

For $x\in\R$ we denote $\fr{x}=x-\left\lfloor x\right\rfloor$, for example $\fr{1.2}=0.2$ and $\fr{1}=0$.

\begin{definition}
For every real number $\beta>1$ we define a dynamical system $(\I,T_\beta)$, where $\I=[0,1]$ and $T_\beta(x)=\fr{\beta x}$ for every $x\in\I$.
\end{definition}

In the following let $\beta>1$ be fixed and denote $T=T_\beta$. Throughout this section $R$ denotes the $T$-invariant equivalence relation on $\I$ such that $\pi_R:(\I,T)\to(\I/R,T_R)$ is a maximal TDS factor.

\begin{lemma}\label{Teq}
Let $R'$ be the smallest equivalence relation on $\I$ which contains $(T^i(1),0)$ for all $i\in\N$, i.e. $0,1,T(1),T^2(1),\dots$ all belong to the same equivalence class, and let $R$ be the topological closure of $R'$ in $\I\times\I$. Then $\pi_R:(\I,T)\to(\I/R,T_R)$ is a maximal TDS factor. Moreover, if $T^i(1)=T^j(1)$ for some $i\neq j$, then $R'=R$.
\end{lemma}

\begin{proof}
Let $R_1,R_2,R_3,R_4$ be the relations of Theorem \ref{construct}. The only discontinuities of the map $T$ are at points $k/\beta$ where $k=1,2,\dots,\left\lfloor\beta\right\rfloor$, and near these points the values of $T$ are near $0$ or $1$, so $R_1=\Delta_\I\cup\{(0,1),(1,0)\}$. Then $R_2=\Delta_I\cup\{(0,T^i(1),(T^i(1),0)\mid i\in\N\}$, $R_3=R'$ and $R_4=R$. By Theorem \ref{construct} $\pi_R:(\I,T)\to(\I/R,T_R)$ is a maximal TDS factor. If $T^i(1)=T^j(1)$ for some $i\neq j$, then $R'\setminus\Delta_I$ is a finite set and $R'$ is a closed relation, so $R'=R$.
\end{proof}

\begin{definition}
The $\beta$-expansion of a number $x\in\I$ is the sequence $d_\beta(x)=(x_i)_{i\in\N}$ where $x_i=\left\lfloor\beta T^{i}(x)\right\rfloor$ for $i\in\N$. The digits $x_i$ belong to the alphabet $A=\{0,1,\dots,\left\lfloor\beta\right\rfloor\}$. The set of $\beta$-expansions of numbers from $[0,1)$ is denoted by $D_\beta$, and it is a shift-invariant subset of $A^\N$.
\end{definition}

\begin{definition}
\label{def:BetaShift}
The $\beta$-shift is the TDS $(S_\beta,\sigma)$, where $S_\beta$ is the topological closure of $D_\beta$ in $A^\N$ and $\sigma$ is the usual shift map.
\end{definition}

We define a special expansion for the number $1$. Let $d_\beta(1)=(t_i)_{i\in\N}$. If $d_\beta(1)$ is finite, i.e. $d_\beta(1)=t_0t_1\dots t_m 0^\infty$ for $m\in\N$ such that $t_m\neq 0$, let $d_\beta^*(1)=(t_0t_1\dots (t_m-1))^\infty$. If $d_\beta(1)$ is infinite, let $d_\beta^*(1)=d_\beta(1)$.

\begin{theorem}[{\cite[Theorem 7.2.9]{Lo02}}]
\label{loth}
Let $s\in\N^\N$. The sequence $s$ belongs to $D_\beta$ if and only if for all $p\in\N$
\[\sigma^p(s) < d_\beta^*(1)\]
and $s$ belongs to $S_\beta$ if and only if for all $p\in\N$
\[\sigma^p(s)\leq d_\beta^*(1)\]
(here $\leq$ denotes the usual lexicographic ordering of infinite sequences).
\end{theorem}

\begin{definition}
Let $\real:S_\beta\to\I$ map $\real((x_i)_{i\in\N})=\sum_{i=0}^{\infty}x_i\beta^{-(i+1)}$ and let $\real_R=\pi_R\circ\real$.
\end{definition}

\begin{lemma}
\label{multirep}
Let $s=(s_i)_{i\in\N},t=(t_i)_{i\in\N}\in S_\beta$ be two distinct sequences. Then $\real(s)=\real(t)$ if and only if $s$ and $t$ are of the form $a_0\dots a_m0^\infty$ and $a_0\dots(a_m-1)d_\beta^*(1)$ for some $m\in\N$, $a_m\neq 0$.
\end{lemma}

\begin{proof}
The ``if'' part is evidently true. To prove the other direction, assume that $\real(s)=\real(t)$. By sufficient shifting of $s$ and $t$ we may assume without loss of generality that $s_0>t_0$. Then
\[\beta^{-1}\leq (s_0-t_0)\beta^{-1}\leq (s_0-t_0)\beta^{-1}+\sum_{i\geq 1}s_i\beta^{-(i+1)}=\sum_{i\geq 1}t_i\beta^{-(i+1)}\]
and by defining the sequence $u=(u_i)_{i\in\N}$ by $u_i=t_{i+1}$ we find that $\real(u)=\sum_{i\in\N}t_{i+1}\beta^{-(i+1)}\geq 1$. Since $u$ is an element of $S_\beta$, in fact the equality $\real(u)=1$ must hold and therefore $s_0-t_0=1$ and $s_i=0$ for $i\geq 1$. To finish the proof we need to show that $u=d_\beta^*(1)$. If this were not so, then by Theorem \ref{loth} it must hold that $u<d_\beta^*(1)$. On the other hand, because $\real(u)=\real(d_\beta^*(1))$, we can repeat the argument above to show that the sequence $d_\beta^*(1)$ eventually contains only zeroes, contradicting the definition of $d_\beta^*(1)$.
\end{proof}

\begin{lemma}\label{equivariance}
If $x\in S_\beta$ and $\real(\sigma(x))\neq 1$, then $T(\real(x))=\real(\sigma(x))$.
\end{lemma}

\begin{proof}
Let $x=(x_i)_{i\in\N}$. We see that
\begin{flalign*}
T(\real(x))&=\fr{\beta\sum_{i\in\N} x_i\beta^{-(i+1)}}=\fr{x_0+\sum_{i\in\N} x_{i+1}\beta^{-(i+1)}} \\
&\overset{*}{=}\sum_{i\in\N}x_{i+1}\beta^{-(i+1)}=\real(\sigma(x)),
\end{flalign*}
where the marked equality follows from the assumption that $\sum_{i\in\N}x_{i+1}\beta^{-(i+1)}=\real(\sigma(x))\neq 1$
\end{proof}

\begin{corollary}
\label{dbcorr}
If there is a minimal $t>0$ such that $\sigma^t(d_\beta^*(1))=d_\beta^*(1)$, then
\begin{flalign*}
&T^i(1)=\real(\sigma^i(d_\beta^*(1)))\mbox{ for } i<t \mbox{ and } \\
&T^t(1)=0 \mbox{ and }\real(\sigma^t(d^*(1)))=1.
\end{flalign*}
If such a $t$ does not exist, then $T^i(1)=\real(\sigma^i(d_\beta^*(1)))$ for all $i\in\N$.
\end{corollary}

\begin{proof}
This follows by repeated application of the previous lemma if there does not exist $0<i<t$ such that $\real(\sigma^i(d_\beta^*(1)))=1$. If such an $i$ were to exist, then in any case $\sigma^i(d_\beta^*(1))\neq d_\beta^*(1)$ by the choice of $t$, so by Lemma \ref{multirep} necessarily $\sigma^i(d_\beta^*(1))$ must be the finite representation of $1$. This is impossible, because $d_\beta^*(1)$ is not a finite sequence.
\end{proof}

\begin{definition}
\label{def:BetaKernel}
We denote by $K_\beta\subset S_\beta\times S_\beta$ the minimal closed $\sigma$-invariant equivalence relation containing
\[K'_\beta=\{(x_1,x_2)\in S_\beta\times S_\beta\mid \real(x_1)=\real(x_2)\}.\]
We omit the subscripts $\beta$ when it is clear by the context.\end{definition}

\begin{lemma}
The map $\real_R:(S_\beta,\sigma)\to(\I/R,T_R)$ is a continuous factor map between TDSs with kernel equal to $K_\beta$. In particular
\[K_\beta=\{(x_1,x_2)\in S_\beta\times S_\beta\mid \real_R(x_1)=\real_R(x_2)\}\subset S_\beta\times S_\beta.\]
\end{lemma}

\begin{proof}
By Lemma \ref{equivariance} $T(\real(x))=\real(\sigma(x))$ holds for $x$ in a dense subset of $S_\beta$. Since $T$ is a piecewise continuous map on $\I$ it follows that preimages of open sets under $T$ have dense interiors. Let $\pi:(\I,T)\to(S_\beta/K,\sigma_K)$ be the maximal TDS factor given by Theorem \ref{symbconstruct} and let $\psi:(S_\beta/K,\sigma_K)\to(I/R,T_R)$ be the unique conjugacy such that $\pi_R=\psi\circ \pi$. Then
\[\real_R=\pi_R\circ\real=\psi\circ\pi\circ\real=\psi\circ\pi_K\]
is a continuous factor map between TDSs that has the same kernel as $\pi_K$.
\end{proof}

We denote
\[S'=\{x\in S_\beta\mid \real(x)\in\{0\}\cup\cl{\{T^i(1)\mid i\in\N\}}\}\]
and
\[\Delta'=\{(wa0^\infty,w(a-1)d_\beta^*(1)),(w(a-1)d_\beta^*(1),wa0^\infty)\in S_\beta^2\mid w\in A^*,a\in A\setminus\{0\}\}.\]
We think of the set $\Delta'$ as an addition to the diagonal $\Delta_{S_\beta}$ in the sense that it contains those nontrivial pairs that represent the same number.

\begin{theorem}\label{Kchar}
\[K=\Delta_{S_\beta}\cup \Delta'\cup(S'\times S').\]
\end{theorem}

\begin{proof}
By Lemma \ref{Teq} the equivalence relation $R$ just identifies the points in the closure of $\{T^i(1)\mid i\in\N\}$ with $0$. Therefore $K=K'_\beta\cup(S'\times S')$. We claim that $K'_\beta=\Delta_{S_\beta}\cup\Delta'$. To prove this, first note that the case $\real(x)=\real(y)$ occurs at least when $x=y$. If on the other hand the sequences $x,y$ are distinct, then by Lemma \ref{multirep} $\real(x)=\real(y)$ can occur exactly if $x$ and $y$ are of the form $wa0^\infty$ and $w(a-1)d_\beta^*(1)$ for some $w\in A^*$, $a\in A\setminus\{0\}$.
\end{proof}

\begin{lemma}
\label{Kint}
We have $d_\beta^*(1)=b^\infty$ for some $b\in A\setminus\{0\}$ if and only if $\beta$ is an integer. In this case $S'=\{0^\infty,d_\beta^*(1)\}$.
\end{lemma}

\begin{proof}
If $d_\beta^*(1)=b^\infty$, then $d_{\beta}(1) = (b+1)0^\omega$ so $\beta = b+1$, and vice versa.

Now suppose the base $\beta$ is an integer, so $T^i(1)=0$ for $i>0$ and $S'=\{x\in S_\beta\mid\real(x)\in\{0,1\}\}$. Our claim follows after showing that $x=d_\beta^*(1)$ is the only element of $S_\beta$ for which $\real(x)=1$. Assuming to the contrary, if $y\in S_\beta$ were another such element, then by Lemma \ref{multirep} $y=b^n(b+1)0^\infty$ for some $n\in\N$. But then $y\notin S_\beta$ by Theorem \ref{loth}, a contradiction.
\end{proof}

\begin{lemma}
\label{Ksofic}
If $d_\beta^*(1)=t_0\dots t_{m-1}(t_{m}\dots t_{m+p-1})^\infty$ with a minimal choice of $m,p>0$ (i.e. $d_\beta^*(1)$ is strictly eventually periodic with period $p$), then
\[S'=\{0^\infty\}\cup\{\sigma^i(d_\beta^*(1))\mid 0\leq i< m+p\}.\]
\end{lemma}

\begin{proof}
By Corollary \ref{dbcorr} $\cl{\{T^i(1)\mid i\in\N\}}=\{T^i(1)\mid 0\leq i<m+p\}$, so $S'=\{x\in S_\beta\mid \real(x)\in\{0\}\cup\{T^i(1)\mid 0\leq i<m+p\}\}$. Our claim follows after showing that $x=\sigma^i(d_\beta^*(1))$ is the only element of $S_\beta$ for which $\real(x)=T^i(1)$ whenever $0\leq i<m+p$. In any case by Corollary \ref{dbcorr} $x$ is one such element. Assuming to the contrary, if $y\in S_\beta$ were another such element, then by Lemma \ref{multirep} the infnite sequence $x=\sigma^i(d_\beta^*(1))$ has $d_\beta^*(1)$ as some proper suffix. But then $d_\beta^*(1)$ would also have $d_\beta^*(1)$ as a proper suffix, contradicting the assumption that $d_\beta^*(1)$ is not totally periodic.
\end{proof}

\begin{lemma}
\label{Ksft}
If $d_\beta^*(1)=(t_0\dots t_{p-1})^\infty$ with a minimal choice of $p>1$ (i.e. $d_\beta^*(1)$ is totally periodic and non-constant), then
\[S'=\{0^\infty\}\cup\{\sigma^i(d_\beta^*(1))\mid 0\leq i<p\}\cup\{\sigma^i(d_\beta(1))\mid 0<i<p\}.\]
\end{lemma}

\begin{proof}
By Corollary \ref{dbcorr} $\cl{\{T^i(1)\mid i\in\N\}}=\{T^i(1)\mid 0\leq i<p\}\cup\{0\}$, so $S'=\{x\in S_\beta\mid \real(x)\in\{0\}\cup\{T^i(1)\mid 0\leq i<p\}\}$. First note that $x=d_\beta^*(1)$ is the only element of $S_\beta$ for which $\real(x)=1$, because by Lemma \ref{multirep} the other possible representatives of $1$ would be greater in the lexicographic ordering, contradicting Theorem \ref{loth}. Now our claim follows after showing that $x=\sigma^i(d_\beta^*(1))$ and $x'=\sigma^i(d_\beta(1))$ are precisely the elements for which $\real(x)=\real(x')=T^i(1)$ whenever $0<i<p$. In any case by Corollary \ref{dbcorr} $x$ is one such element and $T^{p-i}(\real(x))=T^p(1)=0$. Since $\real(x)<1$, the sequence $d_\beta(\real(x))=(s_j)_{j\in\N}\in D_\beta$ is another representative of $\real(x)$ and it satisfies $s_j=0$ for $j\geq p-i$. By Lemma \ref{multirep} $d_\beta(\real(x))$ is of the form $t_i\dots(t_{p-1}+1)0^\infty=\sigma^i(d_\beta(1)))=x'$.
\end{proof}

The following lemma shows the general form of $S'$. Let us call $\Delta'$ the \emph{companion relation} and say $(x, y) \in \Delta' \iff x \mbox{ is a companion of } y$. Then the following says precisely that whenever $d^*_\beta(1)$ is not totally periodic, $S'$ consists of $0^\omega$ and the companions of points in the orbit closure of $d_\beta^*(1)$.

\begin{lemma}
\label{Kgen}
If $d_\beta^*(1)$ is not totally periodic, then
\begin{flalign*}
S'=&\{0^\infty\}\cup \orbc{d_\beta^*(1)} \\
\cup&\{wa0^\infty\mbox{ and }w(a-1)d_\beta^*(1) \\
&\mid wa0^\infty\in \orbc{d_\beta^*(1)}\mbox{ or }w(a-1)d_\beta^*(1)\in \orbc{d_\beta^*(1)}, w\in A^*, a\in A\setminus\{0\}\}.
\end{flalign*}
\end{lemma}

\begin{proof}
First we check that the right hand side of the equation is a subset of $S_\beta$. Clearly $\{0^\infty\}\cup \orbc{d_\beta^*(1)}\subset S_\beta$, so we concentrate on the last set in the union. First, if $wa0^\infty\in \orbc{d_\beta^*(1)}$, then $w(a-1)d_\beta^*(1)\in S_\beta$ by Theorem \ref{loth}. Then it remains to prove that if $x=w(a-1)d_\beta^*(1)\in \orbc{d_\beta^*(1)}$ then $x'=wa0^\infty\in S_\beta$. By Corollary \ref{dbcorr} $T^{\abs{w}+1}(x)=0$. Since $d_\beta^*(1)$ is not totally periodic, it follows by applying Lemma \ref{multirep} that $\real(x)<1$. Since $\real(x)<1$, the sequence $d_\beta(\real(x))=(s_i)_{i\in\N}$ is another representative of $\real(x)$ and it satisfies $s_i=0$ for $i\geq \abs{w}+1$. By Lemma \ref{multirep} $d_\beta(\real(x))\in S_\beta$ is of the form $wa0^\infty=x'$.

To prove the equality of the two given sets, we first prove the inclusion from left to right. Assume that $x\in S'$. If $\real(x)=0$ then $x=0^\infty$. Otherwise there is a sequence $(t_i)_{i\in\N}\in \N^\N$ such that $(T^{t_i}(1))_{i\in\N}$ tends to $\real(x)$, and by Corollary \ref{dbcorr} $(T^{t_i}(1))_{i\in\N}=(\real(\sigma^{t_i}(d_\beta^*(1))))_{i\in\N}$. By passing to a subsequence we may assume that $(\sigma^{t_i}(d_\beta^*(1)))_{i\in\N}$ converges to an element $x'\in\orbc{d_\beta^*(1)}$, and by continuity of the map $\real$ it follows that $\real(x)=\real(x')$. If $x=x'$, we are done. Otherwise by Lemma \ref{multirep} $x$ and $x'$ are of the form $wa0^\infty$ and $w(a-1)d_\beta^*(1)$ and we are done.

Finally we prove the inclusion from right to left. If $x=0^\infty$, then $\real(x)=0$ and $x\in S'$. If $x\in\orbc{d_\beta^*(1)}$, then there is a sequence $(t_i)_{i\in\N}\in \N^\N$ such that $(\sigma^{t_i}(d_\beta^*(1)))_{i\in\N}$ tends to $x$. By Corollary \ref{dbcorr} $(T^{t_i}(1))_{i\in\N}=(\real(\sigma^{t_i}(d_\beta^*(1))))_{i\in\N}$, and by continuity of the map $\real$ it follows that this sequence converges to $\real(x)$ which is therefore an element of $\cl{\{T^i(1)\mid i\in\N\}}$ and thus $x\in S'$. If $x\in\{wa0^\infty,w(a-1)d_\beta^*(1)\}$ and either $wa0^\infty\in \orbc{d_\beta^*(1)}$ or $w(a-1)d_\beta^*(1)\in \orbc{d_\beta^*(1)}$, then one of these sequences is in $S'$ and thus both of them are in $S'$.
\end{proof}

\begin{theorem}[{\cite[Theorems 7.2.13, 7.2.15]{Lo02}}]
\label{thm:LothaireSFTSofic}
$S_\beta$ is an SFT if and only if $d_\beta(1)$ is finite (or equivalently $d_\beta^*(1)$ is totally periodic). $S_\beta$ is a sofic shift if and only if $d_\beta(1)$ is strictly eventually periodic (or equivalently $d_\beta^*(1)$ is strictly eventually periodic).
\end{theorem}

\begin{theorem}
\label{thm:Kint}
If $\beta>1$ is an integer, then $K$ is a relative SFT in $S_\beta\times S_\beta$.
\end{theorem}

\begin{proof}
Let $d_\beta^*(1)=b^\infty$, $b\in A\setminus\{0\}$. The set $K$ has been characterized in Theorem \ref{Kchar} and Lemma \ref{Kint}. We will define a relative SFT $S\subset S_\beta\times S_\beta$ by giving a set of allowed patterns $P\subset(A\times A)^{3}$ with the intention of making $S$ equal to $K$.

For every $w\in A^{3}$ let $(w,w)\in P$, which guarantees that $\Delta_{S_\beta}\subset S$. To guarantee that $\Delta'\subset S$, let $w\in A^{2}$, $a\in A\setminus\{0\}$ and add every subword of length $3$ from the infinite words $(wa0^\infty,w(a-1)d_\beta^*(1))$ and $(w(a-1)d_\beta^*(1),wa0^\infty)$ to $P$. This also guarantees that $S'\times S'\in S$, since $S'=\{0^\infty, b^\infty\}$. In total we have guaranteed that $K\subset S$.

We claim that $S\subset K$. To see this, let $x=(x_i)_{i\in\N}$, $y=(y_i)_{i\in\N}$ be distinct sequences such that $(x,y)\in S$ and let $i\in\N$ be minimal such that $x_i\neq y_i$. It follows that $x_j\neq y_j$ for all $j\geq i$, because in all pairs $(v,u)\in P$ a pair of distinct letters in the beginning is followed by another pair of distinct letters. If $x_i=x_{i+1}$, then without loss of generality $x_j=0$ and $y_j=b$ for $j\geq i$, because the only pairs $(v,u)\in P$ that start with a pair of distinct letters and where the first two letters of $v$ are the same are of the form $(000,bbb)$ or $(bbb,000)$. In this case $i=0$, because $P$ does not contain a pair of the form $(a00,abb)$ ($a\in A$), so $(x,y)=(0^\infty,b^\infty)\in S'\times S'\subset K$. Consider then the case $x_i\neq x_{i+1}$. Without loss of generality $x_i x_{i+1}x_{i+2}=a00$ and $y_i y_{i+1}y_{i+2}=(a-1)bb$ $(a\in A\setminus\{0\})$, because the only pairs $(v,u)\in P$ that start with a pair of distinct letters and where the first two letters of $v$ are distinct are of the form $(a00,(a-1)bb)$ or $((a-1)bb,a00)$. Because $x_{i+1}=x_{i+2}$, it follows as above that $x_i x_{i+1}x_{i+2}\dots=a0^\infty$ and $y_i y_{i+1}y_{i+2}\dots=(a-1)b^\infty$. By the choice of $i$ it follows that $x=wa0^\infty$ and $y=w(a-1)b^\infty$ for some $w\in A^*$ and thus also in this case $(x,y)\in \Delta'\subset K$.
\end{proof}

\begin{theorem}
\label{thm:KIfSSofic}
If $S_\beta$ is a non-SFT sofic shift, then $K$ is not a relative SFT in $S_\beta\times S_\beta$ but it is a sofic shift.
\end{theorem}

\begin{proof}
Let $d_\beta^*(1)=t=(t_i)_{i\in\N}=t_0\dots t_{m-1}(t_m\dots t_{m+p-1})^\infty$ with a minimal choice of $m,p>0$. The set $K$ has been characterized in Theorem \ref{Kchar} and Lemma \ref{Ksofic}.

Now assume that $K$ is a relative SFT which is defined by a set of allowed patterns $P\subset(A\times A)^M$ for some $M>1$ such that $M-1$ is divisible by $p$. We have
\[ (t_{m-1}(t_m\dots t_{m+p-1})^\infty,t_{m+p-1}(t_m\dots t_{m+p-1})^\infty)=(\sigma^{m-1}(d_\beta^*(1)),\sigma^{m+p-1}(d_\beta^*(1)))\in S'\times S'\subset K \]
and by minimality of $m$ it follows that $t_{m-1}\neq t_{m+p-1}$, so at the very least $P$ must contain the pairs $(a(t_m\dots t_{m+p-1})^{(M-1)/p},a'(t_m\dots t_{m+p-1})^{(M-1)/p})$ for some $a\in A\setminus\{0\}$, $a'\in A$ ($a\neq a'$) and $(w,w)$ for all $w\in A^M$. Let $x=a(t_m\dots t_{m+p-1})^{(M-1)/p} 0^\infty$ and $y=a'(t_m\dots t_{m+p-1})^{(M-1)/p}0^\infty$. We have $(x,y)\in K=\Delta_{S_\beta}\cup\Delta'\cup (S'\times S')$ because all its factors of length $M$ are from $P$. Clearly $(x,y)\notin \Delta_{S_\beta}$. Also $(x,y)\notin\Delta'$, because this would require either $x$ or $y$ to consist of only zeroes after the first coordinate where they disagree, but $(t_m\dots t_{m+p-1})$ contains at least one non-zero digit. The only remaining possibility is $(x,y)\in S'\times S'$. This is also impossible, because $0^\infty$ is the only sequence in $S'$ that ends with all zeroes.

To prove that $K$ is a sofic shift, we construct a finite automaton $(Q,(A\times A),\delta)$ that almost recognizes $K$. The state set is $Q=\{q_I,q_{-,i},q_{+,i}\mid 0\leq i<m+p\}$ with transitions
\begin{flalign*}
&\delta(q_I,(a,a))=q_I \\
&\delta(q_I,(a',a'-1))=q_{-,0}\qquad\mbox{and}\qquad\delta(q_I,(a'-1,a'))=q_{+,0} \\
&\delta(q_{-,i},(0,t_i))=q_{-,i+1}\qquad\mbox{and}\qquad\delta(q_{+,i},(t_i,0))=q_{+,i+1} \\
&\delta(q_{-,m+p-1},(0,t_{m+p-1}))=q_{-,m}\qquad\mbox{and}\qquad\delta(q_{+,m+p-1},(t_{m+p-1},0))=q_{+,m}
\end{flalign*}
for $a\in A, a'\in A\setminus\{0\}$ and $0\leq i<m+p-1$. Let $L$ be the intersection of the subshift recognized by the automaton with the sofic shift $S_\beta\times S_\beta$. Clearly $L$ contains the elements of $\Delta_{S_\beta}\cup\Delta'$ and all pairs $(x,y)$ such that $x=0^\infty$ and $y$ is a shift of $d_\beta^*(1)$ or vice versa. The remaining set $K\setminus L\subset S'\times S'$ is finite and therefore sofic, so $K$ is also sofic as the union $L\cup (K\setminus L)$ of two sofic shifts.
\end{proof}

\begin{theorem}
\label{thm:KIfSSFT}
If $S_\beta$ is an SFT and $\beta$ is not an integer, then $K$ is not a relative SFT in $S_\beta\times S_\beta$ but it is a sofic shift.
\end{theorem}

\begin{proof}
Let $d_\beta^*(1)=(t_0\dots t_{p-1})^\infty$ with a minimal choice of $p>1$. The set $K$ has been characterized in Theorem \ref{Kchar} and Lemma \ref{Ksft}.

Now assume that $K$ is a relative SFT which is defined by a set of allowed patterns $P\subset(A\times A)^m$ for some $m>1$. We have $(0^\infty,(t_{p-1}+1) 0^\infty)=(0^\infty,\sigma^{p-1}(d_\beta(1)))\in S'\times S'\subset K$, so at the very least $P$ must contain the pairs $(0^m,(t_{p-1}+1)0^{m-1})$ and $(w,w)$ for all $w\in A^m$. As an element of $S_\beta$, $\sigma^{p-1}(d_\beta(1))$ is lexicographically smaller than $d_\beta^*(1)$ and for all sufficiently large $i\in\N$ it also holds that $(t_{p-1}+1)0^{m+i}10^\infty<d_\beta^*(1)$. Evidently all proper shifts of this sequence are smaller in lexicographic order than $d_\beta^*(1)$, so $(t_{p-1}+1)0^{m+i}10^\infty\in S_\beta$. For all such $i$ the relation $K$ contains pairs $(00^{m+i}10^\infty,(t_{p-1}+1) 0^{m+i}10^\infty)$, because $(0^m,t_p 0^{m-1})\in P$. By their form these infinitely many pairs are not elements of $\Delta_{S_\beta}\cup \Delta'$, so they must be contained in the finite set $S'\times S'$, a contradiction.

To prove that $K$ is a sofic shift, we construct a finite automaton $(Q,(A\times A),\delta)$ that almost recognizes $K$. The state set is $Q=\{q_I,q_{-,i},q_{+,i}\mid 0\leq i< p\}$ with transitions
\begin{flalign*}
&\delta(q_I,(a,a))=q_I \\
&\delta(q_I,(a',a'-1))=q_{-,0}\qquad\mbox{and}\qquad\delta(q_I,(a'-1,a'))=q_{+,0} \\
&\delta(q_{-,i},(0,t_i))=q_{-,i+1}\qquad\mbox{and}\qquad\delta(q_{+,i},(t_i,0))=q_{+,i+1} \\
&\delta(q_{-,p-1},(0,t_{p-1}))=q_{-,0}\qquad\mbox{and}\qquad\delta(q_{+,p-1},(t_{p-1},0))=q_{+,0}
\end{flalign*}
for $a\in A, a'\in A\setminus\{0\}$ and $0\leq i<p-1$. Let $L$ be the intersection of the subshift recognized by the automaton with the sofic shift $S_\beta\times S_\beta$. Clearly $L$ contains the elements of $\Delta_{S_\beta}\cup\Delta'$ and all pairs $(x,y)$ such that $x=0^\infty$ and $y$ is a shift of $d_\beta^*(1)$ or vice versa. The remaining set $K\setminus L\subset S'\times S'$ is finite and therefore sofic, so $K$ is also sofic as the union $L\cup (K\setminus L)$ of two sofic shifts.
\end{proof}

\begin{corollary}
If $\beta>1$ is not an integer and $S_\beta$ is a sofic shift, i.e. if $d_\beta^*(1)$ is non-constant and eventually periodic, then the TDS $(\I/R,T_R)$ belongs to the class \ttfrac{SFT}{sofic} but not in \ttfrac{SFT}{SFT}.
\end{corollary}

The following is obvious from Theorem~\ref{loth}.

\begin{lemma}
\label{lem:SSU}
For any $\beta > 1$ there exists $m \in \N$ such that $(0^m0^*1)^\N \subset S_\beta$.
\end{lemma}

\begin{lemma}
\label{lem:DenseNotSP}
If the orbit of $d^*_\beta(1)$ is dense in $S_\beta$, then it is not eventually periodic.
\end{lemma}

\begin{proof}
If $d^*_\beta(1)$ is eventually periodic and dense in $S_\beta$, then $S_\beta$ is finite. This can never happen by the previous lemma.
\end{proof}

\begin{theorem}
\label{thm:DenseCase}
If the orbit of $d^*_\beta(1)$ is dense in $S_\beta$, then $K_\beta$ is a relative SFT in $S_\beta \times S_\beta$. In fact, $K_\beta = S_\beta \times S_\beta$ so $K_\beta$ is the relative full shift.
\end{theorem}

\begin{proof}
Recall that we have $K_\beta = \Delta_{S_\beta} \cup \Delta' \cup (S' \times S')$ where by Lemma~\ref{Kgen} we have in particular $\orbc{d_\beta^*(1)} \subset S'$. It follows that
\[ S_\beta \times S_\beta \subset S' \times S' \subset K_\beta \subset S_\beta \times S_\beta. \]
\end{proof}

\begin{lemma}
\label{lem:Pumping}
Suppose $\Sigma$ is a finite totally ordered alphabet, suppose $x \in \Sigma^\N$ is not eventually periodic and $\sigma^n(x) \leq_{\mathrm{lex}} x$ for all $n$, let $L$ be any closed $\omega$-automatic language with $x \in L$. Then there exists a point $z \neq x$ in $L$ which 
satisfies $\sigma^n(z) \leq_{\mathrm{lex}} x$ for all $n$.
\end{lemma}

\begin{proof}
By the pumping lemma, there exist $w, u, v \in \Sigma^*, |u|, |v| \geq 1, y \in \Sigma^\omega$, such that $x = wuvy$ and $wuy, wuvvy \in L$.

Since $x$ is not eventually periodic, we cannot have $v y = y$. Suppose first $vy <_{\mathrm{lex}} y$. Then pick $z = wuvvy \in L$. For $n \leq |wuv|$ we have
\[ \sigma^n(z) = \sigma^n(wuvvy) <_{\mathrm{lex}} \sigma^n(wuvy) = \sigma^n(x) \leq_{\mathrm{lex}} x, \]
and for any $n$ we have
\[ \sigma^{|wuv| + n}(z) = \sigma^{|wuv| + n}(wuvvy) = \sigma^{n}(vy) = \sigma^{|wu|+n}(x) \leq_{\mathrm{lex}} x, \]
and clearly $z \neq x$, concluding the proof.

Suppose then $y <_{\mathrm{lex}} vy$. Then let $z = wuy$, and similarly as above for $n \leq |w u|$ we have
\[ \sigma^n(z) = \sigma^n(wuy) <_{\mathrm{lex}} \sigma^n(wuvy) \leq_{\mathrm{lex}} x \]
and for all $n$ we have
\[ \sigma^{|wu|+n}(z) = \sigma^{|wu| + n}(wuy) = \sigma^n(y) = \sigma^{|wuv|+n}(x) \leq_{\mathrm{lex}} x, \]
so $z \in S_\beta$ and again clearly $z \neq x$.
\end{proof}

\begin{theorem}
\label{thm:NotDenseNotSofic}
If $d^*_\beta(1)$ is aperiodic and is not dense in $S_\beta$, then $K_\beta$ is not a sofic subshift of $S_\beta \times S_\beta$.
\end{theorem}

\begin{proof}
Again Lemma~\ref{Kgen} applies so $K_\beta = \Delta_{S_\beta} \cup \Delta' \cup (S' \times S')$ where $\Delta'$ is the companion relation and $S'$ contains $0^\omega$ and the companions of the points in the orbit closure of $d^*_\beta(1)$. Suppose $K_\beta$ is sofic in $S_\beta^2$ and let $\mathcal{A}$ be an automaton for it.

Since $d^*_\beta(1)$ is not dense, for some $n\in\N$ there exists a word $w \in A^n$ and a point $x = wy \in S_\beta$ such that $x \notin \orbc{d^*_\beta(1)}$ and $w \neq d^*_\beta(1)_{[0, n-1]}$. We can choose $x = w0^m10^m10^\infty$ for any large enough $m$ by Lemma~\ref{lem:SSU}. Namely, in addition to the $m$ given by the lemma, it is enough to pick $m$ large enough so that the first $n$ shifts are smaller than $d^*_\beta(1)$ in lexicographical order. Pick $m$ also large enough so that $d^*_\beta(1)$ does not begin with $t0^m1$ for any suffix $t$ of $w$.

Now since $\Delta' \subset K_\beta$ we have
\[ (w0^m10^m10^\infty, w0^m10^m0d^*_\beta(1)) \in K_\beta \]
for all large enough $m$. But then we also have
\[ (w0^m10^m10^\infty, w0^m10^m0y) \in K_\beta \]
for some
$y \neq d^*_\beta(1)$,
by applying Lemma~\ref{lem:Pumping} with $L$ the language of those $z$ such that the pair $(w0^m10^m10^\infty, w0^m10^m0z)$ is accepted by the automaton $\mathcal{A}$, noting that $\sigma^n(z) \leq_{\mathrm{lex}} d^*_\beta(1)$ for all $n$ implies $w0^m10^m0z \in S_\beta$ so indeed $(w0^m10^m10^\infty, w0^m10^m0z) \in K_\beta$.

We clearly do not have $(w0^m10^m10^\infty, w0^m10^m0y) \in \Delta_{S_\beta} \cup \Delta'$, so $w0^m10^m10^\infty, w0^m10^m0y \in S'$. Thus a companion of $w0^m10^m0y$ is in the orbit closure of $d^*_\beta(1)$. By the choice of $m$, all companions of $w0^m10^m0y$ necessarily begin with $w$, so $w \sqsubset d^*_\beta(1)$, a contradiction.
\end{proof}

\section{Case study 2: toral automorphisms}
\label{sec:ToralAutomorphisms}

In this section, we study toral automorphisms, a class of expansive systems, in the sofic-per-sofic framework. By expansivity, these systems are ultimately SFT-per-SFT, but we find that nevertheless the more general framework is useful. In particular, the final kernel we obtain is the composition of two SFT relations, and SFT relations are not in general closed under transitivity (Example~\ref{ex:SFTNotTrans}), so sofic shifts are a better category for the intermediate constructions.

\subsection{\ttfrac{SFT}{SFT}-systems on the two-dimensional torus}

In the two-dimensional case, the structure of expansive toral automorphisms is well-understood. We prove that they can be fully described by finite-state automata.

\begin{theorem}
\label{thm:TorusClassificationProof}
Every expansive homeomorphism on the two-dimensional torus is conjugate to a \ttfrac{$\Z$-SFT}{SFT}. The class of \ttfrac{$\Z$-SFT}{SFT} dynamical systems (represented as pairs $\ccfrac{Y}{Z}$) that are systems on the two-dimensional torus is recursively enumerable.
\end{theorem}

\begin{proof}
By Theorem~\ref{thm:TorusExpansiveClassification}, every such homeomorphism is conjugate to an Anosov diffeomorphism, and by Theorem~\ref{thm:Manning}, every Anosov diffeomorphism on the torus is conjugate to a hyperbolic toral automorphism. Every hyperbolic toral automorphism is \ttfrac{SFT}{SFT} by Theorem~\ref{thm:HyperbolicToralSFTSFT}. This proves the first claim.

For the second, observe that the set of invertible integer matrices defining hyperbolic toral automorphisms is recursively enumerable, by simply enumerating matrices and checking the eigenvalue condition. By Theorem~\ref{thm:HyperbolicToralSFTSFT}, we can then list a sequence of systems $X_1/K_1, X_2/K_2, X_3/K_3, \dots$, such that for every hyperbolic toral automorphism, some \ttfrac{SFT}{SFT} representation of it is listed.

To list all \ttfrac{SFT}{SFT} systems conjugate to expansive homeomorphism on the two-dimensional torus, it is then enough to list all \ttfrac{SFT}{SFT} systems which are conjugate to one of the $X_i/K_i$. Conjugacy of \ttfrac{SFT}{SFT} systems is semidecidable by Lemma~\ref{lem:SFTSFTConjugacySemidecidable}, which concludes the proof.
\end{proof}

Our quest is to recursively enumerate every \ttfrac{SFT}{SFT} (as a finite list of words defining the subshift and the underlying relation) which represents a dynamical system on the two-dimensional torus. These are all expansive by Fried's result.

\begin{theorem}[\cite{Le89}]
\label{thm:TorusExpansiveClassification}
If a homeomorphism $f:\T^2\to\T^2$ is expansive, then it is conjugate to an Anosov diffeomorphism on the torus.
\end{theorem}

\begin{theorem}[\cite{Ma74}]
\label{thm:Manning}
If $f:\T^2\to\T^2$ is an Anosov diffeomorphism, then it is conjugate to a hyperbolic toral automorphism.
\end{theorem}

Since every hyperbolic toral automorphism has an SFT cover, we see that the class of \ttfrac{SFT}{SFT}-systems on $\T^2$ is (up to conjugacy) exactly the set of hyperbolic toral automorphisms on $\T^2$. Using homoclinic points it is simple to construct SFT covers for such automorphisms. We recall this construction, following the exposition of \cite{Ei01}.

Let $T:\T^2\to \T^2$ be a hyperbolic toral automorphism and let $A\in GL(2,\Z)$ be a matrix which determines $T$. It has eigenvalues $\lambda,\mu$ ($\abs{\lambda}>1$ and $\abs{\mu}<1$) with associated eigenlines $l_\lambda,l_\mu$. These lines have irrational slopes, i.e. they intersect $\Z^2$ only at the origin, because a hypothetical lattice point intersecting an eigenline would get arbitrarily close to the origin under action of $T$, yet must stay an integer vector by integrality of the matrix.

Let $\pi:\R^2\to (\R/\Z)^2$ be the usual projection from the plane to the torus. Let $v_\lambda\in l_\lambda$ and $v_\mu\in l_\mu$ be vectors such that $w^\Delta\doteq v_\lambda=\left(\begin{smallmatrix} 0 \\ 1 \end{smallmatrix}\right)+v_\mu$. For every $i\in\Z$, $n\in\Z$ the following hold:
\begin{flalign*}
&T^i(\pi(nw^\Delta))=T^i(\pi(nv_\lambda))=\pi(A^i nv_\lambda)=\pi(n\lambda^i v_\lambda) \\
&T^i(\pi(nw^\Delta))=T^i(\pi(\left(\begin{smallmatrix} 0 \\ n \end{smallmatrix}\right)+nv_\mu))=\pi(A^i(\left(\begin{smallmatrix} 0 \\ n \end{smallmatrix}\right)+ nv_\mu))=\pi(A^i nv_\mu)=\pi(n\mu^i v_\mu)
\end{flalign*}
The point $\pi(nw^\Delta)$ is a homoclinic point of the map $T$, because by the first equation $T^i(\pi(nw^\Delta))$ tends to zero when $i$ tends to $-\infty$ and by the second equation $T^i(\pi(nw^\Delta))$ tends to zero when $i$ tends to $\infty$.

Let $\Sigma\subset\Z$ be a finite alphabet and for every $x\in\Sigma^\Z$ define
\[\phi(x)=\sum_{i\in\Z}T^{-i}(\pi(x_i w^\Delta))=\pi\left(\sum_{i<0}x_i\mu^{-i}v_\mu+\sum_{i\geq 0}x_i\lambda^{-i}v_\lambda\right)\]
For any subshift $X\subset \Sigma^\Z$ the restriction $\phi:X\to \T^2$ is clearly continuous and satisfies $\phi\circ\sigma=T\circ\phi$.

\begin{lemma}
\label{lem:FullShiftFactorMap}
Let $I\subset\R$ be a nondegenerate interval, $0\in I$. If $\pi(Iv_\lambda)\subset \phi(X)$, then $\phi:(X,\sigma)\to (\T^2,T)$ is a factor map.
\end{lemma}

\begin{proof}We need to show that $\phi(X)=\T^2$. First we see that
\begin{flalign*}
\phi(X)&=\bigcup_{i\in\N}\phi(\sigma^i(X))=\bigcup_{i\in\N}T^i(\phi(X)) \\
&\supseteq \bigcup_{i\in\N}T^i(\pi(Iv_\lambda))=\bigcup_{i\in\N}\pi(A^i(Iv_\lambda))=\bigcup_{i\in\N}\pi(\lambda^i Iv_\lambda)\supseteq\pi(\R_+ Iv_\lambda).
\end{flalign*}
Because the eigenline $l_\lambda$ has an irrational slope, it follows that $\phi(X)$ is dense in $\T^2$. By continuity $\phi(X)$ is also closed, so $\phi(X)=\T^2$.
\end{proof}

From the lemma it follows that $\phi:(X,\sigma)\to(\T^2,T)$ is a factor map at least when $X=\Sigma_n^\Z$ and $\Sigma_n=\{0,1,\dots,n-1\}$ with any sufficiently large $n$. In the following we denote $\pi_\lambda(x)=\sum_{i\geq 0}x_i\lambda^{-i}$ and $\pi_\mu(x)=\sum_{i<0}x_i\mu^{-i}$.

The kernel relation is $K\doteq\{(x,y)\in\Sigma_n^\Z\mid \phi(x)=\phi(y)\}$. Equivalently, $(x,y)\in K$ if and only if $\pi_\lambda(x-y)v_\lambda+\pi_\mu(x-y)v_\mu\in\Z^2$, where $x-y$ denotes coordinatewise subtraction of digits in configurations. We wish to find effectively a finite set of forbidden patterns which determines the relation $K$. We begin by giving an effective construction of a finite automaton that determines $K$. The following theorem is has been proved in the case $\beta>1$ and $\xi=0$ as Propositions 2.3.30 and 2.3.33 in \cite{BeRi10}.

\begin{theorem}
\label{thm:Carry}
Let $\beta\in\R$ be a root of a monic polynomial with integral coefficients whose conjugates have absolute values strictly less than $1$, $\abs{\beta}>1$ and let $\xi\in\Q(\beta)$. 
Then the language
\[L_{\beta,\xi}=\left\{x\in(\Sigma_n\cup(-\Sigma_n))^{\N_+}\mid \sum_{i>0}x_i\beta^{-i}=\xi\right\}\subset (\Sigma_n\cup(-\Sigma_n))^{\N_+}\]
is effectively $\omega$-automatic.
\end{theorem}

\begin{proof}
Let $\Sigma=\Sigma_n\cup(-\Sigma_n)$, let $q>0$ be an integer such that $\xi\in \Z[\beta]/q$ and let $Q=\Z[\beta]/q\cap\left[-\frac{n-1}{\abs{\beta}-1},\frac{n-1}{\abs{\beta}-1}\right]$. For given $w=w_0\dots w_{k-1}\in\Sigma^*$ and $x\in\Sigma^{\N_+}$ we will denote $\rho(w)=w_0\beta^{k-1}+\dots+w_{k-1}$ and $\rho(x)=\sum_{i=1}^\infty x_i\beta^{-i}$. Define the automaton $A^\infty=(\Z[\beta]/q,\Sigma,E,\{-\xi\})$ ($\Z[\beta]/q$ is the state set, $\Sigma$ is the input alphabet, $E$ is the set of transitions and $\{-\xi\}$ is the initial state) with transitions $s\overset{a}{\to} t$ if and only if $t=\beta s+a$. Let $A=A_{\beta,\xi}$ be the restriction of $A^\infty$ to the accessible part $Q_0$ of the state set $Q$. We say that $A$ accepts the sequence $x\in \Sigma^{\N_+}$ if it determines an infinite path in $A$ starting from the state $-\xi$. If $-\xi\notin Q$, the interpretation is that $A$ rejects all input sequences.

We can show that the word $w$ labels a path $-\xi\to s$ in $A^\infty$ if and only if $\rho(w)=\xi \beta^\abs{w}+s$ by induction on the length of $\abs{w}$. The case $\abs{w}=0$ is clear, so let us assume that the claim holds for a path $w$ and consider a word $wa$ ($w\in\Sigma^*,a\in\Sigma$) such that $w$ labels a path from $-\xi$ to $s$ and $a$ labels an edge from $s$ to $r$. Then $\rho(wa)=\rho(w)\beta+\rho(a)=(\xi\beta^{\abs{w}}+s)\beta+a=(\xi\beta^{\abs{w}}+s)\beta+(r-\beta s)=\xi\beta^{\abs{wa}}+r$.

Now assume that $x$ labels an infinite computation in $A$. For every $k\geq 1$, $\abs{\rho(x_0\dots x_{k-1})-\xi\beta^{k}}=\abs{s}\leq\frac{n-1}{\abs{\beta}-1}$ for some $s\in Q$ and $\abs{\beta^{-k}\rho(x_0\dots x_{k-1})-\xi}\leq \frac{(n-1)\abs{\beta^{-k}}}{\abs{\beta}-1}$, from which it follows that $\rho(x)=\lim_{k\to\infty}\beta^{-k}\rho(x_0\dots x_{k-1})=\xi$.

Conversely, let $x$ be a path which does not label a computation in $A$, i.e. there is a prefix $w$ of $x$ which determines a path from $-\xi$ to $t$ on $A^\infty$ with $\abs{t}>\frac{n-1}{\abs{\beta}-1}$. Let us assume that $\beta^{-\abs{w}}t>\beta^{-\abs{w}}\frac{n-1}{\abs{\beta}-1}$ (the case $\beta^{-\abs{w}}t<-\beta^{-\abs{w}}\frac{n-1}{\abs{\beta}-1}$ being symmetric). If $\beta>1$ let $z=(-n+1)^\infty$. If $\beta<-1$ and $\abs{w}$ is even let $z=((n-1)(-n+1))^\infty$. If $\beta<-1$ and $\abs{w}$ is odd let $z=((-n+1)(n-1))^\infty$. We have
\begin{flalign*}
\rho(x)&\geq \rho(wz)=\beta^{-\abs{w}}(\rho(w)+\rho(z)) \\
&\geq\beta^{-\abs{w}}\left(\xi \beta^\abs{w}+t-\frac{n-1}{\abs{\beta}-1}\right)>\xi.
\end{flalign*}
Thus $A$ accepts exactly those sequences that represent the number $\xi$.

Finally, we will show that the accessible part $Q_0$ is finite. Let $M$ be the minimal polynomial of $\beta$ and let $\beta_1=\beta,\beta_2,\dots,\beta_g$ be the roots of $M$. On the discrete lattice $(\Z/q)[X]/\left<M\right>\simeq (\Z/q)[\beta]$ we may define a norm by $\vert\vert P(X)\vert\vert=\max_{i}\abs{P(\beta_i)}$. Let $s\in Q_0$ and let $w$ be a path from $-\xi$ to $s$ in $A$. Then $s=s(\beta)=\rho(w)-\xi\beta^{\abs{w}}$ with $s(X)\in(\Z/q)[X]/\left<M\right>$ and $\abs{s}\leq \frac{n-1}{\abs{\beta}-1}$. By assumption $\abs{\beta_i}<1$ for $i\neq 1$ and therefore $\abs{s(\beta_i)}\leq \abs{\frac{n-1}{1-\abs{\beta_i}}}+\abs{\xi\beta_i^\abs{w}}\leq \abs{\frac{n-1}{1-\abs{\beta_i}}}+\abs{\xi}$. Thus the states of $Q_0$ are bounded in norm and so $Q_0$ is finite.
\end{proof}

For the following, note that $\lambda$ and $\mu^{-1}$ are roots of monic polynomials with integer coefficients and they have absolute values $>1$.

\begin{lemma}
The subshift
\[L=\left\{x \in (\Sigma_n\cup(-\Sigma_n))^\Z\mid \pi_\lambda(x)v_\lambda+\pi_\mu(x)v_\mu\in \Z^2\right\}\subset (\Sigma_n\cup(-\Sigma_n))^\Z\]
is effectively sofic.
\end{lemma}

\begin{proof}
The language is clearly closed and shift-invariant, so it is enough to construct a $\Z$-automaton which accepts this language by Lemma~\ref{lem:AutomaticSoficConnection}.

Because $L$ is a set of configurations over a bounded digit set, we can find a finite set $F\subset \Z^2$ such that $\pi_\lambda(x)v_\lambda+\pi_\mu(x)v_\mu\in \Z^2$ if and only if $\pi_\lambda(x)v_\lambda+\pi_\mu(x)v_\mu\in F$. For every $k\in F$ there are unique $\xi(k,\lambda),\xi(k,\mu)\in\Q(\lambda)=\Q(\mu)$ such that $\pi_\lambda(x)v_\lambda+\pi_\mu(x)v_\mu=k$ if and only if $\pi_\lambda(x)=\xi(k,\lambda)$ and $\pi_\mu(x)=\xi(k,\mu)$.

Denote by $A_{\beta,\xi}$ the automaton recognizing the language $L_{\beta,\xi}$ of Theorem~\ref{thm:Carry}. For each $k$ let $A_{k,1}$ be the automaton constructed from $A_{\mu^{-1},\xi(k,\mu)}$ by reversing the direction of all arrows and let $A_{k,2}=A_{\lambda,\lambda^{-1}\xi(k,\lambda)}$. Construct $A_k$ from the union of $A_{k,1}$ and $A_{k,2}$ by adding an $\epsilon$-transition from the initial state of $A_{k,1}$ to the initial state of $A_{k,2}$. The automaton $A_L=\bigcup_{k\in F}A_k$ has $L$ as the set of labels of infinite paths that cross over the $\epsilon$-transition at the zero-coordinate. Since $L$ is a subshift, it is the set of labels of all infinite paths on $A_L$.
\end{proof}

\begin{lemma}
\label{lem:Kkernel}
Let $X \subset (\Sigma_n\cup(-\Sigma_n))^\Z$ be any sofic shift. Then the subshift 
\[ K = \left\{(x,y) \in X^2 \mid \pi_\lambda(x-y)v_\lambda+\pi_\mu(x-y) v_\mu \in \Z^2\right\}\subset X^2 \]
is effectively sofic.
\end{lemma}

\begin{proof}
Let $L$ be the sofic shift in the previous lemma. Then $K = \pi^{-1}(L) \cap X^2$ where $\pi : ((\Sigma_n\cup(-\Sigma_n))^\Z)^2 \to (\Sigma_n\cup(-\Sigma_n))^\Z$ is the block map $\pi(x, y)_i = x_i - y_i$, so the result follows from basic closure properties of sofic shifts.
\end{proof}

\begin{theorem}
\label{thm:HyperbolicToralSFTSFT}
Every hyperbolic toral automorphism $T : \T^2 \to \T^2$ is effectively \ttfrac{SFT}{SFT}.
\end{theorem}

\begin{proof}
By Lemma~\ref{lem:FullShiftFactorMap}, $\phi : (X, \sigma) \to (\T^2, T)$ is a factor map for $X = \Sigma^\Z_n$ for any large enough $n$, and concretely $n = \lfloor \beta \rfloor$ suffices. Lemma~\ref{lem:Kkernel} shows that we can effectively compute the kernel of $\phi$, giving an \ttfrac{$\Z$-SFT}{sofic} system conjugate to $(\T^2, T)$. By Theorem~\ref{thm:ExpansivityCharacterization}, the kernel is necessarily SFT, and we can compute its forbidden patterns by Lemma~\ref{lem:DecidableThings}.
\end{proof}

\subsection{An example computation}

We compute an explicit \ttfrac{SFT}{SFT} representation for the action of $\left(\begin{smallmatrix} 1 & 1 \\ 1 & 0 \end{smallmatrix}\right)$ on the torus.

\begin{theorem}
\label{thm:TorusExampleProof}
Let $M = \left(\begin{smallmatrix} 1 & 1 \\ 1 & 0 \end{smallmatrix}\right)$. Then the linear action of $M$ induces an expansive homeomorphism on $\T^2$. This dynamical system is \ttfrac{SFT}{SFT}, namely conjugate to $X/K$, where the minimal forbidden patterns for $X$ are $\{11\}$ and for $K$ they are listed in Fig.~\ref{fig:ptrns}
\end{theorem}

\begin{remark}
We have chosen this example due to its status as a standard example for the construction of SFT covers. The fact that the golden mean shift covers this particular system is well-known, and theory predicts \cite{Fr87} that the kernel is an SFT. In the spirit of effectivization, we compute the list of forbidden patterns for $K$. Our construction follows the general method described in the previous section, and leads to the same covering map as is used in \cite{Ad98,SiVe98}. Generators for the kernel relation of this map are given in \cite{SiVe98}, in the sense that the relations $\orbc{K_L}$ and $\orbc{K_R}$ defined below, which describe the left and right tail exchanges, are found in this reference. We prove that the kernel is the relation $K_L \circ K_R = K_R \circ K_L$ (the tail exchanges can be done in either order, but they \emph{cannot} be done simultaneously). Since the kernel is an SFT, it seems natural to also compute the forbidden patterns of minimal length in the theorem, as proof-of-concept, and we have done this.
\end{remark}

To do this, we apply the general method. Let $T:\T^2\to \T^2$ be the hyperbolic toral automorphism determined by the matrix $A=\begin{pmatrix}1&1\\1&0\end{pmatrix}$. It has eigenvalues $\lambda=\frac{1+\sqrt{5}}{2}$ and $\mu=\frac{1-\sqrt{5}}{2}=-\lambda^{-1}$ ($\abs{\lambda}>1$ and $\abs{\mu}<1$) that are the zeroes of the characteristic polynomial $x^2-x-1$. The eigenvalues have associated eigenlines $l_\lambda,l_\mu$ with directions $v_\lambda=\sqrt{5}^{-1}(\lambda,1)$ and $v_\mu=\sqrt{5}^{-1}(\mu,1)$ respectively. Let $\pi:\R^2\to (\R/\Z)^2$ be the usual projection from the plane to the torus and let $w^\Delta\doteq v_\lambda=\left(\begin{smallmatrix} 1 \\ 0 \end{smallmatrix}\right)+v_\mu$. From our general considerations we know that the point $\pi(nw^\Delta)$ is a homoclinic point of the map $T$ for every $n\in\Z$. Let $X\subset\{0,1\}^\Z$ be the golden mean shift (i.e. the forbidden clopen set is $[11]$) and let $\phi:X\to \T^2$ be defined for points $x\in X$ by
\[\phi(x)=\sum_{i\in\Z}T^{-i}(\pi(x_i w^\Delta))=\pi\left(\sum_{i<0}x_i\mu^{-i}v_\mu+\sum_{i\geq 0}x_i\lambda^{-i}v_\lambda\right).\]

We show that $\phi:(X,\sigma)\to(\T^2,T)$ is a factor map. To see this, let $X_+=\{x\in X\mid x_i=0\mbox{ for }i<0\}$. It is well known that the $\lambda$-shift is equal to the one-sided golden mean shift and therefore
\[\phi(X)\supseteq\phi(X_+)=\left\{\pi\left(\sum_{i\geq 0}x_i\lambda^{-i}v_\lambda\right)\mid x\in X_+\right\}=\pi([0,\lambda]v_\lambda).\]
Since $[0,\lambda]$ is a non-degenerate interval containing $0$, we reach the conclusion by using lemma \ref{lem:FullShiftFactorMap}.

In the following we denote $\pi_\lambda(x)=\sum_{i\geq 0}x_i\lambda^{-i}$, $\pi_\mu(x)=\sum_{i<0}x_i\mu^{-i}$ and $\real(x)=\pi_\lambda(x)v_\lambda+\pi_\mu(x)v_\mu$ for every $x\in X$, from which it follows that $\phi=\pi\circ\real$. Denote by $X_0\subset X$ (resp. $X_1\subset X$) the set of those configurations that contain the digit $0$ (resp. the digit $1$) at the origin. Now we compute that
\begin{flalign*}
&\pi_\lambda(X_0)\subset \left[0,\sum_{i\geq 0}\lambda^{-(2i+1)}\right]=[0,1] \\
&\pi_\mu(X_0)\subset \left[\sum_{i\leq 0}\mu^{-(2i-1)},\sum_{i<0}\mu^{-(2i)}\right]=[-1,-\mu]=[-1,\lambda^{-1}] \\
&\pi_\lambda(X_1)\subset \left[1,\sum_{i\geq 0}\lambda^{-2i}\right]=[1,\lambda] \\ 
&\pi_\mu(X_1)\subset \left[\sum_{i<0}\mu^{-(2i-1)},\sum_{i<0}\mu^{-(2i)}\right]=[-\mu^2,-\mu]=[-\lambda^{-2},\lambda^{-1}],
\end{flalign*}
and it follows that
\begin{flalign*}
&\real(X_0)\subset [0,1]v_\lambda+[-1,\lambda^{-1}]v_\mu\doteq M_0\subset \R^2 \\
&\real(X_1)\subset [1,\lambda]v_\lambda+[-\lambda^{-2},\lambda^{-1}]v_\mu\doteq M_1\subset\R^2 \\
&\real(X)=\real(X_0)\cup\real(X_1) \subset M_0\cup M_1\doteq M\subset \R^2.
\end{flalign*}
It can be verified that $M_0$ and $M_1$ are rectangles whose union $M$ has area equal to $1$. It must hold that $\real(X)=M$, because otherwise $\real(X)\subsetneq M$ would be a closed set with area strictly less than $1$, contradicting $\phi(X)=\T^2$. We also note that $\pi(M)=\phi(X)=\T^2$ implies that $M$ is a fundamental domain for the projection map $\pi$. $M$ also turns out to be equal to the fundamental domain presented in \cite{Ad98}.

The map $\pi$ determines on $M$ the equivalence relation
\[R=\{(v,v') \in M\times M\mid \pi(v)=\pi(v')\}=\{(v,v') \in M\times M\mid v-v'\in \Z^2\},\]
and because $M$ is a fundamental domain, only points on the boundary of $M$ can belong to a non-trivial equivalence class. Since the basis vectors of $\R^2$ have representations $v_1=(1,0)=v_\lambda-v_\mu$ and $v_2=(0,1)=\lambda^{-1}v_\lambda+\lambda v_\mu$, it is simple to verify that $R$ contains at least the relations
\begin{flalign*}
R_1=&\{(v,v+v_1)\mid v\in [0,\lambda^{-1}]v_\lambda+\lambda^{-1}v_\mu\} \\
&\cup\{(v,v-v_1)\mid v\in[1,\lambda]v_\lambda-\lambda^{-2}v_\mu\} \\
R_2=&\{(v,v-v_2)\mid v\in[\lambda^{-1},\lambda]v_\lambda+\lambda^{-1}v_\mu\} \\
&\cup\{(v,v+v_2)\mid v\in[0,1]v_\lambda-v_\mu\} \\
R_3=&\{(v,v+v_1)\mid v\in[0,\lambda^{-1}]v_\mu\} \\
&\cup\{(v,v-v_1)\mid v\in v_\lambda+[-1,-\lambda^{-2}]v_\mu)\} \\
R_4=&\{(v,v+(v_1+v_2))\mid v\in [-1,0]v_\mu\} \\
&\cup\{(v,v-(v_1+v_2))\mid v\in\lambda v_\lambda+[-\lambda^{-2},\lambda^{-1}]v_\mu\},
\end{flalign*}
and it is not much more difficult to see that $R=\Delta_M\cup R_1\cup R_2\cup R_3\cup R_4$. Here $R_1$, $R_2$ glue boundaries parallel to $v_\lambda$ and $R_3$, $R_4$ glue boundaries parallel to $v_\mu$ in $M$.

\begin{lemma}
\label{multirepbi}If $x,y\in X$ are sequences which differ in some nonnegative coordinate, then $\pi_\lambda(x)=\pi_\lambda(y)$ if and only if $x_{[0,\infty)}$ and $y_{[0,\infty)}$ are of the forms
\[ a_1\dots a_m10^\infty \mbox{ and } a_1\dots a_m0(10)^\infty \]
for some $m\in\N$.

If $x,y\in X$ are sequences which differ in some negative coordinate, then $\pi_\mu(x)=\pi_\mu(y)$ if and only if $x_{(-\infty,-1]}$ and $y_{(-\infty,-1]}$ are of the forms
\[ {}^\infty(01)001a_1\dots a_m \mbox{ and } {}^\infty(10)100a_1\dots a_m \]
for some $m\in\N$.
\end{lemma}

\begin{proof}
The first statement follows by applying Lemma \ref{multirep} to the $\lambda$-shift.

To show ``if'' in the second statement, assume that $x_{(-\infty,-1]}$ and $y_{(-\infty,-1]}$ are of the form ${}^\infty(01)001a_1\dots a_m$ and ${}^\infty(10)100a_1\dots a_m$ and without loss of generality that $m=0$. Then
\begin{flalign*}
&\pi_\mu(x)=\mu+\sum_{i>1}\mu^{2i}=\mu+\mu^4\sum_{i\geq 0}(\mu^2)^i=\mu+\frac{\mu^4}{1-\mu^2}=\mu-\mu^3=-\mu^2 \\
&\pi_\mu(y)=\sum_{i>0}\mu^{2i+1}=\mu^3\sum_{i\geq 0}(\mu^2)^i=\frac{\mu^3}{1-\mu^2}=-\mu^2.
\end{flalign*}

To show ``only if'' in the second statement, we argue as in the proof of Lemma \ref{multirep}.
\end{proof}

The set
\[K=\{(x,y)\in X\times X\mid \phi(x)=\phi(y)\}\subset X\times X\]
is a closed $\sigma$-invariant equivalence relation which makes the map $\sigma_K$ continuous and $\phi$ induces a conjugacy between TDSs $(X/K,\sigma_K)$ and $(\T^2,T)$. The inclusion $(x,y)\in K$ holds exactly if $(\real(x),\real(y))\in R$

We will write the relation $K$ more explicitly. We define sets of one-way infinite sequences $B_r=\{10^\infty,(01)^\infty\}$ and $B_\ell=\{{}^\infty(10)100,{}^\infty(01)001\}$. In any case $K$ contains all pairs $(x,x)$, so in the following we will always assume that $(x,y)\in K$ and $x\neq y$. The first possibility is that $\real(x)=\real(y)$. By Lemma \ref{multirepbi} this can occur exactly if $(x,y)\in(w_1.vB_r,w_1.vB_r)\cup(B_\ell u.w_2,B_\ell u.w_2)\cup(B_\ell u.v B_r,B_\ell u.v B_r)$ for any one-way infinite words $w_1,w_2$ and finite words $u,v$. Let us denote the relation containing these pairs and the diagonal by $K_0$. Denote the shift closure of $K_0$ by $\orb{K_0}$ and the topological closure of the shift closure of $K_0$ by $\orbc{K_0}$. We will show that $K=\orbc{K_0}\circ\orbc{K_0}$, or even more precisely that $K=\orbc{K_L}\circ\orbc{K_R}$, where $K_L$ is the relation containing the diagonal and the pairs $(B_\ell u.w_2,B_\ell u.w_2)$ and $K_R$ is the relation containing the diagonal and the pairs $(w_1.vB_r,w_1.vB_r)$. Note that $K_L, K_R \subset K_0$ and $K_0$ is contained in the join of $K_L$ and $K_R$, so $K$ is also precisely the join of $K_L$ and $K_R$.

The next possibility is that $(\real(x),\real(y))\in R_1$. Let us assume that $\pi_\mu(x)=\lambda^{-1}$, the case $\pi_\mu(y)=\lambda^{-1}$ being symmetric. Then $\pi_\mu(y)=-1+\pi_\mu(x)=-1/\lambda^2$ and $\pi_\lambda(y)=1+\pi_\lambda(x)$. Necessarily $x_{(-\infty,-1]}={}^\infty(10)$ and $y_{(-\infty,-1]}\in B_\ell$. If $\pi_\lambda(x)>0$, then from $\pi_\lambda(y)=1+\pi_\lambda(x)$ it follows that $x_0=0$ and $y_0=1$ and $\pi_\lambda(\sigma(x))=\pi_\lambda(\sigma(y))$. This can occur exactly if $(x,y)\in({}^\infty(10).0w,{}^\infty(10)100.1w)\cup({}^\infty(10).0uB_r,{}^\infty(10)100.1uB_r)$ for any one-way infinite $w$ and finite $u$. It can be seen that all such pairs belong to $\orb{K_L}\circ\orb{K_R}$. In the remaining case $\pi_\lambda(x)=0$, $\pi_\lambda(y)=1$ it is also possible that $(x,y)\in({}^\infty(10).0^\infty,B_\ell.(01)^\infty)$. It can be seen that all such pairs belong to $\orb{K_L}\circ\orb{K_R}$.

The next possibility is that $(\real(x),\real(y))\in R_2$, i.e. $\phi(x),\phi(y)\in \pi([0,1]v_\lambda-v_\mu)$. Assume first that $\phi(x)\in \pi((0,1)v_\lambda-v_\mu)$ and consider instead the configurations $x'=\sigma^{-1}(x),y'=\sigma^{-1}(y)$. Note that
\begin{flalign*}
&T^{-1}(\pi((0,1)v_\lambda-v_\mu))=\pi((0,\lambda^{-1})v_\lambda+\lambda v_\mu)=\pi(v_1+(0,\lambda^{-1})v_\lambda+\lambda v_\mu) \\
&=\pi((1,\lambda)v_\lambda+\lambda^{-1}v_\mu),
\end{flalign*}
and because the case $(\real(x'),\real(y')\in\Delta_M$ has been covered above, we may assume that we have $(\real(x'),\real(y'))\in R_2$ and $\real(x')\in(1,\lambda)v_\lambda+\lambda^{-1}v_\mu$. Then $\pi_\mu(x')=\lambda^{-1}$, $\pi_\mu(y')=-\lambda+\pi_\mu(x')=-1$ and $\pi_\lambda(y')=-\lambda^{-1}+\pi_\lambda(x')$. Necessarily $x'_{(-\infty,-1]}={}^\infty(10)$ and $y'_{(-\infty,-1]}={}^\infty(01)$. Since $\pi_\lambda(x')>1$ and $1+\pi_\lambda(y')=\lambda^{-2}+\pi_\lambda(x')<1$, it follows that $x'_0=1$ and $y'_0=0$ and $\pi_\lambda(\sigma(x'))+\lambda^{-1}=\pi_\lambda(\sigma(y'))$. This can occur exactly if $(x',y')\in({}^\infty(10).100w,{}^\infty(01).001w)\cup({}^\infty(10).100uB_r,{}^\infty(01).001uB_r)$ for any one-way infinite $w$ and finite $u$, and all such pairs belong to $\orb{K_L}\circ\orb{K_R}$.

For the remaining cases of $(\real(x),\real(y))\in R_2$, assume first that $\phi(x),\phi(y)\in\pi(-v_\mu)$ and without loss of generality that $\real(x)=\lambda^{-1}v_\lambda+\lambda^{-1}v_\mu$, $\real(y)=-v_\mu$. Necessarily $x\in{}^\infty(10).0B_r$, $y\in{}^\infty(01).0^\infty$ and such pairs $(x,y)$ are included in $\orb{K_L}\circ\orb{K_R}$. The case $\phi(x),\phi(y)\in\pi(v_\lambda-v_\mu)$ remains, so assume that $\real(x)=\lambda v_\lambda+\lambda^{-1}v_\mu$, $\real(y)=v_\lambda-v_\mu$. Necessarily $x={}^\infty(10).(10)^\infty$, $y={}^\infty(01).(01)^\infty$ and such a pair $(x,y)$ is included in $\orbc{K_L}$.

The next possibility is that $(\real(x),\real(y))\in R_4$, i.e.  $\phi(x),\phi(y)\in \pi([-1,0]v_\mu)$. Assume first that $\phi(x)\in \pi([-1,0)v_\mu)$. By considering instead the configurations $x'=\sigma^{-2n}(x)$, $y'=\sigma^{-2n}(y)$ for sufficiently large $n\in\N$ we may assume that $\phi(x')=\phi(y')\in\pi([-1,-\lambda^{-2})v_\mu)$. The case $(\real(x'),\real(y'))\in\Delta_M\cup R_2$ has already been covered, so we may assume that $(\real(x'),\real(y'))\in R_4$, $\real(x')\in [-1,-\lambda^{-2})v_\mu$ and $\real(y')\in\lambda v_\lambda+(-\lambda^{-2},\lambda^{-3})v_\mu$.
Then necessarily $x'_{[0,\infty)}=0^\infty$, $y'_{[0,\infty)}=(10)^\infty$ and from $\pi_\mu(x')\in[-1,-\lambda^{-2})$ it follows that, $x'_{[-2,0]}=010$ and $x'_{(-\infty,-1]}\neq{}^\infty(01)001$. This can occur exactly if $(x',y')\in({}^\infty w01.0^\infty,{}^\infty w00.(10)^\infty)\cup(B_\ell u01.0^\infty,B_\ell u00.(10)^\infty)$ for any one-way infinite $w$ or finite $u$, and all such pairs belong to $\orb{K_L}\circ\orb{K_R}$. The case $\phi(x)=\pi(0,0)$ remains, so let us assume that $\real(x)=(0,0)$ and $\real(y)=\lambda v_\lambda+\lambda^{-1}v_\mu$. Then necessarily $x={}^\infty 0.0^\infty$ and $y={}^\infty(10).(10)^\infty$ and such a pair $(x,y)$ is already included in $\orbc{K_R}$.

The final possibility is that $(\real(x),\real(y))\in R_3$, i.e. $\phi(x),\phi(y)\in \pi([0,\lambda^{-1}]v_\mu)$. Assume first that $\phi(x)\in\pi((0,\lambda^{-1}]v_\mu)$. Let us consider instead the configurations $x'=\sigma(x)$, $y'=\sigma(y)$. Then $\phi(x')=\phi(y')\in\pi([-\lambda^{-2},0)v_\mu)$ and $(\real(x'),\real(y'))\in\Delta_M\cup R_4$, which possibilities have been covered earlier. The case $\phi(x)=\pi(0,0)$ remains, so let us assume that $\real(x)=(0,0)$ and $\real(y)=v_\lambda-v_\mu$. Necessarily $x={}^\infty 0.0^\infty$ and $y={}^\infty(01).(01)^\infty$ and such a pair is already included in $\orbc{K_R}$.

From the case analysis we can draw the conclusion that $K=\orbc{K_L}\circ\orbc{K_R}$.

\begin{proof}[Proof of Theorem~\ref{thm:TorusExample}]
We have shown that the action of $\left(\begin{smallmatrix} 1 & 1 \\ 1 & 0 \end{smallmatrix}\right)$ on $\T^2$ is conjugate to $X/K$ where $K=\orbc{K_0}\circ\orbc{K_0}$ where $K_0$ is the smallest equivalence relation on $X$ containing the pairs
\[(w_1.vB_r,w_1.vB_r)\cup(B_\ell u.w_2,B_\ell u.w_2)\cup(B_\ell u.v B_r,B_\ell u.v B_r). \]
Clearly $\orbc{K_0}$ is an effectively sofic relation on $X$. Given two sofic equivalence relations, their composition is effectively sofic, and given a finite-state automaton defining a relation $K \subset X^2$, we can compute the forbidden patterns in finite time. This is a programming exercise, which we performed using the finite-state automaton library of SAGE \cite{Sage}, yielding the claimed patterns for $K$. The details are given in Appendix~\ref{sec:ToralAutoCode}.
\end{proof}

\section{Questions}

All our questions are about the acting monoids $\N$ and $\Z$, so we omit the acting group.

\begin{question}
Is the conjugacy of \ttfrac{sofic}{sofic} systems semidecidable?
\end{question}

Semi-deciding the conjugacy of sofically presented systems is at least as hard as semi-deciding the homeomorphism of two given simplicial complexes. We do not know if this is semi-decidable \cite{MOSCHomeo}. If these problems are not semidecidable, one can ask how high conjugacy relation of \ttfrac{sofic}{sofic} systems lies in the lightface hierarchy.

Much about the possible dynamics of \ttfrac{sofic}{sofic} systems stays open. We do now know what these systems actually are, or even what their entropies are.

\begin{question}
Is there a natural intrinsic characterization of \ttfrac{sofic}{sofic} dynamical systems?
\end{question}

\begin{question}
Is the topological entropy of a \ttfrac{sofic}{sofic} system computable from the \ccfrac{\text{sofic}}{\text{sofic}} representation (in some natural sense)? What type of number is it, in particular is it always the logarithm of a root of a Perron number?
\end{question}

It would also be interesting to know what can be said about invariant measures. In light of \cite{Ba91}, it would be interesting to know whether a thermodynamical formalism can be developed for them.

Our study of hyperbolic graph systems, automatic spaces and metrization of \ttfrac{sofic}{sofic} systems leaves many open questions.

\begin{question}
Is the metric defined in Theorem~\ref{thm:MetrizingHyperbolics} computable from an oracle for the graph system and two points? In a \ccfrac{\text{sofic}}{\text{sofic}} system presented as a hyperbolic graph system, given two eventually periodic points (in the numerator system), what type of number is their distance, and can it be computed exactly?
\end{question}

Automatic spaces seem to be an interesting class of topological systems worth separate study.

\begin{question}
Is there an intrinsic characterization of automatic spaces?
\end{question}

By \cite{BoFiFi02} (see Theorem~\ref{thm:TrivialDynamics}), \ttfrac{subshift}{subshift} systems can have arbitrarily large dimension, while \ttfrac{sofic}{sofic} systems cannot. We know that the identity map on an infinite space is not \ttfrac{sofic}{subshift}, but we do not know many other restrictions for this class either.

\begin{question}
\label{q:SubshiftSoficFinDim}
Are all \ttfrac{subshift}{sofic} (resp. \ttfrac{sofic}{subshift}) systems finite-dimensional?
\end{question}

We know that not all \ttfrac{subshift}{subshift} systems are \ttfrac{sofic}{subshift} up to conjugacy (see e.g.\ Section~\ref{sec:TrivDyn}), but we do not know any dynamical restrictions on \ttfrac{subshift}{sofic}; dimension would be a natural (even topological) restriction.

\begin{question}
\label{q:SubshiftZeta}
What can be said about the number of periodic points in a \ttfrac{sofic}{sofic}? Are there finitely many periodic points of a given period? Is the Artin-Mazur $\zeta$-function rational?
\end{question}

Besides questions about metric and dynamical properties of the abstract systems, we also lack a basic understanding of how the presentations relate to the systems presented. In particular, we do not know to what extent \ttfrac{sofic}{sofic} systems have canonical \ccfrac{sofic}{sofic} presentations. A natural first question is whether we always take the kernel to have finite equivalence classes, as in this case, it is easy to show there indeed are only finitely many periodic points of each period.

Many questions are undecidable about the relations generated by a sofic relation. A natural one that we have not solved is the following:

\begin{question}
Given a sofic relation $R \subset (A^2)^\Z$, is it decidable whether $R^* = (A^2)^\Z$?
\end{question}

Finally, in the proof of Theorem~\ref{thm:ExpansivityCharacterizationProof} we used two different covers.

\begin{question}
Is there a uniform construction of a subshift cover that works for both $\N$- and $\Z$-actions, and which at least in the group case gives meagerness of the points where the cover is not injective? 
\end{question}

In particular, can the small orbit subshift always be used (by picking a suitable cover so it works in the $\N$-case)? Or can the large orbit subshift be used (by picking a suitable cover so it gives the meagerness statement in the $\Z$-case)? Is the medium orbit subshift useful for anything?

Finally, it would be interesting to study concrete examples in more depth. The denominator we obtained for the toral translation $\left(\begin{smallmatrix} 1 & 1 \\ 1 & 0 \end{smallmatrix}\right)$ has a rather complicated set of forbidden patterns Figure~\ref{fig:ptrns}. Of course, this is partially due simply to the minimal forbidden patterns being a somewhat inefficient way to present an SFT, but nevertheless we found it surprising -- given that the numerator is the golden mean shift, we originally expected this to be something recognizable as well. Are there other presentations where the denominator is simpler?

We also note that while we found the relation calculus allowed by the computer quite fascinating (see Section~\ref{sec:MultiplicationTable}), and the computer is essentially used to produce the minimal forbidden patterns for the kernel, the fact that the transitive closure of $\orbc{K_0}$ is equal to its square does not actually prove that it is the correct kernel. The proof of Theorem~\ref{thm:TorusExample} in fact takes place entirely in the number-theoretic realm. It would be of interest to know whether more of this argument can be shifted on the side of manipulation of sofic presentations.

It would also be interesting to find more natural non-expansive systems which are \ttfrac{sofic}{sofic}. Such systems popped up naturally for $\N$-actions, in the case of $\beta$-shifts, but in the case of $\Z$-actions we are not aware of interesting ``classical'' examples.

\section*{Acknowledgements}

We thank Mike Boyle for helpful reference suggestions and questions. We were delighted to learn that Boyle had previously studied Question~\ref{q:SubshiftZeta} (but not leading to a publication). Example~\ref{ex:Boyle} was constructed in response to his questions.

\bibliographystyle{plain}
\bibliography{bib}{}

\appendix
\newpage

\section{Rays in groups -- an alternative deduction}
\label{sec:RaysAlternative}

We include an alternative proof of Lemma~\ref{lem:Rays} using Stallings theorem and results of Halin instead of those of Woess \cite{Wo89}. A \emph{ray} in a graph is a one-way infinite simple path, i.e. a sequence $x_1, x_2, x_3, \dots$ of vertices of the graph with no repetitions. Two paths are disjoint if the sets of vertices they visit are disjoint.

\begin{theorem}
\label{thm:RaysApp}
Let $G$ be a finitely generated group, and $\mathcal{G}$ a right Cayley graph of $G$ with respect to a finite symmetric generating set $A \subset G$. The following are equivalent.
\begin{itemize}
\item $G$ is virtually cyclic.
\item there does not exist an infinite set of disjoint rays in $\mathcal{G}$.
\end{itemize}
\end{theorem}

\begin{proof}
If $G$ is virtually cyclic, suppose $\mathcal{G}$ is infinite (as the finite case is trivial). Then $G$ has $\Z$ as a subgroup of finite index, so in particular it is quasi-isometric to $\Z$. Thus, if $\mathcal{G}$ had infinitely many disjoint rays, $\Z$ would have infinitely many disjoint rays with respect to some generating set (whose size is given by the quasi-isometry constants), which is not possible by a counting argument ($\Z$ has linear growth rate as a graph).

Suppose then that $G$ is not virtually cyclic. By Stallings theorem, $\mathcal{G}$ has either $0$, $1$, $2$ or $\infty$ ends, and $0$ and $2$ characterize the virtually cyclic cases. If $\mathcal{G}$ has $\infty$ ends, then obviously there are infinitely many disjoint rays (pick a ray in each end, go through them in order and drop common prefixes).

Suppose finally that $\mathcal{G}$ has one end. If $\mathcal{G}$ does not have an infinite set of rays, then by \cite{Ha65} there is a bound $m$ on how many disjoint rays one can find inside $\mathcal{G}$, and then the graph is ``\emph{$m$-fach kettenf\"ormig}'' in the terminology of Halin, meaning there are subgraphs $\mathcal{G}_i, T_i$ such that
\[ V(\mathcal{G}) = V(\mathcal{G}_1) \cup V(\mathcal{G}_2) \cup V(\mathcal{G}_3) \cup \dots, \]
$V(\mathcal{G}_i) \cap V(\mathcal{G}_{i+1}) = V(T_i)$, $|V(T_i)| = m$, all edges of $\mathcal{G}$ are fully inside at least one of the subgraphs $\mathcal{G}_i$, there are $m$ vertex disjoint paths from $T_i$ to $T_{i+1}$ in $\mathcal{G}_{i+1}$, and the graphs $\mathcal{G}_i$ and $\mathcal{G}_j$ don't have common vertices or edges between them, if $|i - j| \geq 2$.

Let us now show that actually $\mathcal{G}$ is at least two-ended, contradicting the assumption that it is one-ended. Without loss of generality (by joining the first few sets $G_i$ and removing the cutsets $T_i$ between them), we may suppose $1_G \in \mathcal{G}_1$ and the distance from $1_G$ to $T_1$ is at least $2$.

Now, observe that cutting out $T_i$ from $\mathcal{G}$ gives us a path $p_i$ from the identity $1_G$ to some vertex $u_i$ such that for some $v_i \in T_i$, there is an edge $(u_i, v_i)$ in $\mathcal{G}$, and such that the path does not contain any vertex from $T_i$. This path from $1_G$ to $u_i$ is of length at least $i$. From $v_i$, we find at least one path $q_i$ which does not return to $T_i \cup \mathcal{G}_{i-1}$ after the first step.

Since $\mathcal{G}$ is a Cayley graph, in particular it is vertex-transitive, so let $g_i$ be the automorphism of $\mathcal{G}$ mapping $g_i v_i = 1_G$ (of course $g_i = v_i^{-1}$), and thus we find for each $i$ a subset $A_i = g_i T_i \subset V(\mathcal{G})$ of $m$ vertices such that there exist disjoint simple paths
$g_i p_i$ and $g_i q_i$
where $q_i$ is a ray, $p_i$ is of length at least $i$, and there is no path from a vertex of $p_i$ to a vertex of $q_i$ which does not visit one of the elements of $A_i$.

Now, let $i \rightarrow \infty$, and restrict to a subsequence $j_i$ so that $p_{j_i} \rightarrow p$, $q_{j_i} \rightarrow q$, $A_i \rightarrow A$ where the first two convergences are in the usual Cantor topology of paths, and the latter is convergence of the characteristic functions in the Cantor topology. Then $|A| \leq m$, $p$ and $q$ are infinite paths starting from the unit sphere around $1_G$ in $\mathcal{G}$, and no two of their vertices can be connected with a path that avoids $A$. Then $A$ witnesses that $\mathcal{G}$ has at least two ends.
\end{proof}

The part of Stallings theorem stating that a group has $0$, $1$, $2$ or $\infty$ ends with any generating set is much easier to prove than the full algebraic statement. We do need that $2$-ended groups are precisely the virtually cyclic ones; for this one can use the simpler argument from \cite{ToYa16} which shows that linear growth is equal to virtually $\Z$, together with the elementary observation that a $2$-ended group has linear growth.

\section{Finding the kernel in Section~\ref{sec:ToralAutomorphisms}}
\label{sec:ToralAutoCode}

Of course, once we have presented the kernel as the composition of two sofic relations, we have solved the problem of characterizing the kernel. However, the composition of two sofic relations is slightly computationally non-trivial. In this appendix, we explain how the forbidden patterns shown in Figure~\ref{fig:ptrns} were computed in the proof of Theorem~\ref{thm:TorusExample}.

Recall that they are the minimal forbidden patterns for the SFT relation $K = \orbc{K_0}\circ\orbc{K_0} \subset X^2$, where $X$ is the golden mean shift and $K_0$ is the smallest subshift equivalence relation on $X$ containing the diagonal pairs and the tuples in
\[ (w_1.vB_r,w_1.vB_r) \cup (B_\ell u.w_2,B_\ell u.w_2) \cup (B_\ell u.v B_r,B_\ell u.v B_r). \]

We used SageMath \cite{Sage} for this computation. The basic algorithms we used for dealing with languages of sofic subshifts and relations can be found in the sage.combinat.finite_state_machine package, though some additional work was needed to implement product languages and quantifier elimination (used for relation composition).

In Section~\ref{sec:Relations}, we compute some relations generated by $K_L, K_R$: at the end of this program, \texttt{L} and $\texttt{R}$ contain automata for the languages of $\orbc{K_L}$ and $\orbc{K_R}$, while the variables \texttt{RR} and \texttt{LR} contain $\orbc{K_R} \circ \orbc{K_R}$ and $\orbc{K_L} \circ \orbc{K_R}$, respectively.

In Section~\ref{sec:MultiplicationTable}, mainly as a sanity check, we prove that the relations 
\[ \orbc{K_L}, \;\; \orbc{K_R}, \;\; \orbc{K_R} \circ \orbc{K_R}, \;\; \orbc{K_L} \circ \orbc{K_R} \]
form a semigroup under composition of relations, by computing the complete multiplication table for these four relations. It is clear from the table that $\orbc{K_L} \circ \orbc{K_R}$ is a zero element of the semigroup, so we must have $K = \orbc{K_L} \circ \orbc{K_R}$ (as indeed the theory predicts).

In Section~\ref{sec:Minimals}, we compute the minimal patterns for $\orbc{K_L} \circ \orbc{K_R} = K$, from which the correct list is obtained once the forbidden patterns $11$ of $X$ are eliminated.

These snippets use the standard library and functionality found in the file \texttt{generalfunctionality.py}. For completeness and easy testing, we include all of the code as ancillary files. The easiest way to test them is to copy and paste them into a SAGE notebook.

\subsection{The relations}
\label{sec:Relations}

\begin{lstlisting}[language=Python]
# golden mean shift: binary alphabet, and
# ["1","1"] represents the forbidden word 11
golden_mean = SFT_automaton(binary, [["1", "1"]])

# construct X^2
gm2 = product_relation(binary, golden_mean, golden_mean)

# compute automata for the B_\ell and B_r relations
# suffix_automaton takes the alphabet, then the
# preperiods and repeating patterns of the
# eventually periodic points identified
Bl = suffix_automaton(binary, ["1", "0"], ["0"],
                 ["1", "0"], ["0", "1"])
Bl = Bl.disjoint_union(flip_automaton(Bl))
Br = prefix_automaton(binary, ["1"], ["0"],
                 [], ["0", "1"])
Br = Br.disjoint_union(flip_automaton(Br))

# construct the diagonal relation of X
diag = diagonal_automaton(binary)
diag = automaton_intersection(diag, gm2)

# * denotes concatenation of languages
L = factor_close(automaton_intersection(Bl * diag, gm2))
print ("L:", L)
R = factor_close(automaton_intersection(diag * Br, gm2))
print ("R:", R)
RR = compose_relations(binary, R, R)
print ("RR:", RR)
LR = compose_relations(binary, L, R)
print ("LR:", LR)
\end{lstlisting}

\vspace{0.3cm}
\begin{center}...with output...\end{center}
\vspace{0.3cm}

\begin{lstlisting}[language=TeX]
L: Automaton with 13 states
R: Automaton with 11 states
RR: Automaton with 17 states
LR: Automaton with 34 states
\end{lstlisting}

\subsection{The multiplication table}
\label{sec:MultiplicationTable}

\begin{lstlisting}[language=Python]
multiplication_table = {}
for a in [("L", L), ("R", R), ("RR", RR), ("LR", LR)]:
    for b in [("L", L), ("R", R), ("RR", RR), ("LR", LR)]:
        crel = compose_relations(binary, a[1], b[1])
        if crel.is_equivalent(L):
            c = "L"
        elif crel.is_equivalent(R):
            c = "R"
        elif crel.is_equivalent(RR):
            c = "RR"
        elif crel.is_equivalent(LR):
            c = "LR"
        multiplication_table[(a[0], b[0])] = c
        print(a[0] + " \u25cb " + b[0] + " = " + c)
        
print ()
print (("*\t|\t" + "\t".join(["L", "R", "RR", "LR"])).expandtabs(4))
print ("----+-----------------")
for a in ["L", "R", "RR", "LR"]:
    s = a + "\t|\t"
    for b in ["L", "R", "RR", "LR"]:
        s += multiplication_table[(a, b)] + "\t"
    print (s[:-1].expandtabs(4))
\end{lstlisting}

\vspace{0.3cm}
\begin{center}...with output...\end{center}
\vspace{0.3cm}

\lstset{literate={ö}{$\circ$}1}
\begin{lstlisting}
L ö L = L
L ö R = LR
L ö RR = LR
L ö LR = LR
R ö L = LR
R ö R = RR
R ö RR = RR
R ö LR = LR
RR ö L = LR
RR ö R = RR
RR ö RR = RR
RR ö LR = LR
LR ö L = LR
LR ö R = LR
LR ö RR = LR
LR ö LR = LR

ö   |   L   R   RR  LR
----+-----------------
L   |   L   LR  LR  LR
R   |   LR  RR  RR  LR
RR  |   LR  RR  RR  LR
LR  |   LR  LR  LR  LR
\end{lstlisting}

\subsection{The minimal patterns}
\label{sec:Minimals}

\begin{lstlisting}[language=Python]
# compute the minimal automaton and
# list the forbidden patterns
minauto = minimal_forbidden_words(LR)

G = minauto.graph()
G.allow_multiple_edges(False)
maxlen = len(minauto.graph().level_sets())
    
print ("maximal length of first offender:", maxlen)

print ("first offenders are:")
ss = ""
for w in sorted(set(map(tuple,minauto.language(maxlen))),
                key = lambda a:(len(a),a)):
    ss += ".".join(w) + ", " 
print (ss[:-2])
\end{lstlisting}

\vspace{0.3cm}
\begin{center}...with output...\end{center}
\vspace{0.3cm}

\begin{lstlisting}[breaklines=true,breakatwhitespace=false,breakindent=0pt]
maximal length of first offender: 6
first offenders are:
01.01, 01.11, 10.10, 10.11, 11.01, 11.10, 11.11, 01.10.00.11, 10.01.00.11, 11.00.01.00, 11.00.10.00, 00.00.00.01.00, 00.00.00.10.00, 00.00.01.00.00, 00.00.01.00.10, 00.00.01.00.11, 00.00.01.10.01, 00.00.10.00.00, 00.00.10.00.01, 00.00.10.00.11, 00.00.10.01.10, 00.01.00.10.00, 00.01.10.00.00, 00.01.10.00.01, 00.01.10.01.10, 00.10.00.01.00, 00.10.01.00.00, 00.10.01.00.10, 00.10.01.10.01, 01.00.00.01.00, 01.00.01.00.00, 01.00.01.00.10, 01.00.01.00.11, 01.00.01.10.01, 01.10.00.00.00, 01.10.00.00.10, 01.10.00.00.11, 01.10.00.10.01, 10.00.00.10.00, 10.00.10.00.00, 10.00.10.00.01, 10.00.10.00.11, 10.00.10.01.10, 10.01.00.00.00, 10.01.00.00.01, 10.01.00.00.11, 10.01.00.01.10, 11.00.00.01.00, 11.00.00.10.00, 11.00.01.10.01, 11.00.10.01.10, 00.00.01.00.01.10, 00.00.10.00.10.01, 00.01.00.00.10.00, 00.01.00.10.01.10, 00.01.10.01.00.00, 00.01.10.01.00.10, 00.10.00.00.01.00, 00.10.00.01.10.01, 00.10.01.10.00.00, 00.10.01.10.00.01, 01.00.01.00.01.10, 01.00.10.00.10.00, 01.10.00.00.01.10, 01.10.00.01.10.00, 01.10.01.10.00.10, 10.00.01.00.01.00, 10.00.10.00.10.01, 10.01.00.00.10.01, 10.01.00.10.01.00, 10.01.10.01.00.01
\end{lstlisting}

\end{document}